\documentclass[12pt,makeidx]{amsart}
\usepackage{amssymb,amsmath,amsthm,mathrsfs}
\usepackage[usenames,dvipsnames,svgnames,table]{xcolor}
\makeindex

\usepackage{enumerate}

\usepackage{hyperref}

\newtheorem{thm}{Theorem}[section]
\newtheorem{theorem}[thm]{Theorem}
\newtheorem{corollary}[thm]{Corollary}
\newtheorem{lemma}[thm]{Lemma}
\newtheorem{proposition}[thm]{Proposition}

\newtheorem{remarkx}             [thm]{Remark}
\newenvironment{remark}
  {\pushQED{\qed}\remarkx}
  {\popQED\endremarkx}
\newtheorem{example}[thm]{Example}

\theoremstyle{remark}

\def\beq{ \begin{equation}}
\def\eeq{  \end{equation} }
\def\bes{\begin{equation*}}
\def\ees{\end{equation*}}

\def\kappap{\kappa^\prime}
\def\kk{\kappa\times \kappa}

\def\N{{\mathbb N}}
\def\eps{\epsilon}
\def\cO{{\mathcal O}}

\def\tA{{\widetilde A}}
\def\tX{{\widetilde X}}
\def\tv{{\widetilde v}}
\def\cU{{\mathcal U}}

\def\SS{\mathbb{S}}
\def\Smatng{\SS_n(\RR^g)}

\def\cC{\mathcal C}

\def\cD{\mathcal D}

\def\cM{\mathcal M}

\def\cH{\mathcal H}
\def\cK{\mathcal K}

\def\cL{\mathcal L}

\def\cQ{\mathcal Q}

\def\cS{\mathcal S}
\def\cI{\mathcal I}

\def\cP{\mathcal P}
\def\ccP{\mathcal P}
\def\cR{\mathcal R}

\def\cT{\mathcal T}
\def\cW{\mathcal W}
\def\cY{\mathcal Y}

\def\cN{\mathcal N}

\def\hcO{\widehat{\cO}}
\def\AD0{ \hA_D(\hX), \hX}

\def\tg{{\widetilde g}}

\def\smatn{\mathbb S_n} 

\def\smatng{\mathbb S_n(\mathbb R^g)}
\def\smatnk{\mathbb S_{n\kappa}} 
\def\smatntg{\mathbb S_n(\mathbb R^{\tg})}

\def\hA{{ \widehat A}}
\def\hX{{ \widehat X}}
\def\hv{{ \widehat v}}

\def\fP{\mathfrak{P}}

\def\gm{\mathfrak{m}}
\def\ga{\mathfrak{a}}
\def\gt{\mathfrak{t}}

\def\dddd{\kappa \times \kappa}

\def\cV{\mathcal{V}}

\def\allmat{\mathbb S(\mathbb R^{\fg})}
\def\allmatn{\mathbb S_n(\mathbb R^{\fg})}

\def\allmatvk{\allmatn \times \RR^{n\kappa}}
\def\allmatv{\allmatn\times\RR^n}
\def\star{T} 
\def\st{T}  

\def\bep{ \begin{proof} }
\def\eep{ \end{proof}}

\def\ms{\medskip}
\def\bs{\bigskip}
\def\fg{\mathfrak{g}}
\def\msC{K}

\def\hbDp{\partial \widehat{\fP}_p}

\def\wtilde{\widetilde}
\def\ben{\begin{enumerate}}
\def\een{\end{enumerate}}
\def\nk{n\kappa}

\def\xxx{ \partial\widehat{\fP}_p  \cap    \hcO  }
\def\xxxx{\partial \fP_p \cap \cO}
\def\td{\tilde{d}}

\def\cPk{{\cP^{\kappa\times\kappa}}}

\def\chp{\cH_+^{\lambda,\delta}}
\def\ckp{\cK^{\lambda,\delta}}

\def\cCp{\cC_p}

\newcommand{\df}[1]{{\bf{#1}}{\index{#1}}}

\newcommand{\RR}{{\mathbb R}}

\newcommand{\ZZ}{{\mathfrak{Z}}}
\newcommand{\QQ}{\mathfrak{Q}}

\def\tsU{\widetilde{\mathscr U}}

\def\range{\operatorname{range}}

\def\IktV{\wtilde{V}^\kappa_{\degg}}
\def\fns{\cS}
\def\fnsS{\cS}
\def\msUalt{\widetilde{\mathfrak{U}}}
\def\msU{\mathfrak{U}}

\def\MM{\mathfrak{M}}
\def\row{\operatorname{row}}
\def\crep{\bold{c}}
\def\ff{\mathfrak{f}}
\def\fL{\gm\ga\gt}
\def\hZZ{\widehat{\ZZ}}

\def\ft{\mathfrak{t}}
\def\kxk{{\kappa \times \kappa}}

\def\degg{\operatorname{deg}}
\def\Yk{Y^\kappa_{\degg}}
\def\Kk{K^\kappa_{\degg}}
\def\ff{{\tt{f}}}
\def\Vk{V^{\kappa}_{\degg}}
\def\QQ{Q}

\newcommand{\wt}[1]{\widetilde{#1}}

\numberwithin{equation}{section}

 \title[nc polys with convex slices]{Non-commutative polynomials with convex level slices}

\author[H. Dym]{Harry Dym}
\address{Harry Dym, Department of Mathemtics \\
  Weizmann Institute of Science\\
Rehovot, 7610001 Israel}
\email{harry.dym@weizmann.ac.il}

\author[J.W. Helton]{J. William Helton${}^1$}
\address{J. William Helton, Department of Mathematics\\
  University of California \\
  San Diego}
\email{helton@math.ucsd.edu}
\thanks{${}^1$Research supported by the NSF grant
DMS 1201498, and the Ford Motor Co.}

\author[S. McCullough]{Scott McCullough${}^2$}
\address{Scott McCullough, Department of Mathematics\\
  University of Florida\\ Gainesville 
   }
   \email{sam@math.ufl.edu}
\thanks{${}^2$Research supported by the NSF grant DMS-1361501}


\begin{document}

\maketitle




\begin{abstract}
  The  structure of
   symmetric polynomials $p(a,x)$ in freely noncommuting symmetric variables $a=(a_1,\dots,a_\tg)$ and $x=(x_1,\dots,x_g)$
    such that the set \index{$\fP_p^A$} of matrix tuples
  $$ \fP_p^A := \{X : \ p(A,X) \mbox{ is positive  definite}\}
  $$
   is convex for large sets of $A$ is studied. Under fairly general hypotheses it is shown that such polynomials have degree at most two in $x$ and must have special structure. A matrix-valued version of the main result of \cite{BM}
   on quasiconvex noncommutative  polynomials is obtained as a special case of the main result.
\end{abstract}

\section{Introduction}

The main result of this article, Theorem \ref{thm:main}, is a characterization of symmetric polynomials $p(a,x)=p(a_1,\dots,a_{\tg},x_1,\dots,x_{g})$ in freely noncommuting variables such that the
evaluations
$$
p(A,X)=p(A_1, \dots, A_\tg, X_1,\dots,X_g)
$$
on symmetric matrices (of the same size)
belong to the set
\[
  \fP_p^A = \{X=(X_1,\dots,X_g): p(A,X) \mbox{ is $\succ 0$} \}
\]
is nonempty and convex for a large set of tuples of matrices $A=(A_1,\dots,A_\tg)$.
  Under fairly general hypotheses (most of which rule out obvious pathologies) we show that such  polynomials have degree at most two in $x$ and have special structure.

Here, and below, the notation  $M\succ 0$ (resp., $M\succeq 0$) for $M$, a (square) symmetric matrix with real entries means that $u^TMu>0$ (resp., $u^TMu\ge 0$) for all nonzero real vectors $u$.  A symmetric $\kappa\times \kappa$ matrix-valued free polynomial $q$ is \df{quasiconvex} if for each $n$ and  symmetric $n\kappa\times n\kappa$ matrix $M$ the set of $g$-tuples $X$ of $n\times n$ symmetric matrices such that $M-q(X)\succ 0$ is convex.

Theorem \ref{thm:main} is illustrated by two corollaries having much simpler statements.
\smallskip

  \noindent
(1)  Theorem \ref{thm:quasiC} characterizes quasiconvex matrix free polynomials $q$ satisfying some mild additional hypotheses.
\smallskip

  \noindent
(2) Theorem \ref{thm:alinScalar} concerns
 a scalar polynomial $p$ of the form $p(a,x)=ar(x)+r(x)a -2 f(x)$ for a (scalar)  free polynomial $r$ and a (scalar) symmetric free polynomial
 $f$ with $r(0)\ne 0$ and $f(0)=0$.
 If for a ``large" set of matrices $A$ the set $ \fP_p^A $ is convex,
 then $ p$ has a simple form that is of degree 2 in $x$.
 Indeed, in this setting the case $r=1$ reduces to the case
 where $f$ is a scalar quasiconvex polynomial.

 In the remainder of this introduction  the background for and statements of
  Theorem \ref{thm:quasiC}  and  \ref{thm:alinScalar}  are introduced.  It closes with a brief discussion of the motivations coming from systems engineering
  and operator theory
   (Section \ref{sec:motivate}) and a guide to the remainder of the article (Section \ref{sec:readers guide}).

The statement of Theorems \ref{thm:main}
and the prerequisite definitions are the subject of Section \ref{sec:general}.

\subsection{Free Polynomials}
  \label{subsec:ncpolys}
   Fix positive integers $g$ and $\tg$ and let
   $\fg = \tg+g$. Let
   $\ccP$ \index{$\ccP$} denote the real free algebra
   of polynomials in the
   freely noncommuting symmetric variables
   $x=(x_1,\dots,x_g)$ and $a=(a_1,\dots,a_{\tg})$. The
   elements of $\ccP$ are called
   \df{free polynomials}, \df{nc polynomials}
   or often just \df{polynomials} in
   $(a_1,\dots,a_{\tg}, x_1,\dots,x_g).$ 
       Thus, a free polynomial    $p$ is a finite linear combination,
 \begin{equation*}
    p=\sum p_w  w,
 \end{equation*}
   of words $w$ in
   $(a_1,\dots,a_{\tg}, x_1,\dots,x_g)$ with coefficients
 $p_w\in\mathbb{R}$.

   There is a natural involution ${}^{\st}$ on
   $\ccP$  given by
 \begin{equation*}
   p^{\st}=\sum p_w w^{\st},
 \end{equation*}
  where, for a word
   $w=z_{j_1}z_{j_2}\cdots z_{j_n}$
   we have
    $w^{\st} = z_{j_n}
   \cdots z_{j_2}z_{j_1},$ (since $x_j=x_j^T$ and $a_i=a_i^T$).

   A polynomial $p$ is symmetric if
   it is invariant with respect to the involution, i.e., if $p=p^T$.
     In particular, $x_j^{\st}=x_j$ and for this reason
   the variables $x_j$ are sometimes referred to as \df{symmetric free  variables}. \index{symmetric variables}
   The condition $a_j=a_j^T$ was imposed above to ease the exposition. It is less essential and will be violated on occasion.

\subsection{Evaluations}
 \label{sec:free-eval}
  Given a positive integer $n$,
  let $\mathbb{S}_n$ denote the set of real symmetric $n\times n$ matrices
  and let $\smatng$ denote the set of $g$-tuples \index{$\smatng$}
  $X=(X_1,\ldots,X_g)$ of real symmetric $n\times n$ matrices.
 The set of all pairs $(A,X)$ with $A \in \mathbb{S}_n(\RR^{\tg})$ and
  $X\in  \smatng$,  will be denoted
  $\allmatn$

  Let $M_n$ denote the $n\times n$ matrices with real entries.
  A pair $(A,X)  \in\allmatn$ determines a mapping
  $e_{(A,X)}:\mathcal P \to M_n$  by evaluation
  (actually this mapping is a representation of the algebra $\cP$).
    Indeed, by
  linearity, $e_{(A,X)}$ is determined by its action on words where
  $e_{(A,X)}(\emptyset)=I_n$ and, for
   a nonempty word $w$ in $(a,x)$
   and $(A,X)\in\allmatn$, the evaluation
   $e_{(A,X)}(w) $ is the same word in $(A,X)$.
  It is natural to write $p(A,X)$ instead of the more
   formal $e_{(A,X)}(p)$.

  Note that $p(A,X)$ respects the involution in the sense  that $p^{\st}(A,X)=p(A,X)^{\star}$.  In particular, if $p$ is symmetric,  then so is $p(A,X)$.
 It is well known that, taken together,  evaluations determine faithful representations of free polynomials.

\subsection{Matrix-Valued Polynomials}
A free $\kappa \times\kappa$ matrix-valued  polynomial $p(a, x)$
is a linear combination of words in the freely noncommuting
   variables  $a=(a_1,\dots,a_\tg)$ and  $x=(x_1,\dots,x_g)$ with coefficients
  from $M_\kappa$.
Let $\ccP^{\kappa\times \kappap}$ denote \index{$\ccP^{\kappa\times \kappap}$}
the $\kappa\times \kappap$ matrices with entries
  from $\ccP$.

 For   $p\in \ccP^{\kappa\times\kappap}$, evaluation   at $(A,X) \in \allmatn$ is defined  entrywise,  leading to a $\kappa\times\kappap$ block matrix $p(A,X),$ with entries from $M_n$.
  Up to unitary equivalence,
  evaluation at $X$ is conveniently described in terms of the  tensor
  product of matrices  by writing $p$ as  a finite linear combination\footnote{In keeping with our usual usage, we write $\sum C_w\, w$ and not $\sum C_w\otimes w$.}
 \begin{equation}
  \label{pww}
   p
   =\sum_{w}C_w\,  w,
\end{equation}
   with  coefficients
   $C_w\in\RR^{\kappa\times\kappa^\prime}$  and observing that
 \[
   p(A, X)=\sum C_w \otimes w(A,X),
 \]
  where $w(A,X)=e_{A,X}(w)$.

 The involution ${}^{\st}$ naturally extends to
  $\ccP^{\dddd}$ by
 \[
    p^{\st} = \sum_{w} C_w^{\star} \, w^{\st},
 \]
  for $p$ given by equation \eqref{pww}.  The matrix polynomial $p$ is symmetric if $p^T=p$. It turns
out that $p$ is symmetric if and only if $p(X)^T=p^T(X)$ for every $n$ and every tuple $X\in\smatng$.
\index{symmetric matrix polynomial}

   A simple method of constructing new matrix-valued
  polynomials from old ones is by  direct sums.
   \index{direct sum}
  For instance, if
  $p_j\in\ccP^{\kappa_j\times \kappa_j}$ for $j=1,2$, then
 \[
   p_1\oplus p_2=\begin{bmatrix} p_1 & 0\\0 & p_2 \end{bmatrix}
    \in \ccP^{(\kappa_1+\kappa_2)\times (\kappa_1+\kappa_2)}.
 \]

\subsection{Highest degree terms}\label{subsec:hdt}
A matrix-valued nc polynomial $p(a,x)=\sum_{w\in\cW}C_w w(a,x)$ of degree $d$ in $x$ can be expressed as a sum
$$
p=\sum_{j=0}^d p^j(a,x)
$$
of nc matrix-valued polynomials $p^j(a,x)$ that are homogeneous of degree $j$ in $x$. Moreover,
(see subsection \ref{sec:nckron}), every homogeneous nc polynomial $p^j(a,x)$ of degree $j$ in $x$ can be expressed in the form
\begin{equation}
\label{eq:defsffj}
p^j(a,x)=
\varphi_p^j(a) \ff_j(a,x),
\end{equation}
where
$$
\varphi_p^j(a)=\begin{bmatrix}\varphi_{j 1}(a)&\cdots&\varphi_{j,s_j}(a)\end{bmatrix}
$$
is a block row matrix that depends only upon the coefficients $C_w$ of the words $w$ in the polynomial $p^j$ and the variables $a_1,\ldots,a_{\widetilde{g}}$ and $\ff_j(a,x)$ is an $s_j\times 1$ vector polynomial in $a$ and $x$ that is homogeneous of degree $j$ in $x$ of the form,
\[
  \ff_j(a,x) =  \textup{col}(x_1 \ff_{1,1,j}, \, x_1\ff_{1,2,j},\dots, \, x_1\ff_{1,k_1,j},\, x_2 \ff_{2,1,j},
    \dots, \, x_g \ff_{g,k_g,j})
\]
and each $\ff_{i,k,j}=\ff_{i,k,j}(a,x)$.
%
%
We shall say that the {\bf highest degree terms of $p$ majorize at $A$} if
\begin{equation}
\label{eq:apr25a17}
\textup{range}\,\varphi_p^j(A)\subseteq\textup{range}\,\varphi_p^d(A)\quad\textrm{for $j=1,\ldots,d-1$.}
\end{equation}

   The condition \eqref{eq:apr25a17} 
   is essential
    in Theorem \ref{thm:dec29a13}, which is
  a key ingredient in the proof of our main result.

\begin{remark}\rm
 \label{rem:special majorize}
The condition { \eqref{eq:apr25a17}} 
is automatically met are the following special cases.
\begin{enumerate}
\item  If $p$ (matrix-valued) is homogeneous in $x$,
  then, as is immediate from the definitions,  the highest degree terms of $p$ majorize at each $A$.
 \item \label{it:sm2} If $\kappa=1$ (i.e., the polynomial $p$ is scalar),  $p$ is of degree $d$ in $x$  and there is at least one term in $p$ of degree $d$ in $x$ that begins with an $x_j$ on the left, then at least one of the $\varphi_{di}(A)$
is equal to $cI_n$ with $c\ne 0$ when $A\in\SS_n$.
\end{enumerate}
\end{remark}

\subsection{Free quasiconvex polynomials have degree two}
\label{sec:matbm}

One consequence of  our main result (Theorem \ref{thm:main})  is the following matrix version of the result of \cite{BM}.

\begin{theorem}
\label{thm:quasiC}
 Suppose that $f=f(x)=f(x_1,\ldots,x_g)$ is a $\kappa\times\kappa$  nc matrix polynomial such that:  \begin{enumerate}
\item[\rm(1)] The highest degree terms of $f$ majorize.
\item[\rm(2)]  $f(0)=0$.
\item[\rm(3)] There exists a positive number $\gamma$ such that
 for each positive integer $n$ and each   $M\in \SS_{\kappa n}(\RR)$
 that meets the constraints $ M\prec \gamma I_{\kappa n}$
  the set
 $\{X\in\SS_n(\RR^g):\, f(X)\prec M\}$ is a proper convex subset of
  $\SS_n(\mathbb R^g)$.
  \end{enumerate}
 Then there  exists a  symmetric polynomial
  $\ell\in \ccP^{\dddd}$ that is affine  linear in $x$,
        and   a finite number $n$ of
    matrices of  free polynomials
        $s_j\in\ccP^{1\times \kappa}$  each linear (and hence homogeneous of degree one) in $x$
    such that
\begin{equation*}
   f(x)= \ell( x) + \sum_{j=1}^n s_j( x)^T s_j( x).
   \end{equation*}

\end{theorem}

 When $\kappa =1$, Theorem \ref{thm:quasiC} is a variant of the main theorem of \cite{BM}. It turns out  that the justification of  Theorem \ref{thm:quasiC}  for the case $\kappa>1$ is surprisingly involved.

Theorem \ref{thm:main} also applies to the class of free symmetric scalar polynomials of the form $p(a,x)=r(x)a+ar(x)-f(x)$ where $r$ and $f$ are polynomials in one free variable ($g=\tg=1$) subject to some additional constraints.  Given $\eps>0$, let $B_\eps$ denote the sequence $(B_\eps(n))_{n=1}^\infty,$ where  $B_\eps(n) =\{X\in \SS_n : \|X\|<\eps\}$, the open $\epsilon$ ball (centered at $0$) in $\SS_n$.

\begin{theorem}
\label{thm:alinScalar}
If $p(a,x)=r(x)a+ar(x)-2f(x)$, where
$r$ and $f$ are scalar polynomials in a single variable $x$ of degree $k$ and $d$ respectively such that  \begin{enumerate}[\rm(i)]
 \item  $r(0)=1$ and $f(0)=0$.
 \item  $fr^{-1}$ is not constant.
 \item  There exists an $\eps>0$ such that for each $n$ and each $A\in \smatng$ the set
$$
\fP_p^A \cap B_\eps =  \{ X\in\SS : \, p(A,X) \succ 0\ and\ \Vert X\Vert<\eps\}
$$
is convex.
 \end{enumerate}
 Then there  exists a  symmetric polynomial  $\ell\in \ccP$ affine  linear in $x$  and a finite number $n$  free polynomials  $s_j\in\ccP$  each linear (and hence homogeneous) in $x$  and independent of $a$ such that
 \begin{equation*}
   p(a,x)= \ell(a,x) - \sum_{j=1}^n s_j(x)^T  s_j(x).
 \end{equation*}
  In particular,  $d$ is at most two.
\end{theorem}

\subsection{Related results and remarks}
 \label{sec:motivate}

  Theorem \ref{thm:main} falls within the rapidly developing  fields of
  free analysis and free semialgebraic geometry.


Free analytic
  functions, arise naturally in a number of contexts,
  including free probability \cite{VoicI,Voic} and multivariate systems theory.
  The book \cite{DKV-VV} provides an unparalleled introduction to the theory of free analytic functions. Important recent results can be found in \cite{AMglobal} for instance while \cite{AMoka} develops a beautiful analog of the Oka extension theorem and related results of a free several complex variables flavor. The papers \cite{AKV13,AKVimplicit,AMimplicit} establish free implicit and inverse function theorems. It turns out there are serious topological subtleties. The free implicit function theorem in \cite{Pascoe-implicit} has implications for the free version of the Jacobian conjecture.  The theory of Pick and matrix monotone function is extended to the free setting in \cite{PTD} and the article \cite{CPTD} exposes fascinating connections to representation theory.  A good introduction to many of these topics can be found the survey article \cite{AMaspects}.  Two related lines of development of the theory of free functions influenced by single and multivariable operator theory are the series of papers by Muhly and Solel  and that of Popescu for which \cite{MS,MS13} and \cite{Pop,Pop11,Pop13,Pop15} are but just a few of the  references. For a survey article see \cite{MS11}.

 Free rational functions and skew fields have a long history going back to at least \cite{Shu}. A more recent development is their appearance in multivariate systems theory    \cite{BGM,BGtH}. In this context,  for a given rational function, the existence of linear fractional representations whose formal singularities and actual singularities agree is an issue \cite{KVV09,HM14}.

  Free semialgebraic geometry is, by analogy to the
  commutative case, the study of free polynomial inequalities.  An algebraic formula (certificate) equivalent to the validity of a polynomial inequality is known as a positivstellensatz.
  A small sample of  free Positivstellensatze include  \cite{C} \cite{KS1} \cite{KS2}. See also the references therein.


  Convex inequalities (inequalities with convex solution sets) represent an important subclass of these areas and are the subject of {\it convex real algebraic geometry}, an emerging subfield of real algebraic geometry  \cite{BPT}.  The results in this paper lie within the free analog of convex algebraic geometry. Its motivation comes from systems theory  and control engineering
 as well as  from the theory
  of operator spaces and systems and matrix convexity. See, for examples,
  \cite{EW,FP,KPTT}.


To give a taste of the engineering connection consider
 the free polynomial
  $$   p(a,x)=-x bb^T x + a^T x + x a +c,  \qquad c=c^T , \, x=x^T $$
 and the corresponding Riccati inequality
  $$   0 \preceq -X BB^T X + A^T X + X A +C,  \qquad C=C^T , \, X=X^T,$$
  which is  ubiquitous in systems engineering. Here
  the matrices
  $A,B,C$ can have any compatible dimension.

  Note the Riccati inequality is
   {\it dimension free} in the sense that its form does not change
   with the dimension of the matrices.
   This independence of dimension property is typical of systems problems described purely by a signal flow diagram and whose signals are in $L^2$ and whose
   performance is of mean square type, cf. \cite{OHMP, HKMfrg}. Since convexity is a highly desired property in system engineering it is unfortunate that
   for such dimension free systems problems, convexity is obtainable only with  the, a priori more restrictive,
   {\it linear matrix inequalities} as the results here
 strongly suggest.

 This article is a natural successor to \cite{DHM07a,DHM07b,DGHM09,DHMill,BM}.  Here the question is what can be said of the polynomial if there are sufficiently many convex level slices.  The article \cite{HM} characterizes the convex level sets of  polynomials $p(x)$ in the free variables $x=(x_1,\dots,x_g).$ The article \cite{HHLM} characterizes polynomials $p(a,x)$ satisfying various convexity conditions in $a$ and $x$ separately.

\subsection{Readers guide}
 \label{sec:readers guide}
  The remainder of this article is organized as follows.  The central result, Theorem \ref{thm:main},  appears near the end of Section \ref{sec:general} after the needed preliminaries are developed. Theorems \ref{thm:quasiC} and \ref{thm:alinScalar}  are shown to be a consequence of Theorem \ref{thm:main} in  Sections \ref{sec:matrixBM} and \ref{sec:ex} respectively.  The paper then turns, in the next nine sections,  to proving Theorem \ref{thm:main}.  This analysis hinges on the principle that the Hessian restricted to the tangent space, a type of free second fundamental form, is negative definite. The needed machinery is developed in Sections \ref{sec:boundary} and \ref{sec:DSLI}.  Section \ref{sec:identities} reviews a number of Kronecker product identities for the convenience of the reader. The Hessian of a symmetric free polynomial has a {\it middle matrix-border vector} representation described in Sections  \ref{sec:mmhess}, \ref{sec:polycon} and
\ref{sec:CHSY}.  Positivity of the middle matrix of $p$ forces $p$ to have degree two, a fact proved in Section \ref{sec:posM};    Section \ref{sec:moreposM} provides sufficient conditions for positivity of the middle matrix.  The proof of Theorem \ref{thm:main} culminates in Section \ref{sec:proofmain}.  In brief, sufficient negativity of the free second fundamental form forces positivity of the middle matrix.  Appendix \ref{appendix:r=0} contains the proof of
Remark \ref{prop:r=0}.  Several additional appendices provide examples illustrating the various objects and conditions appearing through the paper.
The authors thank S. Balasubramanian for reading and commenting on an early draft of this article.

\section{The general theorem}
\label{sec:general}

The main result of this article, Theorem \ref{thm:main}, is formulated in this section.
Its statement requires a number of definitions that we now introduce
and illustrate with examples.

\subsection{Directional derivatives and the Hessian}
 \label{sec:DDH}
 Given $p\in \cP^{\dddd}$
 we shall compute directional derivatives of $p$ in the $x$ variable
 assuming the $a$ variable is fixed.
 While these are partial directional derivatives with respect to $x$,
 we shall abuse terminology and leave out the word partial. If  $h=(h_1,\dots,h_g)$ is another
 set of freely noncommuting variables and  $t\in\RR$,
\begin{equation}
 \label{eq:def-dervs}
   p(a,x+th)= \sum_{j=0}^d p_j(a,x)[0,h] t^j,
\end{equation}
 where $p_j(a,x)[0,h]$ are polynomials in the freely noncommuting
 variables
 $$(a,x,h)=(a_1,\dots, a_\tg, x_1,\dots,x_g,h_1,\dots,h_g).$$
 The notation
 indicates the different role that these variables play.
  Observe that $p_j(a,x)[0,h]$ is homogeneous of degree $j$ in $h$.

  The polynomial $p_1(a,x)[0,h]$ is the \df{directional derivative}
  or simply the \df{derivative}  of  $p$
  (in the direction $h$) and is denoted \df{$p_x(a, x)[0,h]$};
  the polynomial
\[
  p_{xx}(a,x)[0,h]=2p_2(a,x)[0,h]
\]
  is the \df{Hessian} of $p$. \index{$p_{xx}$}
 When there is no ambiguity, we  shall write $p_x(a,x)[h]$ and $p_{xx}(a,x)[h]$ to simplify the typography.

 The (partial) directional derivative  $p_a(a,x)[e,0]$
and the full directional  derivative
$p^\prime(a,x)[e,h]$ are defined analogously in the  spirit of \eqref{eq:def-dervs}
using $p(a+te,x)$ and  $p(a+ t e, x +t h)$ respectively.

\subsection{Full rank conditions}
Let $p$ be a symmetric $\kk$-valued free polynomial.
A pair $(A, X) \in\allmatn$ is called a \df{full rank point} for $p$ if, for each positive integer $n$,  the map
\[
 (E,H)\in\allmatn\longrightarrow p^\prime(A,X)[E,H]\in  \SS_{n\kappa} 
\]
 is onto $\mathbb S_{n\kappa}$.    Note that the full rank condition places a constraint on the  relative sizes of $\fg$ and $\kappa$. In particular, the inequality
  $\fg(n^2+n) \ge (n\kappa)^2+n\kappa$ is a necessary condition for the existence of full rank points.
 \index{$\cL\cC_w$} \index{chip set} \index{ left chip set} \index{right chip set} \index{$\cR\cC_w$}

\subsection{Chip sets}
\label{sec:chip}
  For $1\le j\le g$, the  \df{right chip set} $\cR\cC_w^j$
  of a word $w$ in the variables $(a,x)$
  is the set of words $v$ such that there exists a word $u$ (empty or not) such that
\[
  w= ux_jv.
\]
 The  left chip set is defined analogously. Thus, for example, if $w=ax_2x_1x_2a$, then
 $$
 \cR\cC_w^1=\{x_2a\} \quad\textrm{and}\quad \cR\cC_w^2=\{a, x_1x_2a\},
 $$
 whereas,
 $$
 \cL\cC_w^1=\{ax_2\} \quad\textrm{and}\quad \cL\cC_w^2=\{a, ax_2x_1\}.
 $$
 Notice that if $w=w^T$  and the variables are symmetric, as in the present case, then  the words in $\cL\cC_w^j$ are the transposes of the words in $\cR\cC_w^j$.

For a given $j$, the right chip set
  $\cR\cC_p^j$ of a polynomial
  $p$ is the union of the right chip sets $\cR\cC_w^j$ of the words $w$ appearing in $p$ (with nonzero coefficients).
  In particular,
  for a given polynomial $p$, the partial of $p$ with respect to $x$ has the form
\begin{equation}
\label{eq:aug13y14}
  p_x =  \sum_j \sum_{u} \sum_v C_{u,v,j} \, uh_jv
\end{equation}
 where $C_{u,v,j}$ is a matrix (of appropriate size) and $u$ and $v$ are from the left
  and and right chip sets   of $p$ respectively. Similarly the  Hessian of $p$ takes the form,
\begin{equation}
\label{eq:aug13z14}
  p_{xx} = \sum C_{u,v,j,\ell} \, u h_j w h_\ell v,
\end{equation}
 where $u$ and $v$ are from the left and right chip sets of $p$ respectively.
 Let
 \begin{align*}
 \cC_p^j&\quad\textrm{denote the  chip space of polynomials in the words of $\cR\cC_p^j$}\\
 \cR\cC_p&\quad\textrm{denote the union of the sets $\cR\cC_p^j$ and}\\
  \cC_p&\quad\textrm{denote
  the \df{ chip space} of polynomials in the words of $\cR\cC_p$}.
\end{align*}
\index{$\cC_p$} \index{$\cR\cC_p$} \index{$\cL\cC_p$}

\begin{example}\rm
\label{ex:sep17a13}
Let $p(a,x)=x_2^2ax_1+x_1ax_2^2+a^2$. Then
$$
\cR{\cC}_p^1=\{1,ax_2^2\},\quad \cR{\cC}_p^2=\{1,x_2,ax_1,x_2ax_1\}
$$
and
$$
\cR\cC_p=\{1,x_2,ax_1,ax_2^2,x_2ax_1\}.
$$
Correspondingly,
$$
p_x(a,x)[0,h]=h_2x_2ax_1+x_2h_2ax_1+x_2^2ah_1+h_1ax_2^2+x_1ah_2x_2+x_1ax_2h_2
$$
is of the form \eqref{eq:aug13y14}, whereas
\begin{align*}
p_{xx}&(a,x)[h]\\ &=2\{h_2^2ax_1+h_2x_2ah_1+x_2h_2ah_1 +h_1ah_2x_2+h_1ax_2h_2+x_1ah_2^2\}\\
&=2\begin{bmatrix}h_1&h_2&x_2h_2&x_1ah_2\end{bmatrix}\,\begin{bmatrix}0&ax_2&a&0\\
x_2a&0&0&1\\ a&0&0&0\\ 0&1&0&0\end{bmatrix}\,
\begin{bmatrix}h_1\\ h_2\\ h_2x_2\\ h_2ax_1\end{bmatrix}
\end{align*}
is of the form  \eqref{eq:aug13z14}.

The polynomial $p(a,x)-p(a,0)$ can also be expressed as a linear combination of terms in
the right chip set $\cC_p$ with polynomial coefficients. In fact,
$p(a,x)-p(a,0)$
can be recovered from the formula for $p_x(a,x)[0,h]$ by deleting those terms in which an $h_j$ is preceded by an $x_i$ and then replacing $h_j$ by $x_j$ in the remaining terms. Thus, in the
case at hand,
$$
p(a,x)-p(a,0)=p(a,x)-a^2=x_1(ax_2^2)+x_2(x_2ax_1).
$$
\qed
\end{example}

\begin{remark}\rm
If $p(a,x) = a - p(x)$, then
the $a$ term will not appear in $p_x(a,x)[0,h]$ or in $p_{xx}(a,x)[0,h].$ Hence $\cC_p^i$
will only consist of polynomials in the $x$ variables.
\end{remark}

\subsection{Free sets, positivity sets  and set domination}

 Let $p$ be a free $\kk$-valued symmetric polynomial. For positive integers $n$, let \index{$\fP_p$}
\[
 \fP_p(n) =\{(A,X)\in\allmatn: p(A,X)\succ 0\}.
\]
 The \df{positivity set} of $p$ is a the sequence $\fP_p= (\fP_p(n))_{n=1}^\infty$.
 The set $\fP_p$ naturally earns the moniker of \df{free open (basic) semialgebraic set}.

 Given a positive integer $n$ and $\tg$-tuple $A\in\SS_n(\RR^{\tg}),$
let
\begin{equation}
 \label{def:fPpA}
 \fP_p^A =\{X\in\smatng: p(A,X)\succ 0\}
\end{equation}
We call $\fP^A_p$ the  \df{$A$-cross section of $\fP_p$}. Letting $n$ denote the size of $A$, it is a subset of $\smatng$.

Let
\[
 \partial \fP_p =\{(A,X)\in\allmat: p(A,X)\succeq 0 \quad\mbox{and $\textup{ker}\,p(A,X)\ne\{0\}$} \}
\]
and
\[
 \partial \fP_p^A =\{X\in\SS(\RR^g): p(A,X)\succeq 0 \quad\mbox{and $\textup{ker}\,p(A,X)\ne\{0\}$} \}.
\]
We call these sets the \df{algebraic boundaries} of $\fP_p$ and $\fP_p^A$.

 The sequence $\hbDp=(\hbDp(n))_{n=1}^\infty$, where
\begin{equation*}
 \begin{split}
  \hbDp(n) &= \{(A,X,v)\in\allmatn\times\RR^{\kappa n} : \\ &p(A,X)v=0, \, v\ne 0\, \textrm{and}\, p(A,X)\succeq 0\}
  \end{split}
\end{equation*}
 is the \df{detailed algebraic boundary} of the positivity set $\fP_p.$

\subsubsection{Free sets}
  Positivity sets are special cases of free sets.
  Typically, constructions in this article  are parametrized over all $n$. Thus,
  sequences $S=(S(n))_{n=1}^\infty$ where, for each $n$, the set $S(n)$ is a subset
  of  $\SS_n(\RR^\ell)$ for an appropriate choice of $\ell$
  figure prominently and are called \df{graded sets}.  Such a sequence is a \df{free set}
  if it is closed under direct sums
  and (simultaneous) real unitary conjugation. This last
  condition means, if $U$ is an $n\times n$ real unitary
  matrix and $(A,X)\in S(n)$ (resp. $(A,X,v)\in S(n)$), then $(U^TAU,U^TXU)\in S(n)$ (resp.
  $(U^T AU, U^TXU,  U v)\in S(n)$) too, where
  \begin{equation*}
  \begin{split}
  &U^T(A_1,\ldots,A_{\tg})U:=(U^TA_1U,\ldots,U^TA_{\tg}U)\quad\textrm{and}\\
 & U^T(X_1,\ldots,X_{g})U:=(U^TX_1U,\ldots,U^TX_{g}U).
 \end{split}
 \end{equation*}
  A graded set $S$ is a \df{open} if each $S(n)$ is open. A graded set $S$ is  \df{nonempty} if   $S(n)\ne\emptyset$ for at least one positive integer $n$.

The  notation for projections of the graded sets  $S \subset(\allmatvk)_{n=1}^\infty$  (resp. $S\subset \allmat$)
 \index{$\pi_1$}
\begin{align*}
\pi_1(S)&=\{A:\, (A,X,v)\in S\} \quad (\mbox{resp. } (A,X)\in S)\\
\end{align*}
will be useful.

\subsubsection{Dominating sets}
 A graded set  $\Omega \subset (\allmatvk)_{n=1}^\infty$ is said to \df{$\cC_p$-dominate} a  free set $\cS\subset (\allmatvk)_{n=1}^\infty$ if  $q\in\cC_p^{1\times\kappa}$ and
 $q(\tA,\tX)\tv=0$  for all $(\tA,\tX,\tv) \in \Omega$ implies $q(A,X)v=0 $  for all $(A,X,v) \in \cS$. Note $\Omega$  is not required to be a subset of  $\cS$.

\begin{example}\rm (Domination and Majorization)
\label{ex:sep17a15}
Let
\[
p(a,x)=(1+x^k)a+a(1+x^k)+x^d
\]
 Thus $p$ is a polynomial of the type appearing in Theorem \ref{thm:alinScalar}.   Observe that
$$
\cR\cC_p=\{1,x,\ldots,x^{\ell-1}, a, xa,\ldots, x^{k-1}a\}\quad \textrm{with $\ell=\max\{k,d\}$}
$$
and a polynomial $q(x,a)$ in the linear span of the terms in the  chip set $\cC_p$ of $p$ has the form
$$
q(a,x)=\sum_{j=0}^{\ell-1}\alpha_j x^j+\sum_{j=0}^{k-1}\beta_j x^j a\quad \textrm{with $\ell=\max\{k,d\}$}.
$$
Because $\kappa=1$ ($p$ is scalar polynomial) and either $x^ka$ or $x^d$ is a highest degree word, the highest degree terms majorize at each $A$ by Remark \ref{rem:special majorize}.

Let $\Omega_n$ denote the set of $(A,X,v)\in\allmatn\times\RR^n$ such that
\begin{enumerate}
\item[\rm(1)] $A$ and $X$ are real $n\times n$ diagonal matrices and $v\in\RR^n$.
\item[\rm(2)]  $I_n+X^k$ is invertible.
\item[\rm(3)] $A=-\frac{1}{2}(I_n+X^k)^{-1}X^d$.
\end{enumerate}
We will  show that $\Omega= (\Omega_n)_{n=1}^\infty$ is $\cC_p$ dominating for the free set $\cS=(\cS_n)_{n=1}^\infty$ defined by
$$
\cS_n:=\{(A,X,v)\in\allmatn\times\RR^n: p(A,X)v=0\}.
$$
If $(A,X,v)\in\Omega_n$,
then $AX=XA$ and hence $p(A,X)=0$ if and only if
$$
 A=-\frac{1}{2}(I_n+X^k)^{-1}X^d.
$$
For this choice of $A$,
\begin{align*}
 q(A,X)&=\sum_{j=0}^{\ell-1}\alpha_j X^j +\sum_{j=0}^{k-1}\beta_j X^j A\\
 &=\sum_{j=0}^{\ell-1}\alpha_j X^j -\frac{1}{2}\sum_{j=0}^{k-1}\beta_j X^j (I_n+X^k)^{-1}X^d\\
 &=\frac{1}{2}(I_n+X^k)^{-1}\left\{2(I_n+X^k)\sum_{j=0}^{\ell-1}\alpha_j X^j -\sum_{j=0}^{k-1}\beta_j X^j X^d
 \right\}.
 \end{align*}
 Now let  $n=\ell+k$ and $X=\textup{diag}\{\mu_1,\ldots,\mu_{\ell+k}\}$ with $0<\mu_1<\cdots<\mu_{\ell+k}$ and suppose first that $k\le d$ so that $\ell=d$. Then
 $q(A,X)=0$ if and only if
 $$
 \begin{bmatrix}1&\mu_1&\cdots&\mu_1^{d+k-1}\\ \vdots & & &\vdots\\
 1&\mu_{d+k}&\cdots&\mu_{d+k}^{d+k-1}\end{bmatrix}\left\{\begin{bmatrix}2\alpha_0\\ \vdots \\ 2\alpha_{d-1}\\ -\beta_0\\ \vdots\\ -\beta_{k-1}\end{bmatrix}+\begin{bmatrix} 0_{k\times 1}\\ 2\alpha_0\\ \vdots\\ 2\alpha_{d-1}\end{bmatrix}\right\}=0.
 $$
 Since the (Vandermonde) matrix on the left is invertible, it is readily seen that $q(A,X)=0$ if and only if
 $\alpha_j=0$ for $j=0,\ldots,d-1$ and $\beta_j=0$ for $j=0,\ldots,k-1$;  i.e., if and only if $q$ is the zero polynomial.

 The proof for $k>d$, i.e., for $\ell =k$, leads to the constraint
  $$
 \begin{bmatrix}1&\mu_1&\cdots&\mu_1^{2k-1}\\ \vdots & & &\vdots\\
 1&\mu_{2k}&\cdots&\mu_{2k}^{2k-1}\end{bmatrix}\left\{\begin{bmatrix}2\alpha_0\\ \vdots \\ 2\alpha_{k-1}\\ 2\alpha_0\\ \vdots\\ 2\alpha_{k-1} \end{bmatrix}+\begin{bmatrix} 0_{d\times 1}\\
  -\beta_0\\ \vdots\\ -\beta_{k-1}\\ 0_{(k-d)\times 1}
 \end{bmatrix}\right\}=0,
 $$
 which implies that $q(A,X)=0$ if and only if
 $\alpha_j=\beta_j=0$ for $j=0,\ldots,k-1$;  i.e., if and only if $q$ is the zero polynomial.
\qed
\end{example}

\begin{example}\rm
 With $\widetilde{g}=1,$ suppose  $p$ is a symmetric polynomial of the form
 $p(A,X)=A-f(X)$ as  considered in
  \cite{BM}. Note that  $p_a(A, X)[E,0]=E$ clearly maps
  $\smatn$ onto $\smatn$ independent of $X$.
  Thus, in this setting every pair $(A,X)$ is a full rank point for $p$.

  In this case, the class ${\cC}_p$
 consists of polynomials that do not depend on $A$. In \cite{BM} it is also shown that
  for every $X\in\smatng$ there exists a matrix $A\succ 0$ such that
  $p(A,X)\succeq 0$ and $\det p(A,X)=0$.   Consequently, a polynomial $q\in\cC_p$
 that vanishes on $\hbDp$  
   is singular for every $X\in\mathbb{S}_n(\RR^g)$  and hence,
 as is well known (see e.g.,  \cite{BM}), $q=0$, i.e.,
 $\hbDp$ is $\cC_p$ dominating for $\allmat$.
   \qed
\end{example}

\subsection{Statement of the main result}
\label{sec:mainRes}
 Suppose $p\in\cP^{\dddd}$ is a $\kappa \times \kappa$  symmetric matrix
  polynomial in $\tg+g$ free variables of  degree $\widetilde{d}$ in $a$ and degree $d$ in $x$  and  $\cO\subset \allmat$ is an
 free open semialgebraic set.  Let
$$
\hcO(n)=\{(A,X,v)\in \allmatvk :(A,X) \in \cO(n)
 \ \textrm{and}\ v\ne 0\},
$$
i.e.,
\begin{equation*}
\hcO=(\hcO(n))_{n=1}^\infty=\{(A,X,v):\,(A,X)\in\cO\quad\textrm{and}\quad v\ne 0\},
\end{equation*}
and, in keeping with the notation $\fP_p^A =\{X\in\smatng: (A,X)\in\fP_p\}$ that was introduced in \eqref{def:fPpA}, let
\[
\cO^A =\{X: (A,X)\in\cO\}
\]
denote the \df{$A$-cross section of $\cO$}. \index{$\cO^A$}

The sets
 \begin{equation*}
  \partial \fP_p \cap \cO =\{(A,X)\in\cO:\,  p(A,X)\succeq 0 \ \textrm{and}\  \textup{ker}\,p(A,X)\ne \{0\}\}
\end{equation*}
and
$$
\fP_p^A \cap \cO^A=\{X:(A,X)\in\cO\quad\textrm{and}\quad p(A,X)\succ 0\}
$$
 play a prominent role in the formulation of the main theorem, which we are now ready to state.

\begin{theorem}
\label{thm:main}
   Suppose  $p=p(a,x)=p(a_1,\ldots,a_{\widetilde{g}},x_1,\ldots,x_g)$  is a $\kappa \times \kappa$  symmetric matrix
   of polynomials in $\tg+g$ free variables of  degree   $\widetilde{d}$ in $a$ and degree $d$ in $x$,
    and  $\cO\subset \allmat$ is a free open semialgebraic set. If
  \begin{enumerate}[\rm(a)]
   \itemsep=5pt
   \item \label{it:theta1} for each $A\in\pi_1(\partial \fP_p \cap \cO)$   the set
 $\fP_p^A \cap \cO^A$
    is  convex;
   \item  \label{it:fullrank}
         the set of tuples $(A,X) \in \partial \fP_p\cap \cO$ that are full rank for $p$ are dense in $\partial \fP_p\cap \cO$;
   \item \label{it:theta2}  there exists an $N\ge \sum_{j=0}^{\tilde{d}} \tilde{g}^j$ such that $\pi_1(\partial \fP_p(N)\cap \cO(N))$  contains an open set \footnote{Unless otherwise indicated, open means nonempty open.};
   \item \label{it:theta3}  for each $A\in\pi_1(\partial \fP_p\cap \cO)$   the highest degree terms of $p$ majorize at $A$; and
   \item \label{it:YYY}  $\hcO \cap \partial \widehat{\fP}_p$ is a ${\cC}_p$ dominating set for $\hcO$,
  \end{enumerate}
\smallskip
  then  there exists
\begin{enumerate}[\rm(i)]
 \item  a  symmetric polynomial   $\ell\in \ccP^{\dddd}$ affine  linear in $x$;
 \item a positive integer $\rho$ and a
   a matrix free polynomial $R(a) \in \ccP^{\rho\times\rho}$ that is  positive semidefinite on $\pi_1(\xxxx)$;
 \item  a  matrix free polynomial       $S\in\ccP^{\rho\times \kappa}$     linear in $x$   such that
\begin{equation}
  \label{eq:wow}
   p(a, x)= \ell(a, x) - S(a,x)^T R(a) S(a, x).
\end{equation}
\end{enumerate}
   Thus, $p$ is a weighted sum of squares of  terms linear in $x$
   plus an affine  linear term  in $x$.  Moreover, if $p$ does not contain any terms of the form $c\tau(a)^T x_k \omega(a) x_j \sigma(a)$, for words $\tau,\omega,\sigma$ in $a$ with $\omega$ not empty and a nonzero $c\in\mathbb R$, then there is a choice of  $R$ that is independent of $a$ (so that $R\in \mathbb R^{\rho\times \rho}$) and positive semidefinite.
\end{theorem}

\begin{remark}\rm
\label{rem:main}
The proof of Theorem \ref{thm:main} occupies  Sections \ref{sec:boundary} through \ref{sec:proofmain}. It explicitly  constructs $S$ and $R.$
 We don't know the extent to which $S$ and $R$ are unique.

As a slightly weaker, but perhaps more palatable,  conclusion, $R(A)\succeq 0$ for each $n$ and $A\in\smatntg$ such that $\emptyset \ne \fP^A_p \ne \smatng$. Indeed, for any such $A$ there is an $X$ such that $(A,X)\in \xxxx$.
\end{remark}

\begin{remark}\rm
 \label{prop:r=0}
 Suppose $r$ is a free polynomial of degree $\td$ in $\tg$ variables. It is well known that if $r$ vanishes on an open set $\cU\subset \SS_N(\mathbb R^\tg)$ for sufficiently large $N$, then $r=0$.
 A very conservative choice of $N$ is  $N\ge k_b(\tg,\td)$
 where $k_b$
   is the number of words of length at most $\td$ in $\tg$ variables:
    \index{$k_b$}
\begin{equation}
 \label{def:kb}
   k_b=k_b(\tg,\td):= \sum_{j=0}^{\td}  \tg^j.
\end{equation}
A proof is provided in (arXiv) Appendix \ref{appendix:r=0}, since
we could not find this statement
for symmetric variables.
This remark explains the choice of $N$ in Theorem \ref{thm:main} item \eqref{it:theta2}. 
\end{remark}

\begin{remark}\rm
\label{rem:allvars}
In stating Theorem \ref{thm:main} the variables $a_i$ and $x_i$
and matrices  $A_i $ and $X_i$
are assumed symmetric, a choice made to simplify the presentation.   We see no obstruction to Theorem \ref{thm:main} holding for
mixtures of symmetric and arbitrary variables and  matrices.
In early papers, e.g.  \cite{CHSY}, various classes of variables were treated, greatly
encumbering  the presentation. It was found that %
the proofs in other cases are essentially simplifications
of the proofs in the symmetric case. Subsequent papers, including this one,  typically just present the case
of symmetric variables with remarks elucidating the situation for other choices of variables.

In the proof of Theorem \ref{thm:quasiC}  it is convenient to use a version of Theorem \ref{thm:main} in which the $a$ variables are mixed,  i.e., $a=(a,b)$ where $a=(a_1,\dots,a_h)$ are symmetric variables and $b=(b_1,\dots,b_\ell)$ are non-symmetric variables. The $x$ variables and $X$ matrices, on the other hand, are all symmetric. In this setting
Theorem \ref{thm:main} is true with little modification in its proof,
since generally $A$ is fixed and  the proof uses directional derivatives in $x$ but not in $a$. Essentially all complication surrounds myriad properties of these derivatives. Indeed, it is only necessary to make the obvious changes to the statement of Theorem \ref{thm:main}. For instance, the full rank condition in this case asks that the mapping
\[
\SS_n(\mathbb R^h)\times M_n(\mathbb R^\ell) \times \smatng \ni  (E,H) \mapsto p^\prime(A,X)[E,H]\in\SS_{n\kappa}
\]
 is onto.  Here $M_n(\mathbb R^\ell)$ is the set of $\ell$-tuples of $n\times n$ matrices with real entries.
\end{remark}

\section{The proof of Theorem \ref{thm:quasiC}}
\label{sec:matrixBM}
In this section we show how to obtain Theorem \ref{thm:quasiC} from a variant of Theorem \ref{thm:main} (see Remark \ref{rem:allvars}).
This  also provides an opportunity to become familiar with the conditions of Theorem \ref{thm:main} before encountering its rather long proof.

Let $M(a)$ denote the matrix
$$
M(a)=\begin{bmatrix}a_{11}&\cdots&a_{1\kappa}\\ \vdots& &\vdots\\
a_{\kappa1}&\cdots&a_{\kappa\kappa}\end{bmatrix}\quad
$$
with $\tg=(\kappa^2+\kappa)/2$ free noncommutative entries $a_{ij}$ subject to $a_{ij}= a_{ji}^T$.  Thus, there are $\ell:=\frac{\kappa (\kappa-1)}{2}$ non-symmetric variables $a_{ij}$ with $i<j$ and $\kappa$ symmetric variables $a_{jj}.$  We view $a_{ii}$ as the first $\kappa$ of the $a$-variables and $a_{ij}$ with $i<j$ as the last $\ell$ of the $a$-variables.
Thus, for example, if $A=(A_1,\ldots,A_\kappa,B_1,\ldots ,B_\ell)$ and $\kappa=3$, then $\ell=3$ and
$$
M(A)=\begin{bmatrix}A_1&B_1^T&B_2^T\\B_1&A_2&B_3^T\\B_2&B_3&A_3\end{bmatrix}.
$$
Set
$$p(a,x):=M(a)  -f(x).$$
Then
\begin{equation*}
\fP_p =\{(A,X):  M(A)   -f(X)\succ 0 \}
\end{equation*}
\begin{equation*}
\begin{split}
\partial \widehat \fP_p &=\{(A,X,v):
  M(A)   - f(X)  \succeq 0, \\
&[M(A)  - f(X)]v=0\ \textrm{and} \ v \ne 0\}.
\end{split}
\end{equation*}

 The set
$$
   \cO:= \{ (A,X) : \ M(A) \prec \gamma I \}
$$
is a nonempty  free open semialgebraic set. Moreover, with this choice of $\cO$,
\begin{equation*}
\begin{split}
\xxx=\{(A,X,v):&
   M(A)\prec \gamma I, \,
M(A)  -f(X)\succeq 0, \\
&[M(A)  -f(X)]v=0 \,  \textrm{and}\, v\ne 0\}.
\end{split}
\end{equation*}

It suffices to verify the hypotheses Theorem \ref{thm:main}  are met and this we do,  one by one.

\bigskip

\begin{enumerate}[(a)]
\itemsep=18pt
\item The set  $\fP_p^A=\fP_p^A\cap \cO^A$  is  convex for each $A\in\pi_1(\xxxx)$ by hypothesis.
\item The set of tuples $(A,X)\in \partial \fP_p \cap \cO$ that are full rank for $p$ are dense in $\fP_p\cap \cO$.
Indeed, each $(A,X)\in \xxxx$ is full rank for $p$ since  %
 $p_{a,x}(A,X)[E,0]= M(E)$ and $M$ maps $\SS_n(\mathbb R^\kappa)\times M_n(\mathbb R^\ell)$  onto $\SS_{\kappa n}$.
\item For each positive integer $N$, the set  $\pi_1(\partial \fP_p(N) \cap \cO(N))$ contains an open set.
Fix $N$.  For a tuple  $A\in \SS_N(\mathbb R^\kappa)\times M_N(\mathbb R^\ell)$, $M(A)\in\SS_{N\kappa}$  and all elements of $\SS_{N\kappa}$ have this form.
 Fix $A$ such that $\gamma I_{\kappa n} \succ M(A)\succ 0$.  There exists an $X$ such that $M(A)-f(X)\not\succ 0$ by hypothesis. On the other hand, $M(A)-f(0)\succ 0$,  since $f(0)=0$.  Hence, there exists a $0<t \le 1$ and a vector $v\ne 0$ such that $M(A)-f(tX)\succeq 0$ and $(M(A)-f(tX))v=0$. Hence $\pi_1(\partial \fP_p(N)\cap \cO(N))$ contains the  set of $A$ satisfying $\gamma I_{\kappa n}\succ M(A)\succ 0$.
\item
  The highest degree terms of $f$ majorize by hypothesis. %
\item
$\xxx$ is a $\cC_p$ dominating set for $\hcO$. Since elements of ${\cC}_f^{1\times \kappa}$ depend only on $x$ it (more than) suffices to show, if $q(X)v=0$  for $(A,X,v)\in\xxx$,
 then $q = 0$.

Suppose $q(X)v=0$ for all $(X,A,v)\in \xxx$.  Fix $n\ge k_b(d,g) = \sum_{j=0}^d g^j$. Suppose  $X\in \SS_n(\RR^g)$ and  $f(X)\prec \gamma I$. Choose $A\in \SS_n(\mathbb R^\kappa)\times M_n(\mathbb R^\ell)$ with $M(A)=f(X)$. It follows that $(A,X,v)\in \xxx$ for all $v\ne 0$. Hence $q(X)v=0$ for all such $v$ and hence $q(X)=0$ for all $X$ such that $f(X)\prec \gamma I$. Since this is an open set of $X$ and $n$ is sufficiently large, $q=0$ by
Remark \ref{prop:r=0}.
\end{enumerate}

All hypotheses of Theorem \ref{thm:main} are verified. Moreover, $p$ contains no terms of the form $c \tau(a)x_k \omega(a) x_j \sigma(a)$ for words $\tau,\omega,\sigma$ in $a$ with $\omega$ not empty and a nonzero $c\in\mathbb R$. Thus, by Theorem \ref{thm:main},
 $p=M(a)-f(x)$ has the form \eqref{eq:wow} with $R(a)=R$ a constant positive semidefinite matrix.   Since $R$ has a unique positive semidefinite square root that can be absorbed into the factors $s_j(a,x)$, it can thus be assumed that $R$ is the identity and hence
\[
 f(x)= -\ell(a,x)+ \sum s_j(a,x)^T s_j(a,x).
\]
 On the other hand, $f$ does not depend upon $a.$ Hence none of the $s_j(a,x)$ can depend upon $a$ as otherwise the highest degree terms in $a$ can not cancel.  Thus $s_j(a,x)=s_j(x)$. Finally, it now follows that $\ell(a,x)=\ell(x)$ also does not depend upon $a$.   We note there is a more direct argument using Lemma \ref{lem:aug6a15}.


\section{The proof of Theorem \ref{thm:alinScalar}}
\label{sec:ex}
In this section we show how to obtain Theorem \ref{thm:alinScalar} from Theorem \ref{thm:main}
Sample computations for the concrete special case where  $r(x)=1+x^k$ and $2f(x)=-x^d$ appear in Example \ref{ex:sep17a15}.

Define the  free open semialgebraic set
$$
\cO:= \{(A,X) \in \SS(\RR^2):\, X\in B_\eps\}
$$
and observe
\begin{equation*}
  \fP_p^A=\{X \in\SS:\,   r(X)A+Ar(X)-2f(X) \succ 0 \}
\end{equation*}
and, for each positive integer $n$,
\begin{equation*}
\begin{split}
 \partial \widehat \fP_p(n)& =\{(A,X,v)\in\allmatv:
  r(X)A+Ar(X)-2f(X) \succeq 0, \\
 &[r(X)A+Ar(X)-2f(X) ]v=0\ \textrm{and} \ v \ne 0\}.
\end{split}
\end{equation*}

Next,  we shall check that the hypotheses of Theorem \ref{thm:main} are met, one by one.

\begin{enumerate}
\item[\rm(a)]  For each $A\in\pi_1(\xxxx)$
   the set  $\fP_p^A\cap{\cO}^A$ is convex by assumption.

\ms

\item[\rm(b)]
 $\{(A,X)\in \xxxx: p_{a,x}(A,X)[E,H]\ \textrm{maps} \ (E,H)\in\SS_n(\RR^{\fg}) \mbox{ onto }\SS_{n}\}$ is dense in $\pi_1(\xxxx)$.

 Let $(A,X)\in\allmat$ be given. It is easily checked that
$$
p_{a,x}(A,X)[E,0]=   r(X) E + E r(X)
$$
and hence $(A,X)$ is a full rank point for $p$ if the map
$$
E \to r(X) E + E r(X)
$$
is invertible, i.e., if
  $\sigma(r(X)) \cap \sigma(-r(X))=\emptyset,$
where $\sigma(\pm r(X))$ denotes the set of eigenvalues of $\pm r(X)$.
But this can be achieved by arbitrarily small perturbations of $X$, since $r(x)$ is a nonzero polynomial.
 The desired conclusion follows.

\bs

\item[\rm(c)]  To prove for each positive integer $n$ the set $\pi_1(\partial \fP_p(n)\cap\cO(n))$ contains an open set, choose an interval $(u,v)\subset (-\eps,\eps)$ on which $r$ is positive and  $f^o=fr^{-1}$ is strictly monotonic. (This is possible because $r(0)=1$  and $fr^{-1}$ is not constant.) For simplicity, suppose $f^o$ is increasing on this interval and thus $f^o((u,v))=(f^o(u), f^o(v)).$ Given points  $f^o(u)<a_1\le a_2 \le \dots \le a_n < f^o(v),$ there are uniquely determined points $u<x_1\le x_2 \le \dots \le x_n <v$ such that $f^o(x_j)=a_j$. Let $X$ denote the diagonal matrix with entries $x_j$ and $A$ the diagonal matrix with entries $a_j=f^o(x_j)$. Then $X\in B_\eps$ and $p(A,X)=0$.  Thus, $A\in \pi_1(\partial \fP_p(n)\cap\cO(n))$.  Moreover, $p(U^TAU,U^TXU)=0$ for every real unitary matrix $U$.  Therefore, $\pi_1(\partial \fP_p(n)\cap\cO(n))$ contains every $A \in\SS_n$ whose spectrum lies in the interval $(f^o(u),f^o(v))$ and hence contains an open subset of $\SS_n$.

\ms

\item[\rm(d)] By Remark \ref{rem:special majorize}\eqref{it:sm2}, the condition the  highest degree terms of $p$ majorize is satisfied.

\ms

\item[\rm(e)]
If $q$ lies in the span of the right chip set of $p$, then $q$ has the form
$$
 q(a,x)=\varphi(x)a+\psi(x),
$$
where $\varphi(x)$ and $\psi(x)$ are polynomials in $x$ alone of degrees at most $k-1$ and
 $\max\{k-1,\,d-1\}$, respectively. To verify the domination condition of item \eqref{it:YYY} of Theorem \ref{thm:main}, it suffices  to show,
if  $q(A,X)v=0$ for each positive integer $n$ and triple $(A,X,v)\in \mathbb{S}_n\times B_\eps(n) \times\RR^n$  such that $v\ne 0$,  $p(A,X)v=0$ and $p(A,X)\succeq 0$, then $q=0$.

\ms

Let $X$ be a diagonal matrix such that $r(X)$ is invertible.
The substitution  $A=B +  \ r(X)^{-1}f(X)$ into $p$ and $q$ gives
\[
  p(A,X)=r(X)A+Ar(X)-2f(X) =r(X)B+Br(X)
\]
and
\[
\begin{split}
  q(A,X)&=\varphi(X)A+\psi(X)\\ &=\varphi(X)B+\psi(X) + \  \varphi(X)r(X)^{-1}f(X),
\end{split}
\]
respectively.  Since $r(X)$ is invertible if $\Vert X\Vert<\varepsilon$ and $\varepsilon$ is a small enough positive number, it suffices to show that, if
\begin{equation}
\label{eq:aug14b12}
\big(\varphi(X)B+\psi(X) + \ \varphi(X)r(X)^{-1}f(X)\big)v=0,
\end{equation}
 for every triple $(B,X,v)$  with $B$ real symmetric, and $X$ real diagonal  that satisfies the constraints
\begin{equation}
\label{eq:aug27a15}
 \| X \|<\eps,\  r(X)B+Br(X)\succeq 0\ \textrm{and}\  (r(X)B+Br(X))v=0,
\end{equation}
then
\begin{equation}
\label{eq:may2a17}
\varphi(X)B+\psi(X) + \ \varphi(X)r(X)^{-1}f(X)=0.
\end{equation}

 The constraints in \eqref{eq:may2a17} are met for every $v$ if $B=0$ and $\Vert X\Vert<\varepsilon$. and hence in this case \eqref{eq:aug14b12} implies that
$$
r(X)\psi(X)+\varphi(X)f(X)=0.
$$
 Thus, \eqref{eq:aug14b12} reduces to $\varphi(X)Bv=0$ and \eqref{eq:may2a17} reduces to $\varphi(X)B=0$.

 Next, since $r(x)$ is a non constant polynomial with $r(0)=1$, we may  choose a pair of real numbers $x$ and $y$ such that  $|x|<\eps$, $r(x)>0$, $|y|<\eps$, $r(y)>0$
and $r(x)\ne r(y)$. Now let
\[
X:= \begin{bmatrix} x & 0\\ 0& y \end{bmatrix}
\quad \textrm{and}\quad
B:= \begin{bmatrix} a & b\\  b & c \end{bmatrix},
\]
where $a,b,c$ are real numbers. Then $\|X\|<\eps$ and
\begin{equation*}
r(X)B+Br(X) = \begin{bmatrix}
2 a r(x) & b (r(x) + r(y) ) \\
b (r(x) + r(y) ) & 2c  r(y)
\end{bmatrix}.
\end{equation*}

Now choose $a>0$, $c>0$,
$$
b=2\frac{\sqrt{acr(x)r(y)}}{r(x)+r(y)},\quad t=-\sqrt{\frac{ar(x)}{cr(y)}}\quad\textrm{and}\quad v=\begin{bmatrix}1\\ t\end{bmatrix}.
$$
Then, $\varphi(X)Bv=0$ for a set of triples $(B,X,v)$ of the requisite form that meet the constraints in
 \eqref{eq:aug27a15}.
  Since $X$ is a diagonal matrix, it is also easily seen that
the constraints in \eqref{eq:aug27a15} are met
if $B$ and $v$ are replaced by $jBj$ and $jv$, where $j$ is the signature matrix $j=\begin{bmatrix}1&0\\0&-1\end{bmatrix}$. Thus, $$\varphi(X)jBj^2v=\varphi(X)jBv=0.$$ Consequently,
$$
\varphi(X)\,\begin{bmatrix}Bv&jBv\end{bmatrix}=\varphi(X)\,\begin{bmatrix}a+bt&a+bt\\ b+ct&-(b+ct)\end{bmatrix}=0.
$$
Since
$$
a+bt=a\,\frac{r(y)-r(x)}{r(x)+r(y)}\ne 0\quad\textrm{and}\quad b+ct=tc\frac{r(x)-r(y)}{r(x)+r(y)}\ne 0
$$
the matrix $\begin{bmatrix}Bv&jBv\end{bmatrix}$ is invertible. Thus, $\varphi(X)=0$, for all $x$ in an open interval. Hence $\varphi=0$.
\end{enumerate}

Since all hypotheses of Theorem \ref{thm:main} are verified, $p$ is of the form of \eqref{eq:wow}.
Moreover, as
  $p$ does not contain terms of the form $c \tau(a) x_k \omega(a) x_j \sigma(a)$ for  words $\tau, \omega,\sigma$ in $a$ with $\omega$ not empty and a nonzero $c\in\mathbb R$,   Theorem \ref{thm:main} guarantees the existence of a  representation \eqref{eq:wow} with $R\in\mathbb R^{\rho\times \rho}$  a positive semidefinite matrix that does not depend upon $a$. Since $R$ has a unique positive semidefinite square root that can be absorbed into the factors $s_j(a,x)$,
we may assume that $R$ is the identity and follow essentially the same argument as at the end of the proof of Theorem \ref{thm:quasiC} (see the end of Section \ref{sec:matrixBM}) to conclude that $s_j(a,x)=s_j(x)$ does not depend upon $a$.

\section{Boundaries and tangent spaces}
\label{sec:boundary}
  Given $(A, X, v)  \in \allmatvk$, let
\begin{equation*}
  \cT(A,X,v)= \{H\in\smatng : p_x(A, X)[0,H]v=0 \} .
\end{equation*}
  In the case $(A,X,v)\in \partial \fP_p$,
  the  subspace $\cT(A, X,v)$  is the
 \df{clamped tangent plane}  at $(A, X,v)$ in the terminology of \cite{DHM07a}. \index{$\cT(A,X,v)$}

 If $p(A,X)v=0$ and  $H\in\cT(A, X,v)$, then, as follows easily from the general formula,
 \begin{equation*}
 \begin{split}
 p(A,X+tH)=&p(A,X)+tp_x(A,X)[0,H]+\frac{t^2}{2!}p_{xx}(A,X)[0,H]\\ &+\frac{t^3}{3!}p_{xxx}(A,X)[0,H]+\cdots\, ,
 \end{split}
 \end{equation*}
\[
 \langle p(A, X+tH)v,v\rangle =  \frac{1}{2}
       t^2 \langle p_{xx}(A, X)[0,H]v,v\rangle
    +t^3 e(t),
\]
 for some polynomial $e(t)$. This identity provides a link between
 convexity and negativity  of the Hessian of $p$, much as in the commutative  case.

\begin{proposition}
 \label{prop:maintool}
  Suppose   $p\in\cP^{\kappa\times\kappa}$ is  symmetric,  $(A,X,v)\in\allmatvk$ and   $v\ne 0$. If
 \begin{enumerate}[\rm(i)]
    \item \label{it:in boundary1}  $p(A,X)\succeq 0$;
    \item \label{it:in boundary2}  $p(A,X)v=0$;
   \item  \label{it:dim one}the dimension of the kernel of
 $p(A, X)$ is one;
 \item \label{it:convex} there is an open subset $\mathscr W$ of $\smatng$ containing  $X$ such that the  open set $\fP_p^A\cap \mathscr W$ is convex;
 \end{enumerate}
  then there exists a subspace $\cH$ of $\cT(A,X,v)$ of codimension at most one  (in $\cT(A,X,v)$)
  such that
 \begin{equation}
  \label{eq:poshess}
    \langle p_{xx}(A, X)[0,H]v,v\rangle \le 0 \quad \text{for}\ H\in\cH.
 \end{equation}
\end{proposition}

\begin{remark}\rm
  Note that the codimension of
  $\cT(A,X,v)$ in $\smatng$ is at most $n\kappa$.
  Hence the inequality of Equation \eqref{eq:poshess} holds
  on a subspace of $\smatng$ of codimension at most $n\kappa+1$.

  Unlike a related argument in \cite{DHM07a},
  the proof here does not rely on choosing a curve lying
  in the boundary of a convex set, thus eliminating the need
  for a corresponding smoothness hypothesis.
  \end{remark}

\begin{proof}
 If $C$ is an open convex subset of $\mathbb R^N$  and $y\notin C$, then there is a vector $w\in\RR^N$ such that
\[
   \langle y,w\rangle   > \langle x,w\rangle
\]
  for all $x\in C$.  This separation result, applied  to the tuple $X$ lying outside  the open convex set
 $\fP_p^A\cap \mathscr W$ (see item \eqref{it:convex}), guarantees the existence of a linear functional
 $\Lambda:\smatng \to \mathbb R$  such that $\Lambda(Z)<1$
 for $Z\in \fP_p^A\cap \mathscr W$ and $\Lambda(X)=1$.   Thus, as
 $$
 \textup{dim}\,\cT=\textup{dim\, ker}\,\Lambda\vert_{\cT}+\textup{dim\,range} \,\Lambda\vert_{\cT}
 $$
 and
  the dimension of the range of $\Lambda$ is one,
  the subspace
\[
  \cH=\{H \in \cT(A, X,v) : \Lambda(H)=0 \}
\]
  has codimension   one in $\cT(A, X,v)$.

  Fix $H\in\cH$ and define $F:\mathbb R\to \smatnk$ by
  $F(t)=p(A, X+tH)$.  Thus, $F$ is a matrix-valued polynomial in
  the real variable $t$.
  Let $U$ be an orthogonal matrix in $\RR^{\nk \times \nk}$ with its last
  column proportional to $v$ and write
  $$
    F(t)=U\begin{bmatrix} Q(t) & g(t) \\ g(t)^T & f(t) \end{bmatrix}U^T,
$$
where $Q(t)\in\SS_{\nk-1}$, $g(t)\in\RR^{\nk-1}$  and $f(t)\in\RR$ are polynomials of degree
  at most $d$ in $t$. Thus,
  $$
  F(t)v=U\begin{bmatrix}g(t)\\f(t)\end{bmatrix}\Vert v\Vert
  $$
  and  the  assumption
  $p(A, X)v=0$  of item \eqref{it:in boundary2} implies that $f$ and $g$ vanish at $0$.
  The supplementary assumption that $H\in\cT(A, X,v)$ implies that
  $$
  F^\prime(0)v=U\begin{bmatrix} Q^\prime(0) & g^\prime(0) \\ g^\prime(0)^T & f^\prime(0) \end{bmatrix}U^Tv=
 U\begin{bmatrix} g^\prime(0) \\ f^\prime(0)
 \end{bmatrix}\Vert v\Vert=0,
    $$
 i.e.,  $f$ and $g$ vanish to second order at $0$.
  Therefore, there are polynomials $\beta$ and $\gamma$ such that
  $g(t)=t^2\beta(t)$ and $f(t)=t^2\gamma(t)$.
Moreover,
\[
   \langle p_{xx}(A, X)[0,H]v,v\rangle=\langle F^{\prime\prime}(0)v,v\rangle= \Vert v\Vert^2f^{\prime\prime}(0)
  =2\Vert v\Vert^2
  \gamma(0).
\]
  Thus, to complete the proof of the proposition it suffices to
  use the choice of $\Lambda$ (and thus the convexity
  hypothesis on $\fP_p^A\cap \mathscr W$) and the
  assumption on the dimension of the kernel
  of $F(0)=p(A,X)$  to show that $\gamma(0)\le 0$.
  Indeed, since the kernel of $F(0)$ has dimension one and $F(0)\succeq 0$ by item \eqref{it:in boundary1}, it follows that
  $Q(0) \succ 0$. Therefore,
  $Q(t) \succ 0$ for $|t|$ sufficiently small.
  On the other hand,  since $H\in\cH$,
  $$
  \Lambda(X + tH) =\Lambda(X)+t\Lambda(H)=
  \Lambda(X) = 1
  $$ for all $t$.  Thus,  $X + tH \not \in \fP_p^A\cap \mathscr W$. %
  On the other hand, for  $|t|$ sufficiently small,   $X+tH\in \mathscr W$ and hence $X+tH\not\in \fP_p^A$. Thus
  for $|t|$ sufficiently small both
  $F(t) = p(A, X + tH) \not \succ 0$ and  $Q(t)\succ 0$.
  \footnote{If $\fP_p^A$ is replaced by $\cD^A_p$, \index{$\cD^A_p$} the component of $0$ of $\fP_p^A$, the failure
  of positivity of $F(t)$ would  not guarantee that $X+tH\not\in \cD_p(A)$.}
   Hence, for such $t$, the Schur complement of $Q$ is nonpositive; i.e.,
\[
  t^2[\gamma(t) - t^2 \beta^T(t)Q^{-1}(t)\beta(t)] \le 0.
\]
Therefore, \[
 \gamma(t) \le t^2 \beta^T(t)Q^{-1}(t)\beta(t)\quad\textrm{for $t \in (0, \epsilon)$}
\]
 and hence $\gamma(0)\le 0$.
\end{proof}


\subsection{Kernel of Dimension 1}
The main result of this subsection, Proposition  \ref{lem:generic},  extends the applicability
 of Proposition \ref{prop:maintool}.
A full rank assumption, similar to condition (b) in  Theorem \ref{thm:main},  is imposed   to justify an application of the implicit function theorem.

First  observe  that if $(A,X)\in\allmatn$,  then the matrix $p(A,X)$ belongs to
$\SS_{n\kappa}$ and can be identified as an
$s\times 1$ vector with $s=(m^2+m)/2$ and $m=n\kappa$.
 Likewise by identifying $\allmatn$ with $\mathbb R^{r}$ for $r=\mathfrak{g} (n^2+n)/2$,
  the mapping
\[
 \allmatn \ni (A,X) \mapsto p(A,X)
\]
 can be identified with a mapping $f:\mathbb R^r\to\mathbb R^s$,
$$
f(y)=\begin{bmatrix}f_1(y_1,\ldots,y_r)\\ \vdots \\ f_s(y_1,\ldots,y_r)\end{bmatrix}
$$
  Moreover, the Jacobian matrix
$$
J_f(y)=\begin{bmatrix}\frac{\partial f_1}{\partial y_1}(y)&\cdots&\frac{\partial f_1}{\partial y_r}(y)\\
\vdots & &\vdots\\ \frac{\partial f_s}{\partial y_1}(y)&\cdots&\frac{\partial f_s}{\partial y_r}(y)
\end{bmatrix}
$$
can be identified with
$$
\begin{bmatrix} p_{a,x}^\prime(Y)(G_1)&\cdots&p_{a,x}^\prime(Y)(G_r)\end{bmatrix}
$$
for an appropriate choice of points $G_j\in\allmatn$, for $j=1,\ldots,r$. Thus, the statement that
$$
p_{a,x}(A,X)[E,H] \quad \textrm{maps $Y=(E,H)\in\allmatn$ onto  $\SS_{n\kappa}$}
$$
is equivalent to the statement that the rank of the Jacobian matrix is equal to  the dimension of $\SS_{n\kappa}$.

\begin{proposition}
 \label{lem:generic}
  Suppose $(A^o,X^o,v)\in \allmatvk$, $v\not =0$, and
  $p(a,x)$ is a symmetric $\kappa\times\kappa$ matrix polynomial in the noncommuting variables
  $a$ and $x$. If
  \begin{enumerate}[\rm(1)]
\item  $p(A^o, X^o)v=0$ and
 \item $p'(A^o, X^o)[E,H]$ maps $(E ,H)\in\mathbb{S}_n(\RR^\mathfrak{g})$ onto $ \SS_{n\kappa}$
(i.e., $(A^o,X^o)$ is a full rank point for $p$),
\end{enumerate}
then, for each $\varepsilon >0,$ there is a full rank point $(A, X) \in\allmatn$
  such that $\|A-A^o\|<\varepsilon$,  $\|X - X^o\|<\varepsilon$    and the kernel of $p(A, X)$  is spanned by $v$.

  Furthermore, if $p(A^o,X^o)\succeq 0,$ then $(A,X,v)$ can be chosen so that
  $p(A,X)\succeq 0$ too.
\end{proposition}

\begin{proof}
 For $A,X\in\SS_n(\RR^{\fg})$ and  $C\in\SS_{n\kappa}$, let
\[ f(A,X,C):= [p(A,X)-C]v \quad\textrm{and}\quad C^o=p(A^o,X^o).
\]
Then  $f(A^o,X^o,C^o)=0$ and,
as the derivative of $f$ with respect to the variables $A$ and $X$ has full rank, the implicit function
theorem implies
that there is a neighborhood $\cN$ of $C^\circ$
  and a neighborhood $\cN^\prime$ of $(A^\circ,X^\circ)$  and a continuous mapping
  $g:\cN\to \cN^\prime$ such that $f(g(C),C)=0$. Let $U\in\RR^{n\kappa\times n\kappa}$ be a unitary matrix with its last column proportional to $v.$ Thus,
\[
  C^o=U\begin{bmatrix} Q^o&\beta^o\\ (\beta^o)^T&\gamma^o\end{bmatrix}U^T\quad\textrm{with
  $\beta^o\in\RR^{n\kappa-1}$ and $\gamma^o\in\RR$}.
\]
  As
  $$
  C^ov=U\begin{bmatrix}\beta^o\\ \gamma^o\end{bmatrix}\Vert v\Vert=0,
$$
it follows that $\beta^o=0$ and $\gamma^o=0$ and hence that
$$
C^o=U\begin{bmatrix} Q^o&0\\ 0&0\end{bmatrix}U^T.
$$
Choose $C\in\cN$ with
$$
C=U\begin{bmatrix} Q&0\\ 0&0\end{bmatrix}U^T,\quad \textrm{$Q=Q^T$  invertible and
$\Vert Q-Q^o\Vert$ small}.
$$
If $(A,X)=g(C)$, then
$$
p(A,X)v=Cv=U\begin{bmatrix} Q&0\\ 0&0\end{bmatrix}U^Tv=U\begin{bmatrix} Q&0\\ 0&0\end{bmatrix}\begin{bmatrix}0_{(n\kappa-1)\times 1}\\ 1\end{bmatrix}=0
$$
and, since $Q$ is invertible,  the dimension of the kernel of $p(A,X)$ is equal to one.

If $p(A^o,X^o)\succeq 0$, then  $Q^o \succeq 0$  and $Q=Q^T$ may be chosen positive definite.
\end{proof}


\subsection{The Hessian on a Tangent Plane vs. the Relaxed Hessian}
 \label{sec:curvature}
 Our main tool for analyzing the curvature of noncommutative real varieties
 is a variant of the Hessian for symmetric $\kk$-valued  polynomials $p$ of
 degree $d$ in $g$ noncommuting variables.
The curvature of an nc variety  $\cV(p)$ is defined in terms of the Hessian of $p$
compressed to tangent planes, for each dimension $n$.
Since this compression of the Hessian is awkward to work with directly,
we introduce a quadratic polynomial $F(a,x)[h]$ (defined
for all $H\in \Smatng$, not just for $H \in \cT(A,X,v)$)
called the relaxed Hessian. This  approach is taken in \cite{DHMill} and
used in \cite{BM};\S 3 of
\cite{DHMill} should be referred to
for motivation and more details; see especially Example 3.1.

Given a symmetric nc matrix polynomial $p\in\cP^{\kappa\times\kappa}$  with right chip sets $\cR\cP_p^j$, $j=1,\ldots,g$, let $\msUalt(a,x)[h]$ denote the column vector with entries $h_j w(a,x)$ for $1\le j\le g$ and $w\in \cR\cP_p^j$.
The \df{relaxed Hessian} of $p$  is defined to be the polynomial\footnote{The tensor product of the matrix $I$ with the vector (column matrix) nc polynomial is defined either abstractly or concretely by the Kronecker product just as for the tensor product of matrices. The Kronecker product of two matrix nc polynomials requires more care as discussed later in Subsection \ref{sec:identities}.}
\begin{equation}
 \label{eq:defrelax}
 \begin{split}
   p^{\prime\prime}_{\lambda, \delta}(a,x)[h]:&=
      p_{xx}(a,x)[h]+ \delta   I_\kappa\otimes  \msUalt(a,x)[h]^T\msUalt(a,x)[h] \\
      &   +\lambda \,  p_x(a,x)[h]^T p_x(a,x)[h].
      \end{split}
\end{equation}

  Suppose $(A,X) \in \allmatn$ and $v\in \mathbb R^{\nk}$.
  We say that the {\bf  relaxed Hessian is positive}
  \index{relaxed Hessian, positive} at $(A,X,v)$ if
  for each $\delta>0$ there is a $\lambda_\delta>0$ so that
  for all $\lambda>\lambda_\delta$
$$
   0\le
    \langle p^{\prime\prime}_{\lambda, \delta}(A,X)[h]v,v\rangle
$$
 for all $H\in\smatng$.
 Correspondingly we say that the {\bf  relaxed Hessian is negative}
 \index{relaxed Hessian, negative} at
$(A, X,v)$ if
 for each $\delta<0$ there is a $\lambda_\delta<0$ so that
 for all $\lambda \le \lambda_\delta$,
$$
   0\le
    -\langle p^{\prime\prime}_{\lambda, \delta}(A,X)[h]v,v\rangle
$$
 for all $H\in\smatng$.
 Given a sequence $\cS=( \cS_n)_{n=1}^\infty$, with
$\cS_n\subseteq(\allmatn\times \mathbb R^{\nk})$,
 we say that the
{\bf relaxed Hessian is positive (resp., negative) on $\cS$}
 \index{relaxed Hessian, positive on $S$} if
 it is positive (resp., negative) at each $(A, X,v)\in \cS$.

 Suppose $f(a,x)[h]$ is a $\kk$-valued free symmetric polynomial in the $\tg+2g$ symmetric variables
  $a$, $x $ and $h$  of degree $s$ in $x$
  and homogeneous of
  degree two  in $h$.
  Given a subspace $\mathcal H$ of
  $\allmatn$, let
\begin{equation}
\label{def:epmn}
e_\pm^n(A,X,v;f,\mathcal H)
\end{equation}
denote the  maximum dimension of a strictly positive/negative subspace  of $\mathcal H$  with respect to the quadratic form
 \begin{equation*}
     \mathcal H \ni H \mapsto \langle f(A, X)[h]v,v \rangle.
 \end{equation*}
   Here  strictly positive (resp., negative) subspace $\mathcal H$ means
$$\langle f(A, X)[h]v,v\rangle >0\ (\textrm{resp.}\ <0)
   \ \textrm{for}\  H \in \mathcal H, \ H\ne 0.$$
    \index{$e_\pm^n$}

 Lemma \ref{thm:signature-clamped-relaxed} below provides a link
 between the signature of the clamped
 second fundamental form (i.e., the Hessian compressed to the
 clamped tangent space $\cT$)  and that of the
 relaxed  Hessian (a quadratic form on all of $\allmatn$).

\begin{lemma}
 \label{thm:signature-clamped-relaxed}
   Suppose $p$ is a $\kappa\times\kappa$
 symmetric nc matrix polynomial of degree $\widetilde{d}$ in $a$, degree $d$ in $x$
    and $(A,X,v)\in\allmatvk$. Then
  there exists an $\varepsilon<0$
    such that for each $\delta\in [\varepsilon,0)$  there exists a
$\lambda_\delta <0$ so that for every $\lambda\le \lambda_\delta$,
\begin{equation*}
    e_+^n(A, X,v;p^{\prime\prime}_{\lambda,\delta}, \Smatng) =
    e_+^n(A,X,v;p_{xx}, \cT(A,X,v) ).
 \end{equation*}

 Similarly,  there exists an $\varepsilon>0$
    such that for each $\delta\in (0,\varepsilon]$  there exists a
$\lambda_\delta >0$ so that for every $\lambda\ge \lambda_\delta$,
\begin{equation*}
    e_-^n(A, X,v;p^{\prime\prime}_{\lambda,\delta}, \Smatng) =
    e_-^n(A,X,v;p_{xx}, \cT(A,X,v) ).
 \end{equation*}
\end{lemma}

\begin{remark}\rm
  The lemma does not require $p(A,X)v=0$ or even $p(A,X)\succeq 0$.
\end{remark}

 The proof of Lemma \ref{thm:signature-clamped-relaxed} is postponed in favor of  two preliminary lemmas.

\begin{lemma}
 \label{lem:I=0}
     Suppose $p$ is a $\kappa\times\kappa$
 symmetric nc matrix polynomial of degree $\widetilde{d}$ in $a$, degree $d$ in $x$
    and $(A,X,v)\in\allmatvk$.  If $H\in\smatng$ and
\begin{equation}
\label{eq:nov6a16}
  \left\{I_\kappa \otimes (\msUalt(A,X)[H]^T \msUalt(A,X)[H])\right\}v=0,
\end{equation}
 then $p_{xx}(A,X)[h]v=0$ and $p_x(A,X)[h]v=0$.
\end{lemma}

\begin{proof}
 Write $v=\oplus_{j=1}^\kappa v_j$ with $v_j\in\mathbb R^n$. Since
$$
 I_\kappa\otimes (\msUalt(A,X)[H]^T \msUalt(A,X)[H])
    =(I_\kappa \otimes \msUalt(A,X)[H])^T(I_\kappa \otimes \msUalt(A,X)[H]),
 $$
 the assumption that $(I_\kappa \otimes \{\msUalt(A,X)[H]\big)^T  (\msUalt(A,X)[H]\big)\}v=0,$
 the hypothesis \eqref{eq:nov6a16} implies that
 \begin{equation}
 \label{eq:Hwv0}
 H_j w(A,X)v_k=0 \mbox{ for each $j,$  $w\in \cC_p^j$
  and $1\le k\le \kappa$. }
  \end{equation}
 By the definition of the right chip sets for $p$, each word in $p-p(0)$
 is of the form $ux_j w$, for some $j,$ some word $w\in \cC_p^j$ and some
 word $u$.
 Correspondingly,
 each word in  the polynomial $p_{x}$
 is of the form $u h_j w$
 and each word in  the polynomial $p_{xx}$
 is of the form $v h_j w$
 where $v$ is a polynomial in $a$,  $x$ and $h$.
 Hence each term in $p_{x}(A,X)[h]v$
 has the form  $u(A,X) H_j w(A,X)v_k$
 which, by \eqref{eq:Hwv0} is $0$, implying $p_{x}(A,X)[h]v$.
 Likewise for $p_{xx}$.
 \end{proof}

Let $\mu_+(R)$ (resp., $\mu_-(R)$) denote the number of positive (resp., negative) eigenvalues of a real symmetric matrix $R$. \index{$\mu_+$}

\begin{lemma}
\label{lem:sep1a14}
Let $E,\,F,\,G\in\mathbb{S}_n$
 be given and  let $P$ denote the
 orthogonal projection of $\RR^n$ onto $\ker(F)$, the kernel of $F$.
If
\begin{enumerate}[\rm(i)]
 \item  $F\neq 0$;
 \item  $F\succeq 0$, $G\succeq 0$;  and
 \item  $\ker(G) \subseteq \ker(F) \cap \ker(E),$
\end{enumerate}
then there exists a number  $\varepsilon<0$ such that for each  $\delta\in [\varepsilon,0),$
there exists a $\lambda_\delta$   such that for each  $\lambda \le \lambda_\delta$,
\begin{equation}
\label{eq:oct5a14}
\mu_+(E+\lambda F+\delta G)=\mu_+(PEP).
\end{equation}

 An analogous result holds for $\mu_-$.
\end{lemma}

\begin{proof}
 By restricting to the orthogonal complement of $\ker(G)$, we assume, without loss of generality, that
  $G\succ 0$.

Suppose first that
$$
\textup{dim}\,\ker(F)=k\quad\textrm{and}\quad k\ge 1.
$$
Then there exists a real unitary matrix
$$
U=\begin{bmatrix} U_1&U_2\end{bmatrix}\quad \textrm{with $U_1\in\RR^{n\times k}$ and $U_2\in\RR^{n\times(n-k)}$}
$$
such that
$$
FU_1=0\quad\textrm{and}\quad U_2^TFU_2\succ 0.
$$
Let
$$
E_{ij}=U_i^TEU_j,\quad F_{ij}=U_i^TFU_j\quad\textrm{and}\quad G_{ij}=U_i^TGU_j
$$
for $i,j=1,2$ and note that
$$
F_{11}=0_{k\times k},\quad F_{12}=0_{k\times(n-k)}\ \textrm{and}\ F_{21}=0_{(n-k)\times k}.
$$
 Thus,
\begin{align*}
\mu_+(E+\lambda F+\delta G)&=\mu_+(U^T[E+\lambda F+\delta G]U)\\
&=\mu_+(Q_{\lambda,\delta})
\end{align*}
with
$$
Q_{\lambda,\delta}=\begin{bmatrix}
E_{11}+\delta G_{11}&E_{12}+\delta G_{12}\\
E_{21}+\delta G_{21}&E_{22}+\lambda F_{22}+\delta G_{22}\end{bmatrix}.
$$
Since $G_{11}\succ 0$, the additive perturbation $\delta G_{11}$ with $\delta<0$  shifts the eigenvalues of $E_{11}$ to the left. However, if  $\delta\in[\varepsilon,0)$ and $\vert\varepsilon\vert$ is sufficiently small, then the positive eigenvalues of $E_{11}$ will stay positive, whereas, the  nonpositive  eigenvalues of $E_{11}$ become negative. Consequently,
$$
\mu_+(E_{11}+\delta G_{11})=\mu_+(E_{11})
\quad \textrm{and\quad $E_{11}+\delta G_{11}$\quad is invertible if $\delta\in[\varepsilon,0)$}.
$$
Thus, for such $\varepsilon$ and $\delta$,
$$
\mu_+(Q_{\lambda,\delta})=\mu_+(E_{11}+\delta G_{11})+
\mu_+(S_{\lambda,\delta})=\mu_+(E_{11})+
\mu_+(S_{\lambda,\delta}),
$$
where $S_{\lambda,\delta}$ denotes the Schur complement of $E_{11}+\delta G_{11}$ in $Q_{\lambda,\delta}$, i.e.,
$$
S_{\lambda,\delta}=E_{22}+\lambda F_{22}+\delta G_{22}
-(E_{12}+\delta G_{12})^T(E_{11}+\delta G_{11})^{-1}(E_{12}+\delta G_{12}).
$$
Therefore, since $F_{22}\succ 0$, there exists a number $\lambda_\delta<0$ such that $$
\mu_+(S_{\lambda,\delta})=0\quad\textrm{\quad for $\lambda\le\lambda_\delta$}.
$$
Thus, to this point we have established the equality
\begin{equation*}
\mu_+(E+\lambda F+\delta G)=\mu_+(U_1^T EU_1) \quad\textrm{when $\delta\in[\varepsilon,0)$,   $\lambda\le\lambda_\delta$}
\end{equation*}
and $k\ge 1$. To complete the proof in this case, note first that
$$
P=U_1U_1^T
$$
and hence, by  a well known theorem of Sylvester,
$$
\mu_+(PEP)=\mu_+(U_1U_1^TEU_1U_1^T)\le\mu_+(U_1^TEU_1)
$$
and, as $U_1^TU_1=I_k$,
$$
\mu_+(U_1^TEU_1)=\mu_+(U_1^TU_1U_1^TEU_1U_1^TU_1)\le
\mu_+(U_1U_1^TEU_1U_1^T)=\mu_+(PEP).
$$
The proof of \eqref{eq:oct5a14} is now complete when $k\ge 1$, i.e., when $\ker(F)\ne 0$.

However, if $\ker(F)=0$, then $P=0_{n\times n}$ and $F\succ 0$. Therefore, $\mu_+(PEP)=0$ and it is readily checked that for any $\delta\le 0$ there exists a $\lambda_\delta$ such that the equality in  \eqref{eq:oct5a14} holds for $\lambda\le \lambda_\delta$.
\end{proof}

 The proof of Lemma \ref{thm:signature-clamped-relaxed} employs a
 bilinear variant of the Hessian  that we now introduce.
  Let $p_{xx}(A,X)[h][0,K]$
  denote the matrix obtained by differentiating
  $p_x(A,X)[h]$ in the direction $(0,K)$ for
  $K\in\Smatng$; i.e.,
$$
  p^{\prime\prime}(A,X)[h][0,K]= \lim_{t\to 0} \frac{1}{t} \big ( p_x(A,X+tK)[h]
- p_x(A,X)[h]\big ).
$$
  In particular,
\begin{equation*}
  p_{xx}(A,X)[h]=p_{xx}(A,X)[h][h]
\end{equation*}
 and
\begin{equation}
 \label{eq:EHK}
p_{xx}(A,X)[h][0,K]=p_{xx}(A,X)[0,K][h]\,.
\end{equation}

\begin{proof}[Proof of Lemma \ref{thm:signature-clamped-relaxed}]
 Again, we shall only consider  the case of positive signatures, since the   case of negative signature is similar
  (and may be obtained by considering $-p$ in place of $p$).
 Recall   that $g$  denotes the number of noncommutative
 symmetric  variables $x$ and $d$ the degree of $p$ in $x$.
 Let $\mathcal H=\Smatng$
endowed with the  Hilbert Schmidt norm:
\begin{equation*}
\langle H, K\rangle_{\cH}=\textup{trace}\,K^TH=\sum_{j=1}^g\textup{trace}\,
K_jH_j\,.
\end{equation*}
By the Reisz representation theorem
any continuous symmetric bilinear form
 $\cQ: \mathcal H\times \mathcal H \to \mathbb R$
can be represented as
\[
\cQ(H,K)=     \langle R H,K \rangle_{\cH}
\]
where $R$ is a bounded  selfadjoint operator on $\cH$.

The mapping $\mathscr B:\mathcal H\times \mathcal H \to \mathbb R$
defined by
$$
  \mathscr B(H,K)=\langle p^{\prime\prime}(X)[H][K]v,v \rangle_{\RR^{\nk}}
$$
   is bi-linear and by Equation \eqref{eq:EHK} it is
   symmetric.
Thus, there is a bounded  selfadjoint
   operator $E$ on $\mathcal H$ so that
\begin{equation*}
   \langle E H, K\rangle_{\mathcal H} =
    \mathscr B(H,K).
\end{equation*}

    Similarly, there are bounded linear selfadjoint operators $F$ and
$G$ on $\cH$ so that
\begin{equation*}
 \begin{split}
    \langle F H,K \rangle_{\cH} = &
\langle p_x(A,X)[h]v,p_x(A,X)[0,K]v \rangle_{\RR^{\nk}} \\
    \langle G H,K \rangle_{\cH} = & %
       \sum_{j,w\in \cR \cC_p^j} \langle H_j w(A,X)v, K_j w(A,X)v\rangle \, .
 \end{split}
\end{equation*}
  In particular, $\langle G H,H\rangle =  \langle I_\kappa\otimes\msUalt(A,X)[H]v,
  I_\kappa\otimes\msUalt(A,X)[H]v\rangle.$
Thus,
$$
\langle p^{\prime\prime}_{\lambda, \delta}(A,X)[H]v,v\rangle_{\RR^{\nk}}=
\langle(E+\lambda F +\delta G)H, H\rangle_{\cH}
$$
 for  $H\in\smatng$.

Note that $H\in\ker(G)$ if and only if $H_j w(A,X)v=0$ for all $j$ and $w\in\cC_p^j$.
 Thus, by Lemma \ref{lem:I=0}, $\ker(G)\subset \ker(F)\cap \ker(E)$.
  Since the kernel of $F$ is exactly $\cT(A,X,v)$,
  an application of Lemma \ref{lem:sep1a14}  completes the proof.
\end{proof}


\section{Direct sums and linear independence}
\label{sec:DSLI}
Recall that $p(a,x)$ is a $\kk$ matrix-valued symmetric polynomial of degree $\widetilde{d}$ in $a$ and $d$ in $x$ and that $\cC_p$ is the  chip space of $p$. Let
  $\cC_p^{\kappa^\prime\times \kappa}$ denote the $\kappa^\prime\times \kappa$
  matrices with entries from $\cC_p$.

\subsection{Dominating points}
\label{sec:dom} \index{dominating point}
 A point $(\hA,\hX,\hv)\in\allmatvk$ is a
  \df{$\cC_p$-dominating point} for a  free set
  $\fnsS\subset (\allmatvk)_{n=1}^\infty$,
  if the set $W=  \{ (\hA,\hX,\hv)  \}$ is
   {$\cC_p$-dominating } for $\fnsS$, i.e.,
 if  $q\in\cC_p^{1\times\kappa}$ and  $q(\hA,\hX)\hv=0$  implies
 $q(A,X)v=0 $  for all $(A,X,v) \in \fnsS$.
  As before, note that $(\hA,\hX,\hv)$  is not required to be in $\fnsS$.

\begin{lemma}
 \label{lem:indDom}
   Suppose $p\in\cPk$ is symmetric nc matrix polynomial
    and $(\hA, \hX,\hv)$ is a $\cC_p$-dominating point for a free set $\fnsS\subset (\allmatvk)_{n=1}^\infty$.
   If $(B,Y,w)$ is a point in $\fnsS$ that is sufficiently
   close to $(\hA, \hX,\hv)$,  then it
  is also a $\cC_p$-dominating point for $\fnsS$.
\end{lemma}

\begin{proof}
  Let
$$
 \fnsS^o=  \{ r \in \cC_p^{1\times \kappa}:\  r(A,X)v=0, \mbox{ for all } (A,X,v) \in\fnsS  \}.
$$
 Choose a subspace $\mathcal{B}$  of $\cC_p^{1\times \kappa}$ that is complementary to $\fnsS^o$, i.e.,
  $$ \cC_p^{1\times\kappa} = \mathcal{B}  \ \dot{+} \ \fnsS^o,$$
  (where the symbol $\dot{+}$ means that $\cC_p^{1\times\kappa} = \mathcal{B}   + \fnsS^o$
   and  $\mathcal{B}  \cap\fnsS^o=\{0\}$) and  let $(\hA,\hX,\hv)$
   be a ${\cC}_p$-dominating point for $\fnsS$.
It is readily seen that  the linear mapping  $\mathcal{B}\to \mathbb R^n$ defined by
\[
  q \mapsto q(\hA,\hX)\hv
\]
  is one-one, because if $q$ and $q_0$ both belong to $\mathcal{B}$ and $q(\hA,\hX)\hv=q_0(\hA,\hX)\hv$,
  then $q-q_0\in\mathcal{B}  \cap\fnsS^o=\{0\}$.

   Hence, if $(B,Y,w)\in\fnsS$ is sufficiently close to $(\hA,\hX,\hv)$, then  the mapping
\begin{equation}
\label{eq:aug21a14}
 \mathcal{B} \ni q \mapsto q(B,Y)w
\end{equation}
 is also one-one.

 If $q\in{\cC}_p^{1\times\kappa}$, then $q=q_1+q_2$ with
 $q_1\in\fnsS^o$ and $q_2\in\mathcal{B} $. Thus,
 $$
 q(B,Y)w=q_1(B,Y)w+q_2(B,Y)w=q_2(B,Y)w,
 $$
 since $(B,Y,w)\in\fnsS$. Consequently,
 $$
 q(B,Y)w=0\Longrightarrow q_2(B,Y)w=0\Longrightarrow q_2=0,
 $$
 since the mapping \eqref{eq:aug21a14} is one-one from   $\mathcal{B}\to \mathbb R^n$.
 Thus,
 $$
 q(A,X)v=q_1(A,X)v=0\quad\textrm{for all
 $(A,X,v)\in\fnsS$},
 $$
 since $q_1\in\fnsS^\circ$. Therefore,
 $(B,Y,w)$ is a $\cC_p$ dominating point for $\fnsS$.
 \end{proof}

\begin{lemma}
 \label{lem:ind1}
   Let $p\in\cPk$ be a symmetric nc matrix polynomial of degree $\wtilde{d}$ in $a$ and degree $d$ in $x$ and let $\fns=(\fns(k))_{k=1}^\infty$  be a free nonempty set such that
\begin{equation*}
\fns(n)\subseteq  \{(A,X,v)\in\allmatvk: p(A,X)v=0\},
\end{equation*}
  then  for every  positive integer $N$ there
   exists an integer $n \ge N$ and  a triple
   $(\widehat{A},\widehat{X},\widehat{v})\in \fns(n)$ that is
  a $\cC_p$ dominating point  for $\fns.$
\end{lemma}

\begin{proof}
For a given $(A,X,v)$ in $ \fns$, define
\[
  \cI(A, X,v)=\{q \in \cC_p^{1\times \kappa} : q(A, X)v=0\}
\]
and
\[
 \cI(\fns)= \bigcap \{\cI(A, X,v): (A, X,v)\in
 \fns
 \}.
\]
Thus, $\cI(\fns)$ consists of all polynomials in ${\cC}_p^{1\times \kappa}$ that vanish on $\fns$.

Since $\cI(A, X,v)$ is a subspace of the
finite dimensional vector
  space $\cC_p^{1\times\kappa}$,  there is a $t \in \N$
  and triples
  $( A^j , X^j, v^j) \in \fns(n_j)$
  for $1\le j\le t$, such that
  \[ \cI(\fns)=
  \bigcap_{j=1}^t\{\cI(A^j, X^j,v^j):
  (A^j, X^j, v^j) \in  \fns  \}.
\]
  Thus, if $q\in  \cC_p^{1\times\kappa}$
  and $q(A^j, X^j)v^j=0$ for $j=1,\ldots,t$,
  then $q \in \cI(\fns)$ and hence  $q=0$ on  $\fns$, i.e., if
  $$
  \begin{bmatrix} q_{11}(A^j,X^j)v^j&\cdots&q_{1\kappa}(A^j,X^j)v^j\end{bmatrix}=0\quad
  \textrm{for} \ j=1,\ldots,t,
  $$
  then
  $$
  \begin{bmatrix} q_{11}(A,X)v&\cdots&q_{1\kappa}(A,X)v\end{bmatrix}=0\quad
  \textrm{for every point} \ (A,X,v)\in\fns.
  $$

Let
$$
A^\prime=\textup{diag}\{A^1,\ldots,A^t\}, \
X^\prime=\textup{diag}\{X^1,\ldots,X^t\} \ \textrm{and}\
v^\prime=\textup{col}(v^1,\ldots,v^t).
$$
 Then $A^\prime\in\cS_{n^\prime}(\RR^{\tg})$, $X^\prime\in\cS_{n^\prime}(\RR^{g})$ and $v^\prime\in\RR^{n^\prime\kappa}$,
where $n^\prime = n_1 + n_2 + \cdots + n_t$, and
$q(A^\prime,X^\prime)v^\prime=0$ if and only if
$q(A^j,X^j)v^j=0\ \textrm{for}\ j=1,\ldots,t.$
Thus, if $q\in  \cC_p^{1\times\kappa}$
  and $q(A^\prime, X^\prime)v^\prime=0$,
  then $q=0$ on $\fns$. If $n^\prime\ge N$,
  then the construction stops here.
  If not, then  choose a positive integer $k$ such that
  $n=kn^\prime \ge N$ and consider the $k$-fold direct sums
$$
\widehat{A}=\textup{diag}\{A^\prime,\ldots,A^\prime\}, \
\widehat{X}=\textup{diag}\{X^\prime,\ldots,X^\prime\}\quad\textrm{and}\ \widehat{v}=\textup{col}\{
v^\prime,\ldots,v^\prime\}.
$$
The triple $(\hA,\hX,\hv)$ is a ${\cC}_p$-dominating point for $\fns$: if $q\in  \cC_p^{1\times\kappa}$
  and $q(\widehat{A}, \widehat{X})\widehat{v}=0$,
  then $q=0$ on $\fns$.
\end{proof}


\section{The middle matrix representation  and its properties}
\label{sec:mm}
In this section we  develop a representation for
nc polynomials $q(a,x,h)$ in the nc variables $a=(a_1,\ldots,a_{\widetilde{g}})$, $x=(x_1,\ldots,x_g)$ and
$h=(h_1,\ldots,h_g)$ that are homogeneous of degree two in $h$ with particular attention given to the Hessian $p_{xx}(a,x)[h]$ of an nc polynomial $p(a,x)$ and two of its relatives.

\subsection{Definition of the middle matrix}
We begin with the case of scalar-valued polynomials, before turning to the matrix case with its additional bookkeeping overhead.

A \df{middle matrix representation} or \df{border vector-middle matrix representation} of a scalar nc polynomial $q(a,x,h)$ that is  homogeneous of degree two in $h$  is a representation of the form
\begin{equation}
\label{eq:mm}
q(a,x,h)=\sum_{i,j=0}^\ell B_i(a,x)[h]^TM_{ij}(a,x) B_j(a,x)[h]
\end{equation}
in which $B_j(a,x)[h]$ is a column  vector with entries of the form
$h_k w(a,x)$, where $w(a,x)$ is a word in $a,x$ that is
homogeneous of degree $j$ in $x$ and  $M_{ij}(a,x)$ is a  matrix polynomial in
$a$  and $x$ (and not $h$). A middle matrix representation for $q(a,x,h)=p_{xx}(a,x)[h]$, the Hessian of
the polynomial $p(a,x)=x_2^2ax_1+x_1ax_2^2+a^2,$ appears in Example \ref{ex:sep17a13}.
There $\ell=1$, $g=2$. The block entries (by degree in $x$) of the border vector
are  $B_0^T = \begin{bmatrix} h_1 & h_2 \end{bmatrix}$,
$B_1^T=\begin{bmatrix} x_2h_2 & x_1 a h_2\end{bmatrix}$, and
\[
 M_{00}(a,x) =\begin{bmatrix} 0 & ax_2\\ x_2 a & 0 \end{bmatrix}, \ \ M_{01}=M_{10} =\begin{bmatrix} a & 0\\0&1\end{bmatrix}, \ \ M_{11}=
 \begin{bmatrix}0 & 0\\0&0\end{bmatrix}.
\]

The middle matrix construction extends naturally to the case of matrix polynomials.
If $q$ is a $\kxk $ matrix valued polynomial, its \df{middle matrix representation}
has the more general but similar form
\[
q(a,x,h)
=\sum_{i.j=0}^{\ell}(I_\kappa\otimes B_i(a,x)[h]^T)
\MM_{ij}(a,x)
(I_\kappa\otimes B_j(a,x)[h]).
\]
Note that the middle matrix for  $q=C_w\, w(a,x,h)$ is simply $C_w\otimes M$, where $M$ is the middle matrix for $w$.

\subsubsection{Uniqueness of the middle matrix}
The middle matrix   depends upon $q$, but once the \df{border vector}  $B$ %
is fixed (so a choice of list of monomials, sorted by degree in $x$, is made), the \df{middle matrix} $M=(M_{ij})$ (resp. $\MM$)
is uniquely determined, justifying the terminology {\it the middle matrix}.  For instance, for a word $w=w(a,x,h)$ that is homogeneous of degree two in $h$,
\begin{equation}
\label{eq:mmlite}
w= w_L(a,x) h_i  \; c w_M(a,x)\;  h_j w_R(a,x),
\end{equation}
the border vector for any middle matrix representation of $w$ must include the words
$h_i w_L(a,x)^T$ and $h_j w_R(a,x)$.

To illustrate the extent of the uniqueness of the middle matrix and its  dependence on the border vector, consider
middle matrix representations for the (scalar polynomial) word $w$ of equation \eqref{eq:mmlite}.
Using a  border vector $B$ with just the two words $h_i w_L(a,x)^T$ and $h_j w_R(a,x)$,
the corresponding middle matrix { representation is
$$
w=\begin{bmatrix}w_L(a,x)h_i&w_R(a,x)^Th_j\end{bmatrix}\begin{bmatrix}0&1\\0&0\end{bmatrix}
\begin{bmatrix}h_iw_L(a,x)^T\\ h_jw_R(a,x)\end{bmatrix}.
$$
If the degrees of $w_L(a,x)$ and $w_R(a,x)$ in $x$ are the same, then there is there is a choice of order (a permutation) in constructing this $B$.
Of course, one could have chosen a border vector $B$ with more words. But then the middle matrix would have more zeros to ensure that the superfluous words are not counted, e.g.,
$$
w=\begin{bmatrix}w_L(a,x)h_i&w_R(a,x)^Th_j&*\end{bmatrix}\begin{bmatrix}0&1&0\\0&0&0\\0&0&0\end{bmatrix}
\begin{bmatrix}h_iw_L(a,x)^T\\ h_jw_R(a,x)\\ *\end{bmatrix}.
$$

\subsection{Middle matrix representations for Hessians}
Our ultimate goal is to describe the middle matrix representation
for the relaxed Hessian in sufficiently fine detail to make it a
powerful tool.  Throughout the remainder of this section $d$ and $\tilde{d}$ are fixed positive integers and our polynomials are assumed
to have degree at most $d$ in $x$ and $\tilde{d}$ in $a$.
Given a word $w=w(a,x)$,  let  $Z^w(a,x)$ \index{$Z^w$} denote the middle matrix of its
Hessian $w_{xx}$
\begin{equation}
\label{eq:Zhatw}
 w_{xx}(a,x) = \sum_{i,j=0}^{d-2} V_i(a,x)[h]^T \,
  Z_{ij}^w(a,x) V_j(a,x)[h],
\end{equation}
 based upon the \df{full border vector}
\begin{equation}
\label{eq:defV}
V(a,x)[h]= \begin{bmatrix} V_0 \\ \vdots \\ V_{d-2}\end{bmatrix},
\end{equation}
 where the $V_j$ lists all words $h_kf(a,x)$, for  $f(a,x)$ of the form
\begin{equation}
\label{eq:defgiantf}
 f= u_0(a)x_{i_1} u_1(a)x_{i_2}\cdots u_{j-1}(a) x_{i_j} u_j(a),
\end{equation}
 and the $u_j$ are words of length at most $\tilde{d}$.
This choice of border vector works over all such choices of $w$.

Given   $C_w\in\RR^{\kappa\times\kappa}$,
\begin{equation*}
\begin{split}
   C_w\, & w_{xx}(a,x)[h] \\ %
 & = \sum_{i,j=0}^{d-2} (I_\kappa \otimes V_i(a,x)[h]^T)\,
 (C_w\otimes Z_{ij}^w(a,x))\,(I_{\kappa}\otimes V_j(a,x)[h]).
\end{split}
\end{equation*}
In particular, for $p=C_w\,  w$, the entries of its middle matrix based on $V$ are $\ZZ_{ij} = C_w\otimes Z_{ij}^w$.
(We use $\ZZ_{ij}(a,x)$ \index{$\ZZ$} to denote the ($\kappa\times\kappa$ block) entries of the middle matrix for matrix-valued polynomials
and $Z_{ij}(a,x)$ \index{$Z$} in the case of scalar-valued polynomials.)

Given a polynomial $p$ expressed as in equation \eqref{pww}, its Hessian has the middle matrix representation %
\begin{equation}
 \label{eq:defpxx}
\begin{split}
 &p_{xx}(a,x)[h] \\
   &=\sum_{i,j=0}^{d-2} (I_\kappa \otimes V_i(a,x)[h]^T)\,
 \big (\sum_{w} C_w\otimes Z_{ij}^w(a,x)\big )\,(I_{\kappa}\otimes V_j(a,x)[h]).
\end{split}
\end{equation}
Thus the middle matrix $\ZZ$ of the Hessian of $p$ based on the border vector $V$ has (block) entries $\ZZ_{ij} = \sum_w C_w\otimes Z_{ij}^w$.

\subsection{Middle matrices for the modified Hessian}
\label{sec:mid-mod-hess}
The \df{modified Hessian} of $p$ is, by definition,
\begin{equation*}
   p_{xx}(a,x)[h] + \lambda p_x(a,x)[h]^T p_x(a,x)[h].
\end{equation*}
The middle matrix of $p_x(a,x)[h]^T p_x(a,x)[h]$ is obtained by expressing the derivative $p_x(a,x)[h]$
in terms of the \df{extended full border vector} \index{$\wtilde{V}$}
\begin{equation}
\label{eq:deftV}
 \wtilde{V}(a,x)[h]= \begin{bmatrix} V_0(a,x)[h] \\ \vdots \\ V_{d-1}(a,x)[h]\end{bmatrix}.
\end{equation}
If $p=\sum_{w\in\cW} C_w\otimes w(a,x)$ is a sum of words with coefficients $C_w\in\RR^{\kappa\times\kappa}$, then
$$
p_x(a,x)[h]=\sum_{j=0}^{d-1} \Phi_j(a,x)   \IktV(a,x)[h]),
$$
where
\[
 \Phi(a,x)= \begin{bmatrix} \Phi_0(a,x) & \dots & \Phi_{d-1}(a,x)\end{bmatrix},
\]
is  a block row matrix with $\kappa$ rows and
\begin{equation}
\label{eq:may9a17}
\IktV(a,x)[h]= \begin{bmatrix}I_\kappa\otimes  V_0(a,x)[h] \\ \vdots \\I_\kappa\otimes  V_{d-1}(a,x)[h]\end{bmatrix}.
\end{equation}
This notation yields the formula
\begin{equation}
\label{eq:gradsquare}
p_x(a,x)[h]^T p_x(a,x)[h]=
\IktV(a,x)[h]^T \big ( \Phi(a,x)^T\Phi(a,x)\big ) \IktV(a,x)[h]
\end{equation}
and thus the middle matrix, in block form,  is $(\Phi_k(a,x)\Phi_j(a,x))_{j,k}$.

Putting things together gives the middle matrix representation
\begin{equation}
\label{def:modHess}
\begin{split}
 p_{xx}(a,x)[h] &+ \lambda p_x(a,x)[h]^T p_x(a,x)[h] \\
   = & \sum_{i,j=0}^{d-1} (I_\kappa\otimes V_i(a,x)[h])^T (\ZZ_\lambda)_{ij}(a,x) (I_\kappa \otimes V_j(a,x)[h]),
\end{split}
\end{equation}
where
\[
 (\ZZ_{\lambda})_{ij}(a,x) = \begin{cases} \ZZ_{ij}(a,x) + \lambda\, \Phi_i(a,x)^T \Phi_j(a,x)  \ \  & i,j\le d-2 \\
 \lambda\,   \Phi_i(a,x) ^T\Phi_j \ \ & \mbox{otherwise}\ \ \end{cases}
\]
and the  $\ZZ_{ij}(a,x)$  are the block entries of the middle matrix of the Hessian of $p$ with respect to the full border vector $\wtilde{V}$.  As with the Hessian, the middle matrix for the modified Hessian is uniquely determined by the choice of border vector and is symmetric if $p$ is.
\index{$\ZZ_\lambda$}

Finally, the \df{middle matrix for the relaxed Hessian}, denoted $\ZZ_{\lambda,\delta}$ (the notational conflict between $\ZZ_{ij}$, the $(i,j)$ block entry of the middle matrix $\ZZ$ of the Hessian of $p$,  and $\ZZ_{\lambda,\delta}$ should cause no confusion) is obtained from the middle matrix for the relaxed Hessian by simply adding $\delta I$,
\[
 \ZZ_{\lambda,\delta} = \ZZ_{\lambda} +\delta I.
\]

\subsection{The reduced border vector}
\label{sec:reduce}
The middle matrix for the modified and relaxed Hessians based upon the full border vector $\wtilde{V}$ has rows and columns of zeros corresponding to words that do not appear in the right chip set of $p$. Let $\msU_j(a,x)[h]$ denote the (column) vector of  words $h_k f(a,x)$ for $1\le k\le g$ and $f$ in the right chip set of $p$ of degree $j$ in $x$ and let
\begin{equation*}
   \msUalt(a,x)[h] =\begin{bmatrix} \msU_0 \\ \vdots \\ \msU_{d-1} \end{bmatrix}.
\end{equation*}
We refer to $\msUalt$ also as the  \df{reduced border vector}.\footnote{Really it is {\it a} reduced border vector as any two reduced border vectors are related by a permutation. The vector obtained by excluding the words of degree $d-1$ in $x$ is also called the reduced border vector.} (Compare with equation \eqref{eq:defrelax}.)  Indeed, $\msUalt$ includes only those words needed to construct a middle matrix-border vector representation for the modified and relaxed Hessians of $p$.  The middle matrix, still denoted $\ZZ$, for the reduced border vector is obtained from the middle matrix associated to $V$ by removing rows and columns of zeros - and applying a permutation if needed.  For instance, the border vector for the border vector middle matrix representation appearing in Example \eqref{ex:sep17a13} is the reduced border vector based on the right chip set of $p(a,x)=x_2^2ax_1+x_1ax_2^2+a^2.$

If the aim is to construct the most parsimonious middle matrix representation for the Hessian of $p$, then one is led to use a border vector built from the \df{secondary right chip set} of  $p$ i.e., the set of words
$v(a,x)$ that appear to the right of at least two $x$'s in a word
 \[
  w(a,x) = w_L(a,x)x_j w_M(a,x)x_\ell v(a,x).
 \]
 that appears in $p$. Thus, for example, the secondary right chip set of the word $w(a,x)=a_1x_1^2a_2x)2x_1a_1$ is the set of words $\{a_1,x_1a_1,a_2x_2x_1a_1\}$.
 As a mnemonic, the chip set is associated to the first derivative and the secondary chip set the second derivative.

The proof of the following proposition is immediate from the preceding discussion.

\begin{proposition}
 \label{prop:mminpractice}
  If $\ZZ,$ $\ZZ_\lambda$, $\ZZ_{\lambda,\delta}$  and $\hZZ$, $\hZZ_\lambda,$ $\hZZ_{\lambda,\delta}$ are the  middle matrices for the Hessian, modified Hessian and relaxed Hessian  of $p$ based upon two different border vectors  and if $(A,X)\in\allmat$, then
  $\ZZ(A,X)$ and $\hZZ(A,X)$ (resp.,    $\ZZ_{\lambda}(A,X)$ and $\hZZ_{\lambda}(A,X)$; $\ZZ_{\lambda,\delta}(A,X)$ and $\hZZ_{\lambda,\delta}(A,X)$) have the same number of (strictly) positive and (strictly) negative eigenvalues.  Thus,  $\ZZ(A,X)$ (resp. $\ZZ_\lambda(A,X)$, $\ZZ_{\lambda,\delta}(A,X)$)  is positive semidefinite if and only if $\hZZ(A,X)$ (respectively $\hZZ_\lambda(A,X)$, $\hZZ_{\lambda,\delta}(A,X)$) is.
\end{proposition}

\begin{remark}\rm
\label{rem:anymiddleinstorm}
In view of Proposition \ref{prop:mminpractice},  we often do not make a notational distinction between choices of the middle matrix based on different choices of border vector. In what follows, typically we prove results first for the full border vector(s) where the bookkeeping is more easily automated and then establish the result for other choices, most notably the reduced border vector(s).
\end{remark}


\subsection{An example of a middle matrix representation for the modified Hessian}
Suppose that $C\in\RR^{\kappa\times\kappa}$ and let
$$
p(a,x)=C\, a_1x_1a_2x_2^2+C^T\, x_2^2a_2x_1a_1.
$$
Its right chip set is given by
$$
\cR\cC_p^1=\{a_1,a_2x_2^2\}\quad\textrm{and}\quad \cR\cC_p^2=\{1,x_2,a_2x_1a_1,x_2a_2x_1a_1\}.
$$
and thus in this case
\[
\begin{split}
\msUalt(a,x)[h]=&\textup{col}
(h_1a_1,h_2,h_2x_2,h_2a_2x_1a_1,h_1a_2x_2^2,h_2x_2a_2x_1a_1),\\
 \msU(a,x)[h]=&\textup{col}
(h_1a_1,h_2,h_2x_2,h_2a_2x_1a_1),
\end{split}
\]
are  the reduced border vectors.

By direct computation,
\begin{align*}
p_x(a,x)[h]&=C\otimes\{a_1h_1a_2x_2^2+a_1x_1a_2h_2x_2+a_1x_1a_2x_2h_2\}\\
&+C^T\otimes\{h_2x_2a_2x_1a_1+x_2h_2a_2x_1a_1+x_2^2a_2h_1a_1\}\\
&=(C\otimes a_1)(I_\kappa\otimes h_1a_2x_2^2)+(C\otimes a_1x_1a_2)(I_\kappa\otimes h_2x_2)\\ &+(C\otimes a_1x_1a_2x_2)(I_\kappa\otimes h_2)+(C^T\otimes 1)(I_\kappa\otimes h_2x_2a_2x_1a_1)\\ &+(C^T\otimes x_2)((I_\kappa\otimes h_2a_2x_2a_1)+(C^T\otimes x_2^2a_2)(I_\kappa\otimes h_1a_1)
\end{align*}
which is of the form
$$
p_x(a,x)[h]=\begin{bmatrix}G_1(a,x)&\cdots&G_6(a,x) \end{bmatrix}
\, (I_\kappa\otimes\msUalt(a,x)[h]).
$$
Correspondingly $p_x(A,X)[h]$  can be written as
\begin{equation}\label{eq:Phipx}
\begin{split}
& \ \ \ \  p_x(A,X)[h]\\
&=\begin{bmatrix}I_\kappa\otimes I_n&I_\kappa\otimes A_1&I_\kappa\otimes X_2&I_\kappa\otimes A_1X_1A_2&I_\kappa\otimes X_2^2A_2&I_\kappa\otimes A_1X_1A_2X_2\end{bmatrix}\\
&\begin{bmatrix}0&0&0&0&0&C^T\otimes I_n\\ 0&0&0&0&C\otimes I_n&0\\ 0&0&0&C^T\otimes I_n&0&0\\  0&0&C\otimes I_n&0&0&0\\
0&C^T\otimes I_n&0&0&0&0\\ C\otimes I_n&0&0&0&0&0\end{bmatrix}
\begin{bmatrix}I_\kappa\otimes H_2\\I_\kappa\otimes  H_1A_1\\  I_\kappa\otimes H_2X_2\\ I_\kappa\otimes H_2A_2X_1A_1\\
I_\kappa\otimes H_1A_2X_2^2\\ I_\kappa\otimes H_2X_2A_2X_1A_1\end{bmatrix}
\end{split}
\end{equation}
and
\begin{equation}\label{eq:anM}
\begin{split} & \ \ \ \ \ p_{xx}(A,X)[h]=
2\begin{bmatrix}I_\kappa\otimes H_2&I_\kappa\otimes A_1H_1
&I_\kappa\otimes X_2H_2&I_\kappa\otimes A_1X_1A_2H_2\end{bmatrix}\\
&\begin{bmatrix}0&C^T\otimes X_2A_2&0&C^T\otimes I_n\\C\otimes A_2X_2&0&C\otimes A_2&0\\0&C^T\otimes A_2&0&0\\
C\otimes I_n
&0&0&0\end{bmatrix}
\begin{bmatrix}I_\kappa\otimes H_2\\ I_\kappa\otimes H_1A_1\\ I_\kappa\otimes H_2X_2\\ I_\kappa\otimes H_2A_2X_1A_1\end{bmatrix}.
\end{split}
\end{equation}
 Thus the  middle matrix for the Hessian of $p$ (based on the reduced border vector)
is the square matrix in formula \eqref{eq:anM}. Likewise, the  middle matrix  for the modified Hessian of $p$ corresponding to $\msUalt$ is
\[
\begin{bmatrix}0&C^T\otimes X_2A_2&0&C^T\otimes I_n&0&0\\C\otimes A_2X_2&0&C\otimes A_2
&0&0&0\\0&C^T\otimes A_2&0&0&0&0\\C\otimes I_n
&0&0&0&0&0\\0&0&0&0&0&0\\ 0&0&0&0&0&0\end{bmatrix} + \lambda \Phi^T \Phi,
\]
where $\Phi$ is the product of the   first two matrices on the right hand side of formula\eqref{eq:Phipx}.
\qed

\subsection{A Structure theorem for the middle matrix of the Hessian}
In this subsection, we state our main result describing the middle matrix for the Hessian and modified Hessian needed for the proof of Theorem \ref{thm:main}.

\begin{theorem}
\label{thm:dec29a13}
Suppose $p$ is a  $\kappa\times\kappa$ matrix polynomial %
and let $\ZZ$ and $\ZZ_\lambda$ denote the middle matrices of the Hessian of $p$ and the modified Hessian of $p$  with respect to (the same) fixed choice of border vector. Given   $(A,X)\in\SS_n(\RR^{\fg})$,  there exists an invertible matrix $S$ such that for all $\lambda\in \RR$,
\[
\ZZ_{\lambda}(A,X)=\lambda \alpha+S^T\beta S
\]
where
\[
\alpha=\begin{bmatrix}U(I-\Pi_{\cR_{W^TW}})U^T&0\\0&0\end{bmatrix} \quad \textrm{and}\quad \beta=\begin{bmatrix}\ZZ(A,0)&0\\0&\lambda WW^T\end{bmatrix}
\]
and $\Pi_{\cR{W^TW}}$ is the projection onto the range of $W^TW$.
In particular, if the highest degree terms of $p$ majorize at $A$, then range of $U^T$ is contained in the range of $W^T$ and therefore $\ZZ_\lambda(A,X)$ is similar to
\[
\begin{bmatrix}\ZZ(A,0)&0\\0&\lambda WW^T\end{bmatrix}.
\]
\end{theorem}

\begin{proof}
The proof of this theorem occupies Subsection \ref{sec:proof-dec29a13}.
\end{proof}

Since it is of independent interest, illustrates both the structure of the middle matrix a and provides the opportunity to introduce notations used in the remainder of this article, we conclude this section with the statement of the following scalar version of a theorem from Section \ref{sec:proofs}. It is a variation on Theorem \ref{lem:aug2a15}. To state the result, let $k_b$ denote  the number of words in $a$ of length at most $\tilde{d}$,  $\ft_j$ is the number of terms (words) in $V_j$ (namely $g$ times the number of words of the form    \eqref{eq:defgiantf}), and  $\ft$ is the number of words of the form   \eqref{eq:defgiantf} when $j=d$.   Thus,
\begin{equation}
\label{eq:aug20b13}
k_b=1+\tg+\cdots +\tg^{\td}, \quad \ft_j =(k_bg)^{j+1}, \quad \ft=k_b\ft_{d-1}.
\end{equation}
We note that $\ft_{j}\ft_{\ell}=\ft_{j+\ell+1}$.

Let $K_j = \mathfrak{c} \otimes I_{\ft_j},$ where $\mathfrak{c}$ denotes the column vector $(x_i u(a))$ %
parametrized over $1\le i\le g$ and $u\in\cU$, the collection of words in $a$ of length at most $\tilde{d}.$  Thus the size of $K_j$ is $\ft_{t_{j+1}}\times \ft_j$.

\begin{equation}
\label{eq:oct3b13}
\msC(a,x)=
\left[\begin{array}{ccccc}
I_{\ft_0}&0&\cdots&0&0\\
-K_0&I_{\ft_1}&\cdots&0&0\\
0&-K_1&\cdots&0&0\\
\vdots&\vdots& &\vdots&\vdots\\
0&0&\cdots&I_{\ft_{d-3}}&0\\
0&0&\cdots&-K_{d-3}&I_{\ft_{d-2}}\end{array}
\right],
\end{equation}

\begin{theorem}
\label{thm:oct1a13}
Let $\ell=d-2$. The middle matrix of the Hessian  $p_{xx}(a,x)[h] =\sum_{i,j=0}^{d-2}V_i(a,x)[h]^TZ_{ij}(a,x)V_j(a,x)[h]$ of $p$  is of the form
\begin{align*}
Z(a,x)& =
\left[\begin{array}{ccccc}
Z_{00}(a,x)&Z_{01}(a,x)&\cdots&  Z_{0,\ell-1}(a,x)&Z_{0\ell}(a,0)\\
Z_{10}(a,x)&Z_{11}(a,x)&\cdots&Z_{1,\ell-1}(a,0)&0\\
\vdots && & &\vdots\\
Z_{\ell-1,0}(a,x)&Z_{\ell-1,1}(a,0)&\cdots &0 & 0\\
Z_{\ell 0}(a,0)&0&\cdots &0&0
\end{array}
\right]\\ \\
&=Z(a,0)\msC(a,x)^{-1}.
\end{align*}
\end{theorem}

From Theorem \ref{thm:oct1a13} it is evident that definiteness (either positive or negative) of the middle matrix of a Hessian imposes serious restrictions on
the middle matrix that we will later exploit.

\subsection{The polynomial congruence for the middle matrix of the Hessian}
The proof of Theorem \ref{thm:dec29a13}  depends essentially upon the following congruence result, which also plays an essential role in the proof of Theorem \ref{thm:main}.

Given a matrix nc polynomial $Y=(Y_{ij})_{i,j=1}^{s,t}$ where the $Y_{i,j}$ are $n_j\times m_j$ matrices and a positive integer $\kappa$, consistent
with the usage in equation \eqref{eq:may9a17},
\[
 \Yk = \begin{pmatrix} I_\kappa\otimes Y_{ij} \end{pmatrix}_{i,j=1}^{s,t}.
\]
We note that $I_\kappa\otimes Y$ differs from $\Yk$ by a (block matrix) permutation as $I_\kappa\otimes Y$ is the block $\kappa\times \kappa$ diagonal block matrixwith diagonal entries $Y$, whereas $\Yk$ is a block $s\times t$ matrix
with $\kappa n_j\times \kappa m_j$ entries $I_\kappa \otimes Y_{ij}$.   Moreover, in the case of the full border vectors
$V$ and $\wtilde{V}$, the block matrices $V^\kappa_{\degg}(a,x)[h]$ and $\IktV(a,x)[h]$ are sorted by (increasing) degree in $h$ (whereas
$I_\kappa\otimes \wtilde{V}$ is the block diagonal $\kappa\times \kappa$ matrix with $\wtilde{V}$ as the diagonal entries).  For instance,
in the case $d=4$ so that $d-2=2$,
\[
 I_2 \otimes V = \begin{pmatrix} \begin{pmatrix} V_0 \\ V_1 \\ V_2 \end{pmatrix} & 0 \\ 0 & \begin{pmatrix} V_0 \\ V_1 \\ V_2 \end{pmatrix}\end{pmatrix},
\]
and
\[
 V^2 = \begin{pmatrix} I_2\otimes V_0 \\ I_2\otimes V_1 \\ I_2\otimes V_2 \end{pmatrix}
 =\begin{pmatrix} \begin{pmatrix} V_0 & 0 \\ 0 & V_0 \end{pmatrix} \\ \begin{pmatrix} V_1 & 0 \\ 0 & V_1 \end{pmatrix} \\
 \begin{pmatrix} V_2 & 0 \\ 0 & V_2 \end{pmatrix} \end{pmatrix}.
\]

\begin{theorem}
\label{lem:aug2a15}
Fix a collection of words $\cW$ of degree at most $d-2$ in $x$ and $\tilde{d}$ in $a$. Let $B$ denote a border vector determined by $\cW$; that is,  $B_j$,  $j=0,\dots,d-2$, the $j$-th block entry of $B$ lists all words of the form $h_k f(a,x)$ for $1\le k\le g$ and $f\in \cW$ of degree $j$ in $x$.
 Let $\nu_j$ denote the length of $B_j$.

  There exists a  matrix polynomial $Y(a,x)=(Y_{ij})_{i,j=0}^{d-2}$ with entries $Y_{ij}$ of size $\nu_i\times \nu_j$ such that
\begin{enumerate}[\rm(i)]
 \item \label{it:Y1} $Y_{ii}=I_{\nu_i}$;
 \item \label{it:Y2} $Y_{ij}=0$ for $i<j$;
 \item \label{it:Y3} $Y$ is invertible;
 \item \label{it:Y4} $Y^{-1}$ is a polynomial;  and
 \item \label{it:Y5}   if $p\in\cP^{\kappa\times\kappa}$ is a nc matrix polynomial whose right chip set is a subset of $\cW$, then the middle matrix $\ZZ$ for the Hessian $p_{xx}$ based upon $B$ satisfies
\[
 \ZZ(a,x) = \Yk(a,x)^T \ZZ(a,0) \, \Yk(a,x);
\]
 i.e., $\ZZ(a,x)$ is polynomially congruent to $\ZZ(a,0)$ via a polynomial that depends only upon the choices $\cW$ and $B$.
\end{enumerate}
\end{theorem}

\begin{proof}
Items \eqref{it:Y3} and \eqref{it:Y4}  follow immediately from items \eqref{it:Y1} and \eqref{it:Y2}. The construction of $Y$ and the proof of the remaining items is carried out in Section \ref{sec:proofs} and concludes in Subsection \ref{sec:polycon}.
\end{proof}

Theorem \ref{lem:aug2a15} and Theorem \ref{thm:dec29a13} are needed for the proof of Theorem \ref{thm:main}. Their proofs proceed by direct calculation and appear in Section \ref{sec:proofs}. The reader who is willing to accept the theorems in this section can skip to Section \ref{sec:CHSY}.

\section{The proofs of Theorems  \ref{lem:aug2a15} and  \ref{thm:dec29a13}}
\label{sec:proofs}

This section begins by introducing a Kronecker product formalism for free polynomials. (See Subsection \ref{sec:identities}.)  In terms of this formalism, clean and useful formulas for derivatives and Hessians and their middle matrix representations are developed in Subsections \ref{sec:nckron} and \ref{sec:mmhess}.  These formulas are applied in Subsections \ref{sec:polycon} and \ref{sec:proof-dec29a13} to prove Theorems \ref{lem:aug2a15} and \ref{thm:dec29a13} respectively.

\subsection{Extending the Kronecker product}
 \label{sec:identities}
 Recall the Kronecker product of matrices $A$ and $B$ with real entries is, by definition, \index{kronecker product} \index{$A\otimes B$}
\begin{equation}
\label{eq:kron}
 A\otimes B=\begin{bmatrix}a_{11}B&a_{12}B&\cdots &a_{1q}B\\ \vdots& & &\vdots\\ a_{p1}B&a_{p2}B&\cdots&a_{pq}B\end{bmatrix}.
\end{equation}
 In particular, the Kronecker product is a coordinate dependent concrete interpretation of the abstract construction of the tensor product of matrices.
 It will be convenient to {\it extend the Kronecker product to allow one or both of the matrices to have free polynomials as entries.}
 A number of the resulting  identities that will play a central role in later calculations are collected in this section for easy future access.

 If $a$ is a free polynomial and $B=(b_{jk})$ is a matrix whose entries are free polynomials, interpret $aB$ as the  matrix $(a b_{jk})$.  With this convention the Kronecker product extends, via equation \eqref{eq:kron} to matrices whose entries are free polynomials.

Because of the lack of commutativity, the Kronecker product of matrices of polynomials does not enjoy all the properties of the Kronecker product of ordinary matrices. For instance,  if $C\in\RR^{p\times q}$,  $D\in\RR^{q\times r}$ and
$Y$ and $Z$ are matrix polynomials of sizes $s\times t$ and $t\times u$, respectively, then
\begin{equation}
\label{eq:aug6a13}
(C\otimes Y)(D\otimes Z)=CD\otimes YZ
\end{equation}
and
\begin{equation*}
(C\otimes Y)^T=C^T\otimes Y^T.
\end{equation*}
These identities fail, however, if $C$, $D$, $Y$ and $Z$ are all arrays of nc variables. Indeed, for free scalar polynomials $c,y,d,z$, the polynomials $(c\otimes y)(d\otimes z)$ is $cydz$, but the polynomial $cd\otimes yz$ is $cdyz$.

As for the transpose, consider the example\footnote{The exposition here is for the case of symmetric variables.}
\begin{equation*}
\left(\begin{bmatrix}y_1\\ y_2\end{bmatrix}\otimes \begin{bmatrix}z_1\\ z_2\\ z_3\end{bmatrix}\right)^T
=\begin{bmatrix}z_1^T y_1^T&z_2^Ty_1^T&z_3^Ty_1^T&z_1^Ty_2^T&z_2^Ty_2^T&z_3^Ty_2^T\end{bmatrix},
\end{equation*}
whereas
\[
 \begin{bmatrix}y_1\\ y_2\end{bmatrix}^T \otimes \begin{bmatrix}z_1\\ z_2\\ z_3\end{bmatrix}^T
  = \begin{bmatrix} y_1^T z_1^T & y_1^T z_2^T & y_1^T z_3^T & y_2^T z_1^T & y_2^T z_2^T & y_2^T z_3^T \end{bmatrix}.
\]
On the other hand,
\begin{equation}
 \label{eq:permuteT}
\left(\begin{bmatrix}y_1\\ y_2\end{bmatrix}\otimes \begin{bmatrix}z_1\\ z_2\\ z_3\end{bmatrix}\right)^T=
\left(\begin{bmatrix}z_1\\ z_2\\ z_3\end{bmatrix}^T\otimes\begin{bmatrix}y_1\\ y_2\end{bmatrix}^T\right)\Pi
\end{equation}
for a suitably chosen $6\times 6$ permutation matrix $\Pi$.

Two other useful formulas for nc polynomials $y_1,\ldots,y_s,z_1,\ldots,z_t$ are:
\begin{equation*}
\begin{bmatrix}y_1&\cdots&y_s\end{bmatrix}\otimes\begin{bmatrix}z_1&\cdots&z_r\end{bmatrix}
=\begin{bmatrix}y_1&\cdots&y_s\end{bmatrix}\left(I_s\otimes
\begin{bmatrix}z_1&\cdots&z_r\end{bmatrix}\right)
\end{equation*}
and
\begin{equation}
\label{eq:aug9b13}
\begin{bmatrix}y_1\\ \vdots\\ y_s\end{bmatrix}\otimes\begin{bmatrix}z_1\\ \vdots \\ z_r\end{bmatrix}
=\left(\begin{bmatrix}y_1\\ \vdots  \\  y_s\end{bmatrix}\otimes I_r\right)
\begin{bmatrix}z_1\\ \vdots\\ z_r\end{bmatrix}.
\end{equation}

If $j$, $k$, $\ell$ and $m$ are positive integers,  $u_1,\dots,u_j,y_1,\dots,y_\ell$ are nc polynomials and $c_1,\dots,c_k$ are real numbers, then
\begin{equation}
\label{eq:aug22a13}
 \left(\begin{bmatrix}u_1\\ \vdots\\ u_j\end{bmatrix}\otimes I_{m\ell}\right)
\left(\begin{bmatrix}y_1\\ \vdots\\ y_\ell\end{bmatrix}\otimes I_m\right)=\begin{bmatrix}u_1\\ \vdots\\ u_j\end{bmatrix}\otimes
\begin{bmatrix}y_1\\ \vdots\\ y_\ell\end{bmatrix}\otimes I_m
\end{equation}
and
\begin{equation}
\label{eq:aug24a13}
\begin{bmatrix}c_1&\cdots&c_{m\ell}\end{bmatrix}\left(\begin{bmatrix}
y_1\\ \vdots \\ y_\ell\end{bmatrix}\otimes I_m\right)=
\begin{bmatrix}
y_1& \cdots & y_\ell\end{bmatrix}\,\begin{bmatrix}c_1&\cdots&c_m\\c_{m+1}&\cdots&c_{2m}\\
\vdots& &\vdots\\
c_v&\cdots&c_{m \ell}\end{bmatrix}
\end{equation}
with $v=(\ell -1)m+1$. Given the column vector $y$ as above, let
\[
 \row(y) = \begin{bmatrix} y_1& \cdots & y_\ell\end{bmatrix}
\]
and, given positive integers $m$ and $\ell$, let $\fL(\ell,m;\cdot)$ denote the linear map from row vectors
 $r=\textup{row}( r_1, \cdots,  r_\ell )$ with components $r_j$ that are row vectors of length $m$,
to matrices of size $\ell\times m$
\begin{equation}
\label{eq:deffL}
 \fL(\ell,m;r)  = \begin{bmatrix} r_1 \\ \vdots \\ r_\ell \end{bmatrix}.
\end{equation}

 With these notations, formula \eqref{eq:aug24a13}  can be expressed in terms of the row vector
$ c=\begin{bmatrix}c_1&\cdots&c_{m\ell}\end{bmatrix}$ as
\begin{equation}
\label{eq:moveinside}
c\,\left(\begin{bmatrix}
y_1\\ \vdots \\ y_\ell\end{bmatrix}\otimes I_m\right)= \row(y)  \fL(\ell,m;c).
\end{equation}

As a special case of the identity of formula \eqref{eq:aug6a13}, if $C\in \RR^{\kappa\times\kappa}$ and $X$ is an $n\times n$ array of nc polynomials, then
\begin{equation*}
(C\otimes I_n)(I_\kappa\otimes X)=(I_\kappa\otimes X)(
C\otimes I_n)=C\otimes X.
\end{equation*}

If $X_{ij}$ are compatibly sized arrays of nc polynomials, then
\begin{equation*}
I_\kappa\otimes\begin{bmatrix}X_{11}&X_{12}&X_{13}\\
X_{21}&X_{22}&X_{23}\end{bmatrix}=\Pi
\begin{bmatrix}I_\kappa\otimes X_{11}&I_\kappa\otimes X_{12}
&I_\kappa\otimes X_{13}\\
I_\kappa\otimes X_{21}&I_\kappa\otimes X_{22}&I_\kappa\otimes X_{23}\end{bmatrix}\Pi^\prime
\end{equation*}
for suitably chosen permutations $\Pi$ and $\Pi^\prime$ (that serve to interchange block rows and block columns, respectively).

\subsection{NC polynomials in Kronecker notation}
\label{sec:nckron}

For a column vector $y=\textup{col}\,( y_1, \dots,  y_k)$, let
\[
[y]_0=\begin{bmatrix}
y_1\\ \vdots \\ y_k\end{bmatrix}_0=1\quad\textrm{and}\quad
[y]_j= \begin{bmatrix}
y_1\\ \vdots \\ y_k\end{bmatrix}_j=\begin{bmatrix}
y_1\\ \vdots \\ y_k\end{bmatrix}\otimes \cdots\otimes
\begin{bmatrix}
y_1\\ \vdots \\ y_k\end{bmatrix}.
\]
Thus $[y]_j$
is a $j$-fold product for $j=1,2,\ldots$. The notation reflects the fact that the product(s) are associative. In particular
$[y]_j\otimes [y]_k =[y]_{j+k}$.

 Let
\begin{equation*}
 x=\begin{bmatrix}x_1\\ \vdots \\ x_g\end{bmatrix},\quad
b=\textup{col}\left\{1, \begin{bmatrix}a_1\\ \vdots \\ a_{\tg}\end{bmatrix},
\begin{bmatrix}a_1\\ \vdots \\ a_{\tg}\end{bmatrix}_2,\ldots,\begin{bmatrix}a_1\\ \vdots \\ a_{\tg}\end{bmatrix}_{\widetilde d}\right\}.
\end{equation*}
 In particular, $(x\otimes b)_{j+1}$ has length $\ft_j$, where $\ft_j$ (and also $\ft$ appearing below) is defined in equation \eqref{eq:aug20b13}.
If $p(a,x)=p^d(a,x)$ %
is an nc polynomial of degree at most $\td$ in $a$ and homogeneous of degree $d$ in $x$, then  it admits a representation of the form  \index{$\crep$}
\begin{equation}
\label{eq:aug19a13}
p^d(a,x)=
\crep^d_p (b\otimes (x\otimes b)_d) = \begin{bmatrix}c_1&\cdots &c_{\ft}\end{bmatrix}\, \big (b\otimes(x\otimes b)_d \big ),
\end{equation}
where $\crep^d_p$ is a row vector of length $\ft=k_b (k_b g)^d$ (the number of words (entries) in $b\otimes (x\otimes b)_d$).

Letting
\begin{equation*}
   \varphi^d_p(a) = \crep^d_p (b\otimes I_{\ft_{d-1}}) =  \begin{bmatrix}c_1&\cdots&c_{\ft}\end{bmatrix}\, (b\otimes I_{\ft_{d-1}})
\end{equation*}
(a polynomial in $a$ alone),   formula \eqref{eq:aug19a13} can be re-expressed as
\begin{equation}
\label{eq:aug8a14}
p^d(a,x)=\varphi^d_{p}(a)(x\otimes b)_d.
\end{equation}

Let
$$
h=\textup{col}(h_1,\ldots, h_g).
$$
With this notation, the vector $V_j$ (the portion of the full border vector homogeneous of degree $j$ in $x$) is given by
\begin{equation}
\label{eq:aug20a13}
V_j=V_j(a,x)[h]=h\otimes b\otimes (x\otimes b)_j. %
\end{equation}
 Correspondingly, from the representation \eqref{eq:aug19a13} and the product rule,
\begin{equation}
\label{eq:px}
\begin{split}
 p_x^d(a,x)[h] = & \crep^d_p \left ( \sum_{j=0}^{d-1}  b\otimes (x\otimes b)_{d-1-j} \otimes V_j(a,x)[h] \right )\\
 = & \varphi^d_p(a) \left ( \sum_{j=0}^{d-1}   (x\otimes b)_{d-1-j} \otimes V_j(a,x)[h] \right )\\
\end{split}
\end{equation}

\subsection{Computing the middle matrix of the Hessian}
In this section we shall present a transparent example that exhibits the main features of the calculation of the middle matrix representation of
$p_{xx}(a,x)[h]$.
Suppose $p(a,x)$ is an nc polynomial
that is homogeneous of degree $d$ in $x$
and that
$p$  is expressed in the form  \eqref{eq:aug19a13} (see also \eqref{eq:aug8a14}).


\begin{example}\rm
\label{ex:jul24a15}
Suppose that $c\in\RR$ and
$$
p(a,x)=c\, a_1x_1a_2x_2^2.
$$
Since $p$ is a homogeneous polynomial of degree  $d=3$ in $x$ and is of degree
$\wt{d}=2$ in $a$ and there are no consecutive strings of $a$, it suffices to choose\footnote{Here, because
there are no consecutive strings of $a$ of length $>1$, it is enough to use  the full border vector
based on words in $a$ of degree at most one, rather than that based on words in $a$ of degree at most two.}
\begin{equation*}
b=\begin{bmatrix}1\\ a_1\\ a_2\end{bmatrix} \quad\textrm{and}\quad x=\begin{bmatrix}x_1\\ x_2\end{bmatrix}.
\end{equation*}
In this case $k_b$, the number of entries in $b$, is equal to $3$ and  $p(a,x)$ is $c$  times one
 of the entries in the vector polynomial $b\otimes (x\otimes b)_3$ of height $\ft=k_b(k_bg)^d=3(6^3)=648$,
i.e.,
$$
p(a,x)=\begin{bmatrix}c_1 &\cdots&c_t\end{bmatrix}b\otimes (x\otimes b)_3,
$$
where one of the coefficients $c_1,\ldots,c_t\in\RR$ is equal to $c$ and the remaining $t-1$
coefficients are equal to zero. It is readily checked that
\begin{equation*}
\begin{split}
p_x(a,x)[h]&=\begin{bmatrix}c_1 &\cdots&c_t\end{bmatrix} \,
b\otimes \{(h\otimes b)\otimes
(x\otimes b)_2\\
&+(x\otimes b)\otimes(h\otimes b)\otimes (x\otimes b)+
(x\otimes b)_2\otimes(h\otimes b)\}
\end{split}
\end{equation*}
and
\begin{equation*}
\begin{split}
p_{xx}(a,x)[h]&=2\begin{bmatrix}c_1 &\cdots&c_t\end{bmatrix}b\otimes \{(h\otimes b)\otimes
(h\otimes b)\otimes(x\otimes b)\\
&+(x\otimes b)\otimes(h\otimes b)\otimes (h\otimes b)+
(h\otimes b)\otimes(x\otimes b)\otimes(h\otimes b)\}.
\end{split}
\end{equation*}
Thus, in terms of the notation
$$
V_0=h\otimes b\quad\textrm{and}\quad V_j=(h\otimes b)\otimes (x\otimes b)_j\quad\textrm{for $j=1,2,\ldots $},
$$
\begin{equation*}
\begin{split}
p_x(a,x)[h]&=\begin{bmatrix}c_1 &\cdots&c_t\end{bmatrix}\\
&[b\otimes \left\{V_2
+(x\otimes b)\otimes V_1+
(x\otimes b)_2\otimes V_0\right\}]
\end{split}
\end{equation*}
and
\begin{equation}
\label{eq:jul24b15}
\begin{split}
p_{xx}(a,x)[h]&=2\begin{bmatrix}c_1 &\cdots&c_t\end{bmatrix}\\
&[b\otimes \left\{V_0\otimes
 V_1+(x\otimes b)\otimes V_0\otimes V_0+V_1\otimes V_0\right\}].
 \end{split}
\end{equation}

The first step in the computation of the blocks $Z_{ij}(a,x)$ in the middle matrix representation
\eqref{eq:mm}
of the Hessian is to  invoke formula \eqref{eq:aug9b13} in order to re-express the term
\[
b\otimes \left\{V_0\otimes
 V_1+(x\otimes b)\otimes V_0\otimes V_0+V_1\otimes V_0\right\}.
\]
in formula \eqref{eq:jul24b15} as
\[
[b\otimes V_0\otimes I_{\ft_1}]\,
 V_1+[b\otimes x\otimes b\otimes V_0\otimes  I_{\ft_0}]\,V_0+[V_1\otimes I_{\ft_0}]\,V_0.
\]
The contributions of each of these three summands will be evaluated separately:
\bigskip

\noindent
{\bf 1.} {\it Verify the formula
\begin{equation}
\label{eq:jul24d15}
2\begin{bmatrix}c_1 &\cdots&c_t\end{bmatrix}\,[b\otimes V_0\otimes I_{\ft_1}]=V_0^TZ_{01}
\end{equation}
where
\begin{equation*}
Z_{01}=2
\Pi_0\begin{bmatrix}c_1&\cdots &c_{108}\\
c_{109}&\cdots &c_{216}\\
\vdots& &\vdots\\
c_{541}&\cdots&c_{648}\end{bmatrix} (b\otimes I_{\ft_1}),
\end{equation*}
$\Pi_0\in\RR^{6\times 6}$ is the permutation matrix defined by $b^T\otimes h^T=V_0^T\Pi_0$,  $t=648$ and $\ft_1=36$}.
\bigskip

To verify \eqref{eq:jul24d15}, note first that in view of formula \eqref{eq:aug9b13},
\begin{equation*}
b\otimes V_0\otimes I_{\ft_1}=b\otimes h\otimes b\otimes I_{\ft_1}=
(b\otimes h\otimes  I_{k_b\ft_1})\,(b\otimes I_{\ft_1}).
\end{equation*}
To ease the notation,  set
$$
b\otimes h=\begin{bmatrix}u_1\\ \vdots\\ u_{k_bg}\end{bmatrix}=\begin{bmatrix}u_1\\ \vdots\\ u_6\end{bmatrix},
$$
where, in the case at hand, $\ft=3\times 6^3$, $\ft_1=6^2$ and $s=\ft/(k_bg)=k_b\ft_1=108$. Since $c_iu_j=u_jc_i$,
\begin{align*}
&\begin{bmatrix}c_1 &\cdots&c_t\end{bmatrix}\,
(b\otimes h \otimes I_s)=\begin{bmatrix}c_1 &\cdots&c_t\end{bmatrix}\,\begin{bmatrix}u_1 I_s\\ \vdots \\u_{k_bg} I_s\end{bmatrix}\\
&=\begin{bmatrix}c_1 &\cdots&c_{648}\end{bmatrix}\,\begin{bmatrix}u_1 I_{108}\\ \vdots \\u_6 I_{108}\end{bmatrix}
=\begin{bmatrix}u_1&\cdots&u_6\end{bmatrix}\begin{bmatrix}c_1&\cdots &c_{108}\\
c_{109}&\cdots &c_{216}\\
\vdots& &\vdots\\
c_{541}&\cdots&c_{648}\end{bmatrix}   \\
&=b^T\otimes h^T\,
\begin{bmatrix}c_1&\cdots &c_{108}\\
c_{109}&\cdots &c_{216}\\
\vdots& &\vdots\\
c_{541}&\cdots&c_{648}\end{bmatrix}  
=V_0^T\Pi_0\begin{bmatrix}c_1&\cdots &c_{108}\\
c_{109}&\cdots &c_{216}\\
\vdots& &\vdots\\
c_{541}&\cdots&c_{648}\end{bmatrix},
\end{align*}
which, upon combining terms, leads easily to \eqref{eq:jul24d15}.
\bigskip

\noindent
{\bf 2.} {\it Verify the formula}
\begin{equation}
\label{eq:jul24e15}
2\begin{bmatrix}c_1 &\cdots&c_t\end{bmatrix}\,
b\otimes ((x\otimes b)\otimes V_0\otimes  I_{\ft_0})=
V_1^TZ_{10},
\end{equation}
where
\begin{equation*}
Z_{10}=2\Pi_1\begin{bmatrix}c_1&\cdots &c_{18}\\
c_{19}&\cdots &c_{36}\\
\vdots& &\vdots\\
c_{631}&\cdots&c_{648}\end{bmatrix} (b\otimes I_{\ft_0}),
\end{equation*}
$\Pi_1\in\RR^{36\times 36}$ is the permutation matrix defined by the formula
$$
b^T\otimes x^T\otimes b^T\otimes h^T=V_1^T\Pi_1
$$
and $\ft_0=6$.
\bigskip

To verify \eqref{eq:jul24e15},  first invoke \eqref{eq:aug9b13} to obtain
$$
b\otimes x\otimes b\otimes V_0\otimes  I_{\ft_0}=[(b\otimes x)\otimes(b\otimes h)\otimes I_{k_b\ft_0}]\,
[b\otimes I_{\ft_0}]
$$
and set
$$
(b\otimes x)\otimes(b\otimes h)=
\begin{bmatrix}u_1\\ \vdots\\ u_{(k_bg)^2}\end{bmatrix}=\begin{bmatrix}u_1\\ \vdots\\ u_{36}\end{bmatrix},
$$
where again $\ft=3\times 6^3$, $\ft_1=6^2$ and $s=\ft/(k_bg)^2=k_b\ft_0=18$. Then, since $c_iu_j=u_jc_i$,
\[
\begin{bmatrix}c_1&\cdots&c_{648}\end{bmatrix}\,\begin{bmatrix}u_1\\ \vdots\\ u_{36}\end{bmatrix}\otimes I_{18}=\begin{bmatrix}u_1&\cdots&u_{36}\end{bmatrix}\,\begin{bmatrix}c_1&\cdots&c_{18}\\c_{19}&\cdots&c_{36}\\  \vdots& &\vdots\\ c_{631}&\cdots&c_{648}\end{bmatrix}.
\]
The verification of \eqref{eq:jul24e15} is completed by noting that
\[
\begin{bmatrix}u_1 &\cdots& u_{36}\end{bmatrix}=b^T\otimes x^T\otimes b^T\otimes h^T=V_1^T \Pi_1
\]
and combining formulas.
\bigskip

\noindent
{\bf 3.} {\it Verify the formula}
\begin{equation}
\label{eq:jul26d15}
2\begin{bmatrix}c_1&\cdots&c_t\end{bmatrix}\,
[b\otimes V_1\otimes I_{\ft_0}]=V_0^TZ_{00},
\end{equation}
where
\begin{equation*}
Z_{00}=2\Pi_0 \begin{bmatrix}c_1 &\cdots&c_{648}\end{bmatrix}\,\begin{bmatrix}u_1 I_{108}\\ \vdots \\u_6 I_{108}\end{bmatrix}\, (b\otimes x\otimes b\otimes I_6).
\end{equation*}
\bigskip

The verification of formula \eqref{eq:jul26d15} is similar to the verification of \eqref{eq:jul24d15}. The  main new ingredients are the formulas
\begin{align*}
b\otimes V_1\otimes I_{\ft_0}&=(b\otimes h\otimes b\otimes x\otimes b)\otimes I_{\ft_0}\\
&=(b\otimes h\otimes I_{k_b\ft_1})(b\otimes x\otimes b \otimes I_{\ft_0})
\end{align*}
and, upon setting
$$
\begin{bmatrix}u_1\\ \vdots \\u_6\end{bmatrix}=b\otimes h,
$$
\begin{align*}
\begin{bmatrix}c_1&\cdots&c_t\end{bmatrix}\,\begin{bmatrix}u_1 I_{108}\\
\vdots\\ u_6  I_{108}\end{bmatrix}&=\begin{bmatrix}c_1&\cdots&c_{648}\end{bmatrix}\,\begin{bmatrix}u_1 I_{108}\\
\vdots\\ u_6  I_{108}\end{bmatrix}\\
&=V_0^T\Pi_0\begin{bmatrix}c_1&\cdots&c_{108}\\
\vdots& &\vdots\\
c_{541}&\cdots &c_{648}\end{bmatrix},
\end{align*}
which lead easily to
  \eqref{eq:jul26d15}.
\end{example}

\subsection{Formulas for the middle matrix of the Hessian of  a homogeneous $p$}
\label{sec:mmhess}
This section is devoted to the  detailed computation of the middle matrix $(Z_{ij})$ of the Hessian (see equation \eqref{eq:defpxx})
of a scalar nc polynomial $p(a,x)$ that is homogeneous of degree $\ell$ in $x$.

In view of formulas \eqref{eq:aug20a13} and  \eqref{eq:permuteT},
\begin{equation}
 \label{eq:ViT}
\begin{split}
  \row(b\otimes (x\otimes b)_i \otimes h) = &b^T\otimes (x^T\otimes b^T)_i \otimes h^T \Pi_i\\
    = & ((h\otimes b)\otimes (x\otimes b)_i)^T\Pi_i \\
    = & V_i(a,x)[h]^T \Pi_i,
\end{split}
\end{equation}
for a suitably chosen permutation matrix $\Pi_i$ of size $\ft_i$.

If $p$ is homogeneous of degree $\ell$ in $x$, then the constant row vector $\crep^\ell_p$ defined in formula  \eqref{eq:aug19a13}  has length $k_b \ft_{\ell-1}= k_b \ft_{i-1} \ft_{\ell-i}$.
Moreover, in terms of the notation $\ff^j=(x\otimes b)_j$, \index{$\ff^j$}
\[
 p_{xx}(a,x)[h] = 2 \crep^\ell_p \, \left ( b\otimes \big (\sum_{i+j+k=\ell-2} \ff^i \otimes (h\otimes b) \otimes \ff^k  \otimes (h\otimes b) \otimes \ff^j \big ) \right ).
\]
Fix $i,j$ and $k=\ell-2-i-j$, set $s=k_b \ft_{\ell-i-2}$ and compute, using $\fL$ as defined in equation \eqref{eq:deffL} and  formulas
\eqref{eq:ViT}, \eqref{eq:moveinside} and \eqref{eq:ViT} in that order,
\[
 \begin{split}
 \crep^\ell_p & \left ( b\otimes  \ff^i \otimes (h\otimes b) \otimes \ff^k  \otimes (h\otimes b) \otimes \ff^j \right )\\
 = & \, \crep^\ell_p \big ( (b\otimes  \ff^i \otimes h) \otimes I_{k_b \ft_{\ell-2-i}} \big )\, ((b \otimes \ff^k)\otimes I_{\ft_j})\,  \big ((h\otimes b) \otimes \ff^j \big )\\
 = & \, \row(b\otimes  \ff^i \otimes h) \, \fL(\ft_i,s;\crep^\ell_p) \, ((b \otimes \ff^k)\otimes I_{\ft_j})\,  \big ((h\otimes b) \otimes \ff^j \big )\\
 = & \, V_i(a,x)[h]^T  \, \left [ \big ( \Pi_i \, \fL(\ft_i,s;\crep^\ell_p) \big) \, \big ( (b \otimes \ff^k)\otimes I_{\ft_j} \big) \right ] \, V_j(a,x)[h].
\end{split}
\]
Hence, the $(i,j)$ block entry of the middle matrix of $p$ is given by
\[
Z_{i,j}= 2 \big ( \Pi_i \, \fL(\ft_i,s;\crep^\ell_p) \big) \, \big ( (b \otimes \ff^{\ell-2-i-j})\otimes I_{\ft_j} \big).
\]
On the other hand, the $(i,j+1)$ block is (so long as $\ell-2-i-j-1\ge 0$)
\[
Z_{i,j+1} = 2 \big ( \Pi_i \, \fL(\ft_i,s;\crep^\ell_p) \big) \, \big ( (b \otimes \ff^{\ell-3-i-j})\otimes I_{\ft_{j+1}} \big).
\]
Using formula  \eqref{eq:aug22a13} and the relation $\ft_{j+1} =\ft_j (gk_b)$ to justify the identity
\[
\left ((b\otimes \ff^{\ell-3-i-j})\otimes I_{\ft_{j+1}}\right )\,((x\otimes b) \otimes I_{\ft_j})
 = b\otimes \ff^{\ell-2-i-j}\otimes I_{\ft_j},
\]
it now follows that
\begin{equation}
\label{eq:aug23a13}
Z_{i,j}(a,x) = Z_{i,j+1}(a,x) \left ( (x\otimes b)\otimes I_{\ft_j}\right ) = Z_{i,j+1}(a,x) K_j(a,x).
\end{equation}
when $i+j<\ell-2$, where
 \begin{equation*}
 \label{eq:defK}
   K_j(a,x) = (x\otimes b)\otimes I_{\ft_j}
\end{equation*}
is a matrix polynomial of   size $\ft_{j+1}\times \ft_j$.

The following proposition summarizes and expands a bit on this discussion.

\begin{proposition}
\label{prop:aug22a13}
If $\ell \ge 2$,  then
\begin{equation*}
Z_{ij}(a,x)=\left\{\begin{array}{ll}
0&\textrm{if}\ i+j>\ell-2\\
\vspace{2mm}
 { Z_{ij}(a,0) } &\textrm{if}\ i+j=\ell-2   \\
\vspace{2mm}
Z_{i,j+1}(a,x) K_j(a,x) &\textrm{if}\ i+j<\ell-2
\end{array}\right.
\end{equation*}
Moreover,
if $i+j=\ell-2$, then
\begin{equation*}
Z_{ij}(a,0)=
2\Pi_i  \fL(\ft_i,s;\crep^\ell_p)
 (b\otimes I_{\ft_j} ) =2\Pi _i\begin{bmatrix} c_1&\cdots &c_s \\
c_{s+1}&\cdots &c_{2s}\\
\vdots& &\vdots\\
 c_u&\cdots&c_t
\end{bmatrix} (b\otimes I_{\ft_j} )
\end{equation*}
with $u=t-s+1$ and
$s=k_b\ft_{\ell-2-i}$.
\end{proposition}

 The formulas in Proposition \ref{prop:aug22a13} can be expressed conveniently in terms of the matrix
$K(a,x)$ defined in \eqref{eq:oct3b13}
as
\begin{equation*}
Z(a,x)\,K(a,x)=Z(a,0),
\end{equation*}
where $Z_{ij}(a,0)=0$ if $i+j\ne\ell-2$.

If $i+j<\ell-2$, then the formula for $Z_{ij}(a,x)$ in Proposition \ref{prop:aug22a13} may be iterated to obtain
\begin{align*}
 Z_{ij}(a,x)&=Z_{i,j+1}(a,x)K_j(a,x)=Z_{i,j+2}(a,x)K_{j+1}(a,x)K_j(a,x)\\
 &=\cdots=Z_{i,\ell-2-i}(a,0)K_{\ell-1-i}(a,x)\cdots K_j(a,x).
 \end{align*}
 These formulas can be expressed  in terms of the   block matrix $L(a,x)= (L_{ij}(a,x))_{i,j=0}^{d-2}$  with entries
\begin{equation*}
\index{$L_{ij}$}
L_{ij}=\begin{cases} I_{\ft_j}&\quad\textrm{if}\ i=j\\ K_{i-1}\cdots K_j = (x\otimes b)_{i-j} \otimes I_{\ft_j} &\quad \textrm{if}\ i>j \\ 0_{\ft_i\times \ft_j} & \quad \textrm{if}\ i<j
\end{cases}.
\end{equation*}
of  size $\ft_{i}\otimes \ft_j$.  In fact
\begin{equation}
\label{eq:mar10b17}
L(a,x)\,K(a,x)=I.
\end{equation}
In terms of this notation, formula \eqref{eq:px} can be expressed as
$$
p_x^d(a,x)[h]=\varphi^d_p(a) \left ( \sum_{j=0}^{d-1} L_{d-1,j}(a,x) \, V_j(a,x)[h]\right ).
  $$

Let  $(\Kk)_{ij} = I_\kappa\otimes K_{ij}$.

\begin{theorem}
\label{thm:h-tell-all}
 Suppose $d\ge \ell \ge 2$. If  $p$ is a  $\kappa\times\kappa$  nc polynomial homogeneous of degree $\ell$ in $x$ and if  $\ZZ$ is the middle matrix of $p$ with respect to the full border vector of formula \eqref{eq:defV}, then
\[
 \ZZ(a,x) \Kk(a,x) = \ZZ(a,0),
\]
 where $\msC(a,x)$ is defined in formula \eqref{eq:oct3b13}.
\end{theorem}

\begin{proof}
 Observe that since $\msC(a,x)$ depends only upon the degree $d$ of the polynomial $p$ in $x$ and not upon the words in $p$, it suffices to prove the result for $p=C\otimes w$, for $C$ a $\kappa\times\kappa$ matrix and $w$ a word of length $\ell \le d$ in $x$.  Using the notation introduced in equation \eqref{eq:Zhatw}, the entries of the middle matrix of such a $p$ have the form
\[
\ZZ_{ij} = C\otimes Z^w_{ij}.
\]
 Now apply the recursion \eqref{eq:aug23a13} and the relation $Z_{i,\ell-i-2}(a,x)=Z_{i,\ell-i-2}(a,0)$ from Proposition \ref{prop:aug22a13} to conclude that the result holds for $C\otimes w$.
\end{proof}

\begin{remark}\rm
 \label{rem:h-tell-all}
   Theorem \ref{thm:h-tell-all} is the matrix version of Theorem \ref{thm:oct1a13}. Note too that $\msC(a,x)$ is lower triangular with the identity on the diagonal and depends only upon $d$ (and not on $p$).
\end{remark}

\subsection{Beyond homogeneous}
In this section we  illustrate how to extend Theorem \ref{thm:h-tell-all} to polynomials that are not necessarily homogeneous.
The main points and structure involved in this process are present in the scalar case ($\kappa=1$). Accordingly suppose
 $p(a,x)$ is a polynomial that is homogeneous of degree $5$ in $x$.
Proposition \ref{prop:aug22a13} implies that the entries $Z_{ij}(a,x)$ in the middle matrix
$Z(a,x)$ in formula (\ref{eq:defpxx}) with $i+j\le 2$ can be expressed as
\begin{align*}
Z_{00}&=Z_{03}K_2K_1K_0,\ Z_{01}=Z_{03}K_2K_1,\ Z_{02}=Z_{03}K_2, \\
 Z_{10}&=Z_{12}K_1K_0,\
Z_{11}=Z_{12}K_1,\ Z_{20}=Z_{21}K_0.
\end{align*}
Equivalently,
\begin{equation*}
Z(a,x)=Z(a,0)\,\left[\begin{array}{cccc}
I_{\ft_0}&0&0&0\\
K_0&I_{\ft_1}&0&0\\
K_1K_0&K_1&I_{\ft_2}&0\\
K_2K_1K_0&K_2K_1&K_2&I_{\ft_3}\end{array}\right],
\end{equation*}
with $K_j=K_j(a,x)$ as in \eqref{eq:defK} and
\begin{equation*}
Z(a,0)=\begin{bmatrix}
0&0&0&Z_{03}(a,0)\\
0&0&Z_{12}(a,0)&0\\
0&Z_{21}(a,0)&0&0\\
Z_{30}(a,0)&0&0&0\end{bmatrix}.
\end{equation*}
Moreover,
\[
\begin{split}
L(a,x)=&\left[\begin{array}{cccc}
I_{\ft_0}&0&0&0\\
K_0&I_{\ft_1}&0&0\\
K_1K_0&K_1&I_{\ft_2}&0\\
K_2K_1K_0&K_2K_1&K_2&I_{\ft_3}\end{array}\right]\\
=&
\left[\begin{array}{cccc}
I_{\ft_0}&0&0&0\\
-K_0&I_{\ft_1}&0&0\\
0&-K_1&I_{\ft_2}&0\\
0&0&-K_2&I_{\ft_3}\end{array}\right]^{-1}=\msC(a,x)^{-1}.
\end{split}
\]

\qed

%


Next we indicate how the
 basic recursion formulas (\ref{eq:aug23a13}) are extended to the middle matrix $Z(a,x)$ in the representation of the Hessian $p_{xx}(a,x)[h]$ of polynomials $p(a,x)$ not constrained to be homogeneous in $x.$  We suppose now that $p$ has degree four and express it as a sum of its homogeneous of degree $0\le j\le 4$ in $x$ parts,
\begin{equation*}
p(a,x)=\sum_{j=0}^4p^j(a,x).
\end{equation*}
In particular,
\begin{equation}
\label{eq:oct1b13}
p^j\quad\textrm{is a linear combination of the $k_b(k_bg)^j$ entries in}\ b\otimes (x\otimes b)_j.
\end{equation}

Next apply Proposition \ref{prop:aug22a13} to each term $p^j(a,x)$ separately. The condition (\ref{eq:oct1b13}) insures that the same border vectors intervene in the representation of the Hessian $p_{xx}^j$ for each choice of $j$ and hence that the entries $Z_{ij}(a,x)$ in the middle matrix satisfy the recursion (\ref{eq:aug23a13}) of equation \eqref{eq:aug23a13}.  In particular,
\begin{align*}
p_{xx}(a,x)[h]&=p_{xx}^2(a,x)[h]+p_{xx}^3(a,x)[h]
+p_{xx}^4(a,x)[h]\\
&=V_0^TZ_{00}^2V_0+\begin{bmatrix}V_0^T&V_1^T\end{bmatrix}\begin{bmatrix}Z_{00}^3&Z_{01}^3\\
Z_{10}^3&0\end{bmatrix}\begin{bmatrix}V_0\\ V_1\end{bmatrix}\\ &+
\begin{bmatrix}V_0^T&V_1^T&V_2^T\end{bmatrix}
\begin{bmatrix}Z_{00}^4&Z_{01}^4&Z_{02}^4\\
Z_{10}^4&Z_{11}^4&0\\
Z_{20}^4&0&0\end{bmatrix}\begin{bmatrix}V_0\\ V_1\\ V_2\end{bmatrix}\\
&=\begin{bmatrix}V_0^T&V_1^T&V_2^T\end{bmatrix}\left\{ \cdots\right\}
\begin{bmatrix}V_0\\ V_1\\ V_2\end{bmatrix},
\end{align*}
where the term in curly brackets is
\[
\left\{\cdots\right\}=\begin{bmatrix}Z_{00}^2&0&0\\
0&0&0\\0&0&0\end{bmatrix}+\begin{bmatrix}Z_{00}^3&Z_{01}^3&0\\
Z_{10}^3&0&0\\
0&0&0\end{bmatrix}+\begin{bmatrix}Z_{00}^4&Z_{01}^4&Z_{02}^4\\
Z_{10}^4&Z_{11}^4&0\\
Z_{20}^4&0&0\end{bmatrix}.
\]
On the other hand, the recursion \eqref{eq:aug23a13} gives
\[
\left\{ \cdots\right\} \, \begin{bmatrix}I&0&0\\-K_0&I&0\\0&-K_1&I\end{bmatrix}^{-1}
= \begin{bmatrix}Z_{00}^2&Z_{01}^3&Z_{02}^4\\
Z_{10}^3&Z_{11}^4&0\\
Z_{20}^4&0&0\end{bmatrix} = Z(a,0).
\]

\subsection{Polynomial congruence}
\label{sec:polycon}
 A pair of $p\times p$ polynomial matrices $F(a,x)$ and $G(a,x)$ are
{\bf polynomially congruent},  denoted $F\sim_p G$,  if there exists a $p\times p$ polynomial matrix $R(a,x)$ such that $R(a,x)^{-1}$ is a polynomial matrix and
\[
F(a,x)=R(a,x)G(a,x)R(a,x)^T.
\]

\begin{lemma}
\label{lem:jul29a15}
Suppose $X$ is an $n\times n$ matrix-valued free polynomial  and $X^k=0$ for some positive integer $k$.
\begin{enumerate}[\rm(1)]
\item
The matrix polynomial
\begin{equation*}
Y= I_n+ \sum_{j=1}^{k-1}  {{\frac12}\choose{j}} X^j
\end{equation*}
has a polynomial inverse.
\item \quad   $Y^2=(I_n+X)$.
\item
 If $M$ is also a $n\times n$ matrix-valued free polynomial and $M=M^T$ and $MX=X^TM$, then $MY=Y^TM$.
\item
 If $N=XM$ is symmetric ($N=N^T)$, then  $MY^2 = Y^T M Y$. In particular, $M$ and $N$ are congruent.
\end{enumerate}
\end{lemma}

\begin{proof}
That $Y^2=(I_n+X)$ follows from  Newton's generalized binomial theorem (and the fact that $X^k=0$).
Since $X$ is nilpotent, $(I_n+X)$ has a polynomial inverse
and thus $Y(I+X)^{-1}$ is a polynomial inverse for $Y$.

If $MX=X^TM$, then
$$
MY= M(I_n+ \sum_{j=1}^{k-1}  {{\frac12}\choose{j}} X^j )
=(I_n+ \sum_{j=1}^{k-1}  {{\frac12}\choose{j}} X^j ) M=Y^T M
$$
and hence $MY^2=Y^TMY.$
\end{proof}

The next result is adapted from \cite{DHM07a}.

\begin{theorem}
\label{thm:ZAiscZa}
If $K(a,x)$ is defined by formula \eqref{eq:oct3b13}, then  $N(a,x)=\msC(a,x)-I$ is nilpotent of order $d-1$ and  the matrix polynomial
\begin{equation*}
  Y(a,x):= I+\sum_{j=1}^{d-2} \binom{\frac12}{j} N(a,x)^j
\end{equation*}
satisfies the conditions of Theorem \ref{lem:aug2a15} based upon the full border vector $V$.  In particular $Y$ is invertible and its inverse is a polynomial.

Moreover, if $p$ is a symmetric $\kappa\times \kappa$ matrix polynomial of degree at most $d$ in $x$ and $\ZZ$ is the middle matrix of its Hessian, then
\begin{equation*}
  (\Yk(a,x))^T \ZZ(a,x) (\Yk(a,x)) = \ZZ(a,0)
\end{equation*}
\end{theorem}

\begin{proof}
We prove the case $\kappa=1$. Since $p$ is assumed symmetric  $Z(a,x)=Z(a,x)^T$. Hence, using Theorem \ref{thm:oct1a13} (see also Theorem \ref{thm:h-tell-all}),
\begin{equation*}
   Z(a,x)\msC(a,x)=Z(a,0)=Z(a,0)^T=\msC(a,x)^TZ(a,x).
\end{equation*}
Therefore,
\begin{equation*}
     Z(a,x)N(a,x)=N(a,x)^TZ(a,x).
\end{equation*}
Consequently,  by Lemma \ref{lem:jul29a15}, $Y^2(a,x)=I+N(a,x)=K(a,x)$,
\[
 Z(a,0)=Z(a,x)Y(a,x)^2 = Y(a,x)^T Z(a,x)Y(a,x)
\]
and $Y$ has a polynomial inverse.
\end{proof}

\begin{remark}\rm
 Theorem \ref{thm:ZAiscZa} covers Theorem \ref{lem:aug2a15} for the case of the full border vector $V$.
\end{remark}

\begin{proof}[Proof of Theorem \ref{lem:aug2a15}]
Let $V$ denote the full border vector and
\[
\begin{split}
 p_{xx}(a,x)[h]= & \Vk(a,x)[h]^T \ZZ(a,x) \Vk(a,x)\\
 = & \sum_{j,k} (I_\kappa\otimes V_j(a,x)[h]^T) \ZZ_{j,k}(a,x) (I_\kappa\otimes V_k(a,x)[h])
\end{split}
\]
the border-middle matrix representation for the Hessian of $p$.
Fix a set of words $\cW$ and a corresponding border vector $B$ with (block) entry $B_j$ consisting of precisely
 of all the words of the form {$h_i f(a,x)$ for $1\le i\le g$} and $f\in \cW$ of degree $j$ in $x$. 
For each $k$ there is a natural inclusion map $\iota_k$ from the words of degree $k$ in $x$ in $\cW$ to
the words of degree $k$ in $x$ in the full border vector $V_k$ such that
\[
 B_k(a,x)=\Pi_k\iota_k^T V(a,x), 
\]
 where $\Pi_k$ is a permutation matrix.  Thus,
writing $V_j$ for $V_j(a,x)[h]$ and
letting $\QQ_k=\Pi_k\iota_k^T$, 
\[
\begin{split}
p_{xx}(a,x)[h]& = \sum_{j,k} (I_\kappa\otimes V_j^T)(I_\kappa\otimes \iota_j\iota_j^T) \ZZ_{j,k}(a,x) (I_\kappa\otimes \iota_k \iota_k^T)
(I_\kappa\otimes V_k) \\
  = & \sum_{j,k} (I_\kappa\otimes V_j^T)(I_\kappa\otimes \QQ_j^T \QQ_j) \ZZ_{j,k}(a,x) (I_\kappa\otimes  \QQ_k^T\QQ_k) (I_\kappa\otimes V_k)\\
 = &  \sum_{j,k} (I_\kappa\otimes V_j^T \QQ_j^T) (I_\kappa\otimes \QQ_j) \ZZ_{j,k}(a,x) (I_\kappa\otimes \QQ_k^T)
(I_\kappa\otimes \QQ_k  V_k)\\
= & \sum_{j,k} (I_\kappa \otimes B_j(a,x)[h]) \widehat{\ZZ}_{j,k}(a,x) (I_\kappa\otimes B_k(a,x)[h])
\end{split}
\]
where 
$$
 \widehat{\ZZ}_{j,k}(a,x)=(I_\kappa \otimes Q_j) \ZZ_{j,k}(a,x) (I_\kappa\otimes Q_k^T).
$$
Hence, letting $\QQ$ denote the block diagonal matrix with diagonal
entries $I_\kappa\otimes \QQ_j$, 
\[
\widehat{\ZZ}(a,x) = \QQ \,  \ZZ(a,x)\QQ^T
\]
is the middle matrix in the border vector-middle matrix representation of $p_{xx}(a,x)[h]$ with
respect to the border vector $B$.

Letting $\widehat{Y}_{j,k}=\Pi_j \iota_j^T Y_{j,k}(a,x) \iota_k \Pi_k^T  =Q_jY_{j,k}(a,x) Q_k^T$
and using Theorem \ref{thm:ZAiscZa}, 
\[
\begin{split}
\widehat{\ZZ}(a,x) = & \QQ \, \ZZ(a,x) \QQ^T\\
 = & \QQ \, \Yk(a,x)^T \, \ZZ(a,0) \Yk(a,x) Q^T\\
= & \QQ \Yk(a,x)^T \QQ^T \, (\QQ \ZZ(a,0)\QQ^T) \QQ \Yk(a,x) \QQ^T\\
 = & \widehat{Y}^\kappa_{\degg}(a,x)^T \, \widehat{\ZZ}(a,0) \widehat{Y}^\kappa_{\degg}(a,x),
\end{split}
\]
which is item \eqref{it:Y5} in Theorem \ref{lem:aug2a15}. 
Since items \eqref{it:Y1} and \eqref{it:Y2} are evident from the construction, the proof of Theorem \ref{lem:aug2a15} is complete.
\end{proof}

\subsection{The proof of Theorem \ref{thm:dec29a13}}
\label{sec:proof-dec29a13}

Recall that matrices $A,\,B\in\RR^{n\times n}$ are {\bf congruent} if  there exists
an invertible matrix $C\in\RR^{n\times n}$ such that $A=C^TBC$. The matrix $C$ (or $C^{-1}$, or $C^T$ or $(C^{-1})^T$) is often referred to as the {\bf congruence}.

\begin{lemma}
\label{lem:jul28a15}
If $M\in\RR^{s\times s}$, $U\in\RR^{s\times t}$, $W\in\RR^{t\times t}$, $\lambda\in\RR$ and
$\Pi_{\cR_{W^TW}}$ denotes the orthogonal projection of $\RR^{s+t}$ onto the range of $W^T$,
 then there exists
an invertible matrix $S\in\RR^{(s+t)\times (s+t)},$ independent of $\lambda,$ such that
\begin{equation*}
\begin{bmatrix}M&0\\0&0\end{bmatrix}+\lambda\begin{bmatrix}U\Pi_{\cR_{W^TW}}U^T&UW^T\\ \\   WU^T&WW^T\end{bmatrix}=S^T\begin{bmatrix}M&0\\0&\lambda WW^T\end{bmatrix} S.
\end{equation*}

In particular, if
\begin{equation}
\label{eq:jul28a15}
\textup{range}\,U^T\subseteq\textup{range}\,W^T,
\end{equation}
then
\begin{equation}
\label{eq:jul28b15}
\begin{bmatrix}M&0\\0&0\end{bmatrix}+\lambda\begin{bmatrix}UU^T&UW^T\\ WU^T&WW^T\end{bmatrix}=S^T\begin{bmatrix}M&0\\0&\lambda WW^T\end{bmatrix} S.
\end{equation}
\end{lemma}

\begin{proof} If $W=0$, then, in view of  the constraint \eqref{eq:jul28a15}, the matrix  $U=0$ and the asserted conclusion is self-evident. Thus, it suffices to consider the case $W\ne 0$.

Let
$(WW^T)^\dagger$ denote the Moore-Penrose inverse of $WW^T$ and, for $K\in\RR^{p\times q}$,
let $\Pi_K$ denotes the orthogonal projection of $\RR^q$ onto $\cR_{K}$, the range of $K$.
As is well known,
$$
(WW^T)(WW^T)^\dagger=\Pi_{\cR_W}=\Pi_{\cR_{WW^T}}
$$
is the orthogonal projection onto the range of $W$ and
$$
W^T(WW^T)^\dagger W=\Pi_{\cR_{W^T}}=\Pi_{\cR_{W^TW}}
$$
is the orthogonal projection onto the range of $W^TW$. In particular,
$$
W^T(WW^T)^\dagger(WW^T)=W^T,\quad (WW^T)(WW^T)^\dagger W=W.
$$
Hence
\begin{equation}
\label{eq:oct9a13}
\begin{split}
&\begin{bmatrix}I_s& UW^T(WW^T)^\dagger\\  0&I_t\end{bmatrix}
\begin{bmatrix}M&0\\ 0&\lambda WW^T\end{bmatrix}
\begin{bmatrix}I_s&0\\(WW^T)^\dagger WU^T&I_t\end{bmatrix}\\
&=\begin{bmatrix}M+\lambda U\Pi_{\cR_{W^T}} U^T&\lambda UW^T\\ \lambda WU^T&\lambda WW^T\end{bmatrix}.
\end{split}
\end{equation}
Thus the matrices
\begin{equation*}
\begin{bmatrix}M&0\\ 0&\lambda WW^T\end{bmatrix}\quad\textrm{and}\quad
\begin{bmatrix}M+\lambda U\Pi_{\cR_{W^T}} U^T&\lambda UW^T\\ \lambda WU^T&\lambda WW^T\end{bmatrix}
\end{equation*}
are congruent and, as follows from\eqref{eq:oct9a13}, the congruence is independent of $\lambda$.

If the constraint \eqref{eq:jul28a15} is in force,
 then
$U \, \Pi_{\cR_{W^TW}}U^T  =U U^T $
and consequently \eqref{eq:jul28b15} holds.
\end{proof}

The following proposition contains Theorem \ref{thm:dec29a13}.

\begin{proposition}
\label{thm:dec29a13-full}
 Let $p$ denote a $\kappa\times \kappa$ symmetric matrix nc
 polynomial of degree $d$ in $x$
 as in  equation \eqref{pww}. 

  There exists matrix nc polynomials $f(x)$ and $g(a,x)$
   such that the middle matrix $\ZZ_{\lambda}$ in the representation of the modified Hessian (with respect to the full border vector) of equation \eqref{def:modHess}
 evaluated at a tuple $(A,X)\in\SS_n(\RR^{\fg})$ is of the form
\begin{equation}
\label{eq:dec29a13}
  \ZZ_\lambda(A,X) := \begin{bmatrix}\ZZ(A,X)&0\\0&0\end{bmatrix}+\lambda\begin{bmatrix}UU^T&UW^T\\WU^T&WW^T\end{bmatrix},
\end{equation}
 where $W=f(A)$ and $U=g(A,X)$.

Let  $\Pi_{\cR_{W^T}}$ denote the orthogonal projection onto the range of $W^T$
and suppose $\lambda\in\RR\setminus\{0\}.$
\begin{enumerate}[\rm(1)]
\item There exists an invertible matrix $S$, depending on $(A,X)$ and $d$, but not upon $p$ or  $\lambda$, such that
\begin{equation}
\label{eq:sep4a14}
 \begin{split}
\begin{bmatrix}
   \ZZ(A,X)&0\\0&0\end{bmatrix}+& \lambda\begin{bmatrix}U
   \Pi_{\cR_{W^T}}U^T&UW^T\\WU^T&WW^T\end{bmatrix} \\
= & S^T \begin{bmatrix}\ZZ(A,X)&0\\0&\lambda WW^T\end{bmatrix} S. \end{split}
\end{equation}
\item \label{eq:sep11a14} If the range of $U^T$ is contained in the range of $W^T,$
then the left hand side of \eqref{eq:sep4a14}
is equal to $\ZZ_\lambda(A,X)$ and hence
\[
 \mu_+(\ZZ(A,X)) = \mu_+(\ZZ_\lambda(A,X)),
\]
 for all $\lambda <0$.
\item
 The range inclusion condition of item \eqref{eq:sep11a14} holds if and only if the highest degree terms of $p$ majorize at $A$.
 \end{enumerate}
 \end{proposition}

Recall,
 \begin{equation*}
    \IktV =  \begin{bmatrix} I_\kappa\otimes V_0 \\ \vdots \\ I_\kappa \otimes V_{d-1}\end{bmatrix},
\end{equation*}
from the definition of $\wtilde{V}$ in equation \eqref{eq:deftV} and the discussion preceding.

\begin{proof}
Observe that the $\ff_j$ in equation \eqref{eq:defsffj} are precisely (up to, as usual, a permutation), $(x\otimes b)_j$ and
$$
p(a,x)=\sum_{j=0}^d \varphi_p^j(a) \ff_j(a,x)=\sum_{j=0}^d \varphi_p^j(a) (x\otimes b)_j.
$$
Thus the derivative $p_x$ is
$$
p_x(a,x)[h]=\begin{bmatrix}\varphi_p^1(a)&\cdots&\varphi_p^d(a)\end{bmatrix}\begin{bmatrix}(\ff_1)_x(a,x)[h] \\ \vdots\\
(\ff_d)_x(a,x)[h]\end{bmatrix},
$$
which, %
can be rewritten as
$$
p_x(a,x)[h]=\begin{bmatrix}\varphi_p^1(a)&\cdots&\varphi_p^d(a)\end{bmatrix}
L(a,x)
\begin{bmatrix}V_0(a,x)[h] \\ \vdots\\
V_{d-1}(a,x)[h]\end{bmatrix},
$$
where $L(a,x)$ 
is the block lower triangular matrix polynomial with $L_{ii}=I_{\mathbf{t}_j}$ that was defined in terms of the blocks of $K(a,x)$ just above  \eqref{eq:mar10b17}.

 Thus, upon writing 
$$
L(a,x)=\begin{bmatrix}M_{11}(a,x)&0\\ M_{21}(a,x)&I_{\mathbf{t}_{d-1}}\end{bmatrix}
$$
and setting
$$
w^T=\phi_1^d(a)\quad\textrm{and}\quad u^T=\begin{bmatrix}\varphi_p^1(a)&\cdots&\varphi_p^{d-1}(a)\end{bmatrix}M_{11}(a,x)
+ \varphi_p^d(a)M_{21}(a,x),
$$
it follows that
$$
p_x(a,x)[h]=\begin{bmatrix}u^T&w^T\end{bmatrix}\wtilde{V}(a,x)[h]
$$
and hence that (see equation \eqref{eq:gradsquare})
$$
p_x(a,x)[h]^T\,p_x(a,x)[h]=(\IktV)(a,x)[h]^T\begin{bmatrix}uu^T&uw^T\\wu^T&ww^T\end{bmatrix}\,
(\IktV)(a,x)[h],
$$
which serves to explain the formula \eqref{eq:dec29a13}.

In view of  equation \eqref{eq:dec29a13}, an application of Lemma \ref{lem:jul28a15} gives item (1);  item (2) follows immediately from (1). Let $U=u(A,X)$ and $W=w(A,X)$.  Since   $M_{11}$ is invertible, it is readily seen that
$$
\textup{range}\,U^T\subseteq\textup{range}\,W^T
$$
if and only if \eqref{eq:apr25a17} holds.
Hence the range inclusion $\range(U^T)\subset \range(W^T)$ holds if and only if the highest degree terms majorize, completing the proof of (3).
\end{proof}

\begin{corollary}
\label{cor:sep12a14}
 If $p\in\cP^{\dddd}$ is an nc matrix polynomial that is homogeneous of degree $d$ in $x$,  then the range inclusion condition in
 Theorem \ref{thm:dec29a13}  is met for every choice of $A\in \SS_n(\RR^{\wtilde{g}})$.
\end{corollary}

\begin{proof}
 If $p$ is homogeneous of degree $d$ in $x$, then the range inclusion condition of Proposition \ref{thm:dec29a13-full} item \eqref{eq:sep11a14}  is automatically met, since $\varphi^j_p=0$ for $j=1,\ldots,d-1.$ Compare with Remark \ref{rem:special majorize}.
\end{proof}

\section{The CHSY lemma}
\label{sec:CHSY}
Let $\{w_{i1}, \dots, w_{i\beta_i}\}$ denote  the set of  words in the right chip set
$\cR\cC_p^i$ of the nc polynomial $p$ of degree at most $d$ in $x$.
Given a positive integer $\kappa$,  a triple $(A,X,v)\in\allmatvk$ with $v\ne 0$, an  $H\in\smatn$ and a subspace
$\cH$ of $\smatn,$
let
\begin{equation*}
\cR_\kappa^i(H):= \begin{bmatrix}I_\kappa \otimes Hw_{i1}(A,X) \\ \vdots \\ I_\kappa \otimes Hw_{i\beta_i}(A,X) \end{bmatrix}
\end{equation*}
and let $\cR^i(\cH)=\{\cR^i(H):H\in \cH \}.$
The number $\beta_i$ of distinct words
in $\cR\cC_p^i$ is equal to the dimension of the  chip space $\cC_p^i,$  which is equal the span of $\cR\cC_p^i$ in the space of nc polynomials.
For $H=(H_1,\ldots,H_g)\in\smatng$, let
\begin{equation}
\label{eq:aug28c14}
\cR_\kappa^{\cCp}(H):=  \begin{bmatrix} \cR_\kappa^1(H_1) \\ \vdots \\ \cR_\kappa^g(H_g)\end{bmatrix}
\end{equation}
 and let $\cR^{\cCp}(\cH)=\{\cR^{\cCp}(H): H \in \cH\}.$
Note that the words $w_{ij} $ in  \eqref{eq:aug28c14} all have degree at most $d-1$ in $x$ and that
\begin{equation*}
\cR^{\cCp}(\cH)v\subseteq \RR^{n\beta}\quad\textrm{with $\beta=\beta_1+\cdots+\beta_g$}.
\end{equation*}

The border vector obtained by using just the words from $\mathcal C$ will be called the \df{reduced border vector} based on $\mathcal C$ and is generally much smaller than the border vector based on all words of degree at most $\tilde{d}$ in $a$ and $d-1$ in $x$. The reduced border vector  evaluated at $(A,X,H)$ is $\cR^{\cCp}(H)$.

The main tool developed here, Theorem \ref{thm:chsymatrix} below, is an elaboration of the following result.

\begin{lemma}
{\bf (The CHSY lemma)}
\label{lem:chsyRef}
If $\{u_1,\ldots,u_k\}$ is a set of linearly independent vectors in $\RR^n$, then
\begin{equation}
\label{eq:aug7a13}
\textup{dim}\,\textup{span}\,\left\{\begin{bmatrix}Hu_1\\ \vdots\\
Hu_k\end{bmatrix}:\,H\in\SS_n\right\}=kn-\frac{k(k-1)}{2}.
\end{equation}
Therefore, the codimension of the space on the left hand side of (\ref{eq:aug7a13}) in $\RR^{kn}$ is equal to
$k(k-1)/2$, independently of $n$.
\end{lemma}

\begin{proof}
This result is  Lemma 9.5 of  \cite{CHSY}.
\end{proof}

If $w$ is a word in $(a,x)$,  $(A,X)\in\allmatn$, $C\in\RR^{\kappa\times\kappa}$ and
$H\in\SS_n$,
then
\begin{equation*}
C\otimes Hw(A,X)=(C\otimes I_n)(I_\kappa\otimes Hw(A,X)).
\end{equation*}
Thus, if $Hw(A,X)$ belongs to the reduced border vector
  based on the   chip set of $p=\sum_{w\in\cW}c_ww$, then $I_\kappa\otimes Hw(A,X)$ will belong to the reduced border
  vector for the matrix  polynomial $p=\sum_{w\in\cW}C_w\otimes w$ with $C_w\in\RR^{\kappa\times\kappa}$;  the matrix $C_w$
gets absorbed into the middle matrix of the Hessian.

\begin{lemma}
\label{lem:aug28b14}
Set $v=\textup{col}(v_1,\ldots,v_\kappa)$ with components
$v_1,\ldots,v_\kappa\in\RR^n$ and let $\{w_1,\dots,w_s\}$ be a given set of
 words in $a$ and $x$.  If $(A,X)\in\allmatn$ and  if  the set of $\kappa s$ vectors
\[
 \{ w_i(A,X)v_k : 1\le i\le s, \ \ 1\le k\le \kappa\}
\]
 is linearly independent, then the codimension of
$$
\textup{span}\,\left\{\begin{bmatrix}[I_\kappa\otimes Hw_1(A,X)]v\\
 \vdots\\ [I_\kappa\otimes Hw_s(A,X)]v\end{bmatrix}:\, H\in\SS_n\right\}
$$
in $\RR^{s\kappa n}$ is equal to $s\kappa(s\kappa-1)/2$,  independently of $n$.
\end{lemma}

\begin{proof}
This lemma follows from the CHSY lemma (Lemma \ref{lem:chsyRef}),  by choosing an enumeration
$\{u_1,\dots, u_{s\kappa}\}$  of $\{w_i(A,X)v_k: 1\le i\le s, \, 1\le k\le \kappa\}$
 and observing for $H\in\SS_n$ and up to permutation,
\[
\begin{bmatrix}[I_\kappa\otimes Hw_1(A,X)]v\\
 \vdots\\ [I_\kappa\otimes Hw_s(A,X)]v\end{bmatrix}
=\begin{bmatrix}Hu_1\\ \vdots\\
Hu_k\end{bmatrix}.
\]
\end{proof}

\begin{theorem}
\label{thm:chsymatrix}
 Fix $1\le j\le g$ and  $(A,X,v)\in\allmatvk.$  Let $v_k$ for $1\le k\le \kappa$
 denote the entries of $v$.
 If
\begin{equation*}
 \{w(A,X)v_k: w\in\cR\cC_p^j, \ \ 1\le k\le\kappa \}
\end{equation*}
 is a linearly independent subset of $\RR^n$ or, equivalently,
  the mapping
\begin{equation*}
     (\cC_p^j)^\kappa  \ni q \mapsto q(A,X)v \in \RR^n
\end{equation*}
  given by
\[
  (q_1,\dots,q_\kappa)\mapsto \sum_{k=1}^\kappa  q_k(A,X)v_k
\]
  is one-one, then
\begin{equation*}
\textup{codim}\,
\cR_\kappa^j(\smatn)v
=\frac{\kappa\beta_j(\kappa\beta_j-1)}{2}
\end{equation*}
 as a subspace of  $\RR^{ n \kappa \beta_j}.$

 If for each $j$, the set  $\{w(A,X)v_k: w\in \cR\cC_p^j, \, 1\le k\le \kappa\}$ is linearly independent, then
\begin{equation}
\label{eq:aug28h14}
\textup{codim}\, \cR_\kappa^{\cCp}(\smatng)v
= \sum_{i=1}^g \frac{\kappa\beta_i(\kappa\beta_i-1)}{2}
\end{equation}
as  a subspace of  $\RR^{n\kappa \beta .}$
\end{theorem}

\begin{proof}
 The first conclusion is immediate from Lemma \ref{lem:aug28b14}.
The second is its manifestation in the direct sum $\cC_p$
 of the $\cC_p^j$ and is thus an immediate consequence of the first statement.
\end{proof}

\section{Positivity of the  middle matrices forces $d\le 2$}
 \label{sec:posM}

\begin{lemma}
 \label{lem:Z=0 d=0}
   If $p(a,x)$ is an  nc symmetric matrix-valued polynomial of degree $d\ge 2$  in $x$, then  the  $\ZZ_{0,d-2}$ entry in the middle matrix $\ZZ$ of its Hessian with respect to $x$ is nonzero and depends only on $a$.
\end{lemma}

\begin{proof}

Write
$$
p(a,x)=\sum_{w\in\cW}C_w\otimes w(a,x),  %
$$
with $C_w\ne 0$ for every $w\in\cW$. Here $\cW$ is a set of words $w(a,x)$ of  degree at most $d$ in $x$ and $C_w\in\RR^{\kappa\times\kappa}$.

A word  $w\in\cW$ of degree $d\ge 2$  in $x$ must be of the form
\begin{equation*}
w(a,x)=u(a)x_iv(a)x_jf(a,x)\quad\textrm{for some
 $i,\,j\in\{1,\ldots,g\}$},
\end{equation*}
where $u(a)$ and $v(a)$ are words that depend only upon $a$ and $f(a,x)$ is a word of degree  $d-2$ in $x$ that may depend upon both $a$ and $x$.  Let
$$
\cW^\prime=\cW_{u,i,j,f}^\prime
$$
denote the set of words in $\cW$ that begin with $u(a)x_i$ and end with $x_jf(a,x)$ and,
 assuming that $\cW^\prime\ne\emptyset$, let
\begin{align*}
q(a,x)&=\sum_{w\in\cW^\prime}C_w\otimes w(a,x)\\
&=\sum_{w\in\cW^\prime}C_w\otimes u(a)x_iv_w(a)x_jf(a,x)\\
&=\left(I_\kappa\otimes u(a)x_i\right)\left(\sum_{w\in\cW^\prime}C_w\otimes v_w(a)\right)\left(I_\kappa\otimes x_jf(a,x)\right).
\end{align*}
 Since, for $w\in \cW_{u,i,j,f}^\prime$,
\begin{align*}
w_{xx}(a,x)[h]&=2u(a)h_iv(a)h_jf(a,x)\\ &+2u(a)h_iv(a)x_jf_x(a,x)[h] +2u(a)x_iv(a)h_jf_x(a,x)[h]\\ &+u(a)x_iv(a)x_jf_{xx}(a,x)[h],
\end{align*}
the $\kappa\times\kappa$ subblock of  $\ZZ_{0,d-2}$ that is specified by the  $I_\kappa\otimes ( u(a)h_i)$ and $I_\kappa\otimes h_jf(a,x)$ entries of the border vectors $V^\kappa_{\degg}(x,a)[h]^T$ and $V^\kappa_{\degg}(a,x)[h]$
 respectively is equal to
$$
 2\left(\sum_{w\in\cW^\prime}C_w\otimes v_w(a)\right).
$$
Thus, if $\ZZ_{0,d-2}=0$, then
$$
2\left(\sum_{w\in\cW^\prime}C_w\otimes v_w(a)\right)=0
$$
and hence $q(a,x)=0$. Since the same argument applies for every such subblock of $\ZZ_{0,d-2}$, and each such subblock depends only upon $a$ and not upon $x$,  for each word $w$  of degree $d$ in $x$  the coefficient $C_w=0$ and we have reached a contradiction.
\end{proof}

One last piece of notation is needed to state our next lemma.
Let $p_2(a,x)$ denote the homogeneous of degree two in $x$ portion of the nc symmetric polynomial $p$. The polynomial $p_2$ can be expressed as
\[
 p_2(a,x) = \sum_{\sigma,\tau,j,k} \sigma^T(a) x_k r_{\sigma,\tau,j,k}(a) x_j \tau(a),
\]
 where $\sigma,\tau$ are words in $a$ and $r_{\sigma,\tau,j,k}(a)$ is a $\kappa\times\kappa$ matrix polynomial in $a$ such that the sum of the degrees of $\sigma,\tau,r_{\sigma,\tau,j,k}$ is at most the degree of $p_2$ in $a$.  Let $R^p(a)$ denote the matrix indexed by $1\le j\le g$ and words of length at most that of the degree of $p_2$ in $a$,
with $((\sigma,k),(\tau,j))$ entry,
\[
 -r_{\sigma,\tau,j,k}(a)
\]
 and let $S^p$ denote the vector polynomial   with $(\tau,j)$ entry equal to the $\kappa\times \kappa$ matrix,
\[
 S^p_{\tau,j}(a,x)= x_j \tau(a) I_\kappa.
\]
 Thus,
\begin{equation*}
 p_2(a,x) = -S^p(a,x)^T R^p(a) S^p(a,x).
\end{equation*}

\begin{lemma}
 \label{lem:aug6a15}
   Let $\ZZ$ denote the middle matrix of the Hessian of a nonzero nc symmetric polynomial $p$ of degree
    $d$ in $g$ variables $x$ and degree
   $\tilde{d}$ in $\tilde{g}$ variables $a$ and suppose $N\ge \sum_{j=0}^{\tilde{d}} \tilde{g}^j$. For positive integers $n$, let $\mathscr U(n)$ denote the set of those $A\in\mathbb S_n(\mathbb R^{\tilde{g}})$ for which there is an $X$ such that $\ZZ(A,X)\preceq 0$.
If $\mathscr U(N)$ has nonempty interior, then
  \begin{enumerate}[\rm(i)]
   \item
    \label{it:degree2}
   $p(a,x)$ has degree at most two in $x;$
   \item there exists a polynomial $\ell(a,x)$ that is affine linear in $x$ such that
   \label{it:pform}
\begin{equation}
 \label{eq:wow+}
 p(a,x) =\ell(a,x)- S^p(a,x)^T R^p(a) S^p(a,x);
\end{equation}
\item \label{it:Rpos} $R^p(A)\succeq 0$ for each $A\in\mathscr U$.
 \end{enumerate}
\end{lemma}

\begin{proof}
Suppose $d>2$.  In this case, $\ZZ_{0,d-2}(a,x)=\ZZ_{0,d-2}(a)$ is not zero and depends only upon $a$ by Lemma \ref{lem:Z=0 d=0}. Given $(A,X)$,
if $\ZZ(A,X)\preceq 0,$ then the submatrix
$$
M=\begin{bmatrix}\ZZ_{00}(A,X) &\ZZ_{0,d-2}(A)\\ \ZZ_{d-2,0}(A) &0\end{bmatrix}\preceq 0,
$$
therefore, $\mu_+(M)=0.$ The inequality
$$
 \mu_{\pm}(M)\ge \mbox{rank}(\ZZ_{0,d-2}(A))
 $$
 now implies that $\ZZ_{0,d-2}(A)=0$. Therefore,
 $\ZZ_{0,d-2}(A)=0$ on an open set in $\mathbb S_N(\mathbb R^{\tg})$ and hence that $\ZZ_{0,d-2}=0$, contradicting Lemma \ref{lem:Z=0 d=0}.
 Consequently,  $d\le 2$ and in particular $\ZZ_{00}(a,x)=\ZZ_{00}(a)$ depends only upon $a$.

 Next, to verify item \eqref{it:pform}, write
 $$
 p(a,x) =\sum_{w\in\cW} C_w \otimes w + \ell(a,x),
 $$
 where $\cW$ is a set of  words $w(a,x)$ of degree two in
$x$  and $\ZZ_{0,0}(A,X)\preceq 0$ and $\ell(a,x)$ is affine linear in $x$.  Then
\[
  p_{xx}(a,x)[0, h] =  \sum_{w\in\cW} C_w \otimes w_{xx}(a,x)[h]= 2[p(a,h)-\ell(a,h)].
\]
 On the other hand,
\[
 p_{xx}(a,x)[h] = 2 (I_\kappa\otimes V_0(a,x)[h])^T\ZZ_{0,0}(a,0) (I_\kappa \otimes V_0(a,x)[h]),
\]
 $I_\kappa\otimes V_0(a,x)[h]=I_\kappa\otimes(h\otimes b)$  depends only on $h$ and
 $a$ (and not on $x$).
 Hence,
\[
  p(a,x) = \frac{1}{2}\, 2(I_\kappa\otimes V_0(a,x)[x])^T \ZZ_{0,0}(a,0) (I_\kappa \otimes V_0(a,x)[x]) +\ell(a,x)
\]
 with $I_\kappa\otimes V_0(a,x)[x]=I_\kappa\otimes(x\otimes b)$ linear in $x$.  In particular, $p(A,X)\preceq 0$
whenever  $\ZZ_{0,0}(A,0)\preceq 0.$ Choosing $R(a)=-\ZZ_{00}(a,0)$, $S=I_\kappa\otimes V_0(a,x)[x]$ produces the
representation of equation \eqref{eq:wow+} from which item \eqref{it:Rpos} immediately follows.
\end{proof}

\section{Positivity of Middle Matrices}
 \label{sec:moreposM}

This section focuses on the  positivity of  middle matrices for the Hessians, the
relaxed Hessians and the positivity of various second derivatives. The results will be used in
 the proof of  Theorem \ref{thm:main}.  Throughout the polynomial $p$ is fixed.

The following consequence of Theorem \ref{thm:dec29a13} generalizes Proposition 5.2 in \cite{DHMill}. \index{$\ZZ_{\lambda,\delta}$}

\begin{lemma}
\label{prop:scalar-middle-relaxed}
Let $(A, X)  \in \allmatn  $ be given and let $\ZZ_{\lambda,\delta}(A, X)$ denote the middle matrix for the relaxed Hessian
of a symmetric nc matrix polynomial $p\in\cP^{\kappa\times\kappa}.$
 There exists an $\epsilon <0$ such  that if $\varepsilon \le \delta \le 0$ and $\lambda \le 0$,   then
    \begin{equation}
    \label{eq:aug26e14}
   \mu_+(\ZZ_{\lambda,\delta}(A, X)) \le \mu_+(\ZZ(A,0))
 \end{equation}
 with equality if the matrices $U$ and $W$
  in  formula \eqref{eq:dec29a13} for the middle matrix of the modified Hessian,
 $$p_{xx}(A,X)[h]+\lambda p_x(A,X)[h]^Tp_x(A,X)[h],$$
  meet the range inclusion condition \eqref{eq:jul28a15}
   (equivalently the  highest degree terms of $p$ majorize at $A$).
\end{lemma}

\begin{proof}
In view of Theorem \ref{thm:dec29a13},
$$
\ZZ_{\lambda,\delta}(A,X)=\lambda \alpha+S^T\beta S+\delta I,
$$
where
$$
\alpha=\begin{bmatrix}U(I-\Pi_{\cR_{W^TW}})U^T&0\\0&0\end{bmatrix}, \quad
\beta=\begin{bmatrix}\ZZ(A,0)&0\\0&\lambda WW^T\end{bmatrix}
$$
and $S$ is  an invertible matrix  that  depends on $(A,X)$, but not on $\lambda$.
Since the additive perturbation
by  the negative definite matrix
$\delta I$ with $\delta<0$ shifts the eigenvalues of the matrix $S^T\beta S$ to the left, the nonpositive eigenvalues of $S^T\beta S$ become negative but the positive eigenvalues will stay positive if $\delta\in[-\varepsilon,0]$ and $\varepsilon>0$ is small enough.
For such $\delta$ and each $\lambda\le 0$,
\begin{align*}
\mu_+(S^T\beta S+\delta I)&=\mu_+(S^T\beta S)=\mu_+(\beta)\\
&=\mu_+\left(\begin{bmatrix}\ZZ(A,0)&0\\0&\lambda WW^T\end{bmatrix}\right)\\
&=\mu_+(\ZZ(A,0)).
\end{align*}
The inequality \eqref{eq:aug26e14}
 follows from the fact that, for $\lambda \le 0$,
$$
\mu_+(\lambda\alpha+S^T\beta S)\le\mu_+(S^T\beta S)=\mu_+(\beta).
$$
Equality prevails in \eqref{eq:aug26e14} under the added assumption that the range of
 $U^T$ is contained in the range of $W^T$ because in that case  $\alpha=0$.
\end{proof}

The next result is a  variant of Proposition 6.7 of
\cite{DHMill}. It provides a bound on $\mu_+(\ZZ(A,0))$ which, as will be seen later, turns out to be independent of $n$, the size of $A$.

\begin{lemma}
 \label{lem:chsy-applied}
    Suppose $(A,X,v)\in \allmatvk$ and  there exists an $\varepsilon<0$ such that
  equality holds in \eqref{eq:aug26e14} for
  all  $\epsilon\le \delta \le 0$ and $\lambda\le 0.$
  If $\epsilon\le \delta\le 0$ and $\lambda\le 0$ and if  $\mathcal H_+^{\lambda,\delta}$ is a
  maximal strictly positive subspace  for the quadratic form
\begin{equation*}
\smatng \ni H\mapsto
  \langle p^{\prime\prime}_{\lambda,\delta}( A , X)[h]v,v\rangle,
\end{equation*}
then
\begin{equation}
\label{eq:jul26a9}
        \mu_+(\ZZ(A,0))
          \le \textup{dim}\chp + \textup{codim} \cR_\kappa^{\cCp}(\Smatng),
    \end{equation}
 where the codimension is computed in the space $\RR^{ n \kappa \beta}$,
    $\beta=\beta_1+\cdots+\beta_g$ and $\beta_j=\textup{dim }\cC_p^j$ for $j=1,\ldots,g.$
\end{lemma}

\begin{proof}
Given $\lambda,\delta\le 0$, let $\chp$ be a maximal strictly positive subspace of $\Smatng$ with respect to the quadatic form
\begin{equation*}
\langle p^{\prime\prime}_{\lambda,\delta}( A , X)[h]v,v\rangle=v^T(\IktV)(A,X)[H]^T\ZZ_{\lambda,\delta}(A,X)(\IktV)(A,X)[H]v.
\end{equation*}
 Let $\ckp$ be a complementary subspace  of $\cH_+^{\lambda,\delta}$ of $\Smatng$ on which this quadratic form is nonpositive. Recall
$$
(\IktV)(A,X)[H]v\in\RR^{n\nu\kappa}\quad\textrm{with $\nu=\ft_0+\cdots+\ft_{d-1}$}
$$
and
\[
 \cR_{\kappa}^{\cCp}(H)v \in \RR^{n\beta\kappa}
\]
 is  obtained from $(\IktV)(A,X)[H]v$ by deleting those entries corresponding to words not in the  chip set $\cC_p$. In particular, there is a canonical inclusion  $\iota:\mathscr R \to\mathscr V$, where $\mathscr R$ and $\mathscr V$ are the subspaces  $\cR_{\kappa}^{\cCp}(\Smatng)v$ and $(\IktV)(A,X)[\Smatng]v$
of $\RR^{n\beta\kappa}$ and $\RR^{n\nu\kappa}$ respectively.

In particular,
\[
 \iota^T(\IktV)(A,X)[H]v=\cR_{\kappa}^{\cCp}(H)v.
\]
 It follows that
 \begin{align*}
&v^T(\IktV)(A,X)[H]^T\ZZ_{\lambda,\delta}(A,X)(\IktV)(A,X)[H]v\\ &=v^T\IktV(A,X)[H]^T\iota\iota^T\ZZ_{\lambda,\delta}(A,X)\iota\iota^T\IktV(A,X)[H]v\\
&=v^T \cR_{\kappa}^{\cCp}(H)^T\iota^T\ZZ_{\lambda,\delta}(A,X)\iota\cR_{\kappa}^{\cCp}(H)v.
\end{align*}

It is now readily checked that
$$
\cR_{\kappa}^{\cCp}(\chp)v \cap\cR_{\kappa}^{\cCp}(\ckp)v=\{0\},
$$
because if $y\in\cR_{\kappa}^{\cCp}(\chp)$, then either $y=0$ or $y^T\iota^T\ZZ_{\lambda,\delta}(A,X)\iota y>0$, whereas if
$y\in\cR_{\kappa}^{\cCp}(\ckp)$, then $y^T\iota^T\ZZ_{\lambda,\delta}(A,X)\iota y\le 0$. Thus, the sum $\cR_{\kappa}^{\cCp}(\chp)v+\cR_{\kappa}^{\cCp}(\ckp)v$ is direct.
Consequently, letting $\dot{+}$ denote the algebraic direct sum,  there exists a subspace $\cY$ of $\RR^{n\beta\kappa}$ such that
$$
\RR^{n\beta\kappa}=\cR_{\kappa}^{\cCp}(\chp)v\dot{+}\cR_{\kappa}^{\cCp}(\ckp)v\dot{+}\cY.
$$
Therefore,
\begin{align*}
n\beta\kappa &=\textup{dim}\cR_{\kappa}^{\cCp}(\chp)v+\textup{dim}\cR_{\kappa}^{\cCp}(\ckp)v+\textup{dim}\cY\\
&=\textup{dim}\cR_{\kappa}^{\cCp}(\Smatng)v+\textup{dim}\cY
\end{align*}
and
\begin{equation}
 \label{eq:6sep14s1}
\begin{split}
n\beta \kappa -\textup{dim}\cR_{\kappa}^{\cCp}(\ckp)v&=\textup{dim}\cR_{\kappa}^{\cCp}(\chp)v+\textup{dim}\cY\\
&\le \textup{dim}\chp + \textup{codim} \cR_{\kappa}^{\cCp}(\Smatng)v.
\end{split}
\end{equation}

The next step is to verify  the bound
\begin{equation*}
\mu_+(\ZZ_{\lambda,\delta}(A,X))=\mu_+(\iota^T\ZZ_{\lambda,\delta}(A,X)\iota)\le n\beta\kappa-\textup{dim}\cR_{\kappa}^{\cCp}(\ckp).
\end{equation*}
To this end, let $U\in\RR^{n\beta\kappa \times n\beta\kappa}$   be an invertible matrix in which the first $k$ columns is a basis for
$\cR_{\kappa}^{\cCp}(\ckp)v$ and let $\mu_{\le}(M)$ denote the
dimension of the space spanned by the eigenvectors of a real symmetric matrix $M$ corresponding to its nonpositive eigenvalues. Then, as
$$
\mu_{\le}(\iota^T\ZZ_{\lambda,\delta}(A,X)\iota)=\mu_{\le}(U^T\iota^T\ZZ_{\lambda,\delta}(A,X)\iota U)\ge k=\textup{dim}\cR_{\kappa}^{\cCp}(\ckp)v,
$$
it follows that
\begin{equation*}
\begin{split}
\mu_+(\ZZ_{\lambda,\delta}(A,X))&=\mu_+(\iota^T\ZZ_{\lambda,\delta}(A,X)\iota)=n\beta\kappa-\mu_{\le}(\iota^T\ZZ_{\lambda,\delta}(A,X)\iota)\\ &\le n\beta\kappa-\textup{dim}\cR_{\kappa}^{\cCp}(\ckp)v.
\end{split}
\end{equation*}
By assumption,
   $\mu_+(\ZZ_{\lambda,\delta}(A,X)) =  \mu_+(\ZZ(A,0))$
    for appropriate choices of $\delta$ and $\lambda$.
  Hence,
\[
 \mu_+(\ZZ(A,0)) \le n\beta\kappa-\textup{dim}\cR_{\kappa}^{\cCp}(\ckp)v.
\]
 An application of the inequality in \eqref{eq:6sep14s1} completes the proof.
 \end{proof}


 The next proposition is a variant of Lemma 7.2 from \cite{DHMill} and follows the strategy  of  Proposition  5.4 of \cite{BM}.

\begin{proposition}
 \label{prop:dhm2}
   Let $p\in\cP^{\kappa\times\kappa}$ be a symmetric nc matrix polynomial of degree $\wtilde{d}$ in $a$
  and degree $d$ in $x$.
Assume that $( A,  X, v) \in \allmatn \otimes \RR^{\kappa n}$ with  $v=\textup{col}(v_1,\ldots,v_\kappa)$
meets the following conditions:
\begin{enumerate}[\rm(1)]
  \item
    There exists a subspace $\cH$ of $\cT(A,  X, v)$ of codimension
    at most one in $\cT(A,X,v)$ such that
  \begin{equation*}
   \langle p_{xx}(A, X)[h]v,v\rangle \le 0 \quad \text{for each}\ H\in\cH;
 \end{equation*}
  \item \label{it:linind}
  The set
\[
  \{ w(A,X)v_k: 1\le k\le \kappa, \, w\in \cC_p^i\}
\]
  is linearly independent for each choice of $i=1,\ldots,g$; and
 \item \label{it:rangeWU2} The highest degree terms of $p$ majorize at $A$.
\end{enumerate}

Then there exists an integer $\gamma$ that depends  on  the number of variables  and the degree of
the polynomial $p(a,x)$ but is independent of $n$ such that the middle matrix $\ZZ(A,X)$ of the Hessian $p_{xx}(A,X)[h]$  is subject to the constraint
\begin{equation*}
 \mu_+( \ZZ(A, 0)) \le \gamma.
 \end{equation*}
\end{proposition}

\begin{proof}
Recall the definition of $e_{\pm}^n$; it is given just below \eqref{def:epmn}.  In view of assumption (i),
$e_+^n(A,X,v,p_{xx},\cT(A,X,v))\le 1$.
Therefore, by Lemma \ref{thm:signature-clamped-relaxed},
  there is a $\delta_0<0$ such that for
  each $\delta_0\le\delta<0$ there exists
  a $\lambda<0$ such that
   \begin{equation*}
  e^n_+(A, X,v;p^{\prime\prime}_{\lambda,\delta},\Smatng)\le 1.
\end{equation*}
By item \eqref{it:rangeWU2}, all the conditions of Lemma \ref{prop:scalar-middle-relaxed} are met. Hence the hypotheses of Lemma \ref{lem:chsy-applied} are met.  By \eqref{eq:jul26a9},
\begin{align*}
\mu_+(\ZZ(A,0))&\le \textup{dim}\cH_+^{\lambda,\delta}+\textup{codim}\cR_{\kappa}^{\cCp}(\Smatng) v \\
& = e^n_+(A, X,v;p^{\prime\prime}_{\lambda,\delta},\Smatng)+\textup{codim}\cR_{\kappa}^{\cCp}(\Smatng) v \\
&\le1+\textup{codim}\cR_{\kappa}^{\cCp}(\Smatng)v=\gamma.
\end{align*}
The hypothesis in item \eqref{it:linind} justifies the use of formula \eqref{eq:aug28h14} of Theorem \ref{thm:chsymatrix} to
conclude that $\gamma$ is independent of $n$.
\end{proof}

\begin{remark}\rm
 \label{rem:dhm2}
   In the setting of Proposition \ref{prop:dhm2}, item \eqref{it:linind} can be replaced with the condition
\begin{enumerate}
 \item[$(2^\prime)$]
    if $q\in\cC_p^{1\times \kappa}$ and $q(A,X)v=0$, then $q=0$.
\end{enumerate}
 To see that $(2')$ implies $(2)$, fix an $i$ and suppose, for $c_w\in \RR$,
\[
 \sum_{w\in \cC_p^i} c_w w(A,X)v_i =0.
\]
 Let $q\in \cC_p^i$ denote the polynomial $ q = \sum_{w\in \cC_p^i} c_w w \otimes e_i$
(so the only non-zero entry of the vector polynomial $q$ is in the $i$-th position). With this choice of $q$,
\[
 q(A,X)v = \sum_{w\in \cC_p^i} c_w w(A,X)v_i \otimes e_i =0.
\]
 Thus, by hypothesis, $q=0$ and therefore $c_w=0$ for each $w\in \cC_p^i$.
\end{remark}

\section{Proof of the Main Result, Theorem \ref{thm:main}}
 \label{sec:proofmain}
 Recall $p\in\cP^{\kappa\times\kappa}$ is a symmetric nc matrix polynomial in $a$ and $x$,  $\cO\subset \SS(\RR^{\fg})$ is a free open semialgebraic set and
\[
\begin{split}
 \xxx(n) =\{(A,X,v) & \in \allmatvk :(A,X) \in \cO,
      p(A,X)\succeq 0, \\   & \ v\ne 0 \ \textrm{and}\  p(A,X)v=0 \}.
\end{split}
\]
 We will first show, if $(\hA,\hX,\hv)\in\xxx$ is a $\cC_p$ dominating point  for $\xxx$, then $\ZZ(A,0)$, its middle matrix evaluated at $(A,0)$, is positive semidefinite. Accordingly,  given  a positive integer $m$, let
 \[
 (B,Y,w):= (I_m \otimes  \hA,
   I_m \otimes  \hX,\textup{col}_m\{\hv,0,\ldots,0\})
\]
with $m-1$ zero vectors of the same height as $\hv$ in the last entry and proceed in steps.
\bigskip

\noindent
{\bf 1.} {\it There exists a point $(\wtilde{B},\wtilde{Y},\wtilde{w})\in \xxx$ arbitrarily close to $(B,Y,w)$
 that is $\cC_p$ dominating for $\xxx$ such that  the pair $(\wtilde{B},\wtilde{Y})$ is full rank and}
 \begin{enumerate}[\rm(1)]
 \item  $p(\wtilde{B},\wtilde{Y})\succeq 0$
 \item $p(\wtilde{B},\wtilde{Y})\wtilde{w}=0$.
 \item $\textup{dim kernel } p(\wtilde{B},\wtilde{Y})=1$.
 \item $\fP_p^{\wtilde{B}}\cap \mathscr W$ {\it is
    convex for some open subset $\mathscr W$ of $\SS(\RR^{\fg})$ containing $\wtilde{Y}$.}
 \end{enumerate}
 \bigskip

The point  $(B,Y,w)\in\xxx$ and is a $\cC_p$ dominating point for $\xxx$.
  By hypotheses \eqref{it:fullrank} of Theorem \ref{thm:main} and
  Lemma \ref{lem:indDom}, there exists a point  $(\wtilde{B},\wtilde{Y},\wtilde{w})\in \xxx$
 that is $\cC_p$ dominating for $\xxx$ such that  the pair $(\wtilde{B},\wtilde{Y})$ is full rank.
  By Proposition \ref{lem:generic}  and another application of Lemma \ref{lem:indDom},
  it can be assumed that the kernel of $p(\wtilde{B},\wtilde{Y})$
  is spanned by multiples of a single nonzero vector $\wtilde{w}$ and  that $(\wtilde{B},\wtilde{Y},\wtilde{w})$ is still $\cC_p$  dominating
  for $\xxx$.   It remains only to verify that item(4) in the list is in force. But this
 is covered by hypothesis \eqref{it:theta1} of Theorem \ref{thm:main}.
\qed

 \bigskip

 \noindent
 {\bf 2.} {\it  There exists a positive integer $\gamma$ that is independent
of the integer $m$ appearing in the construction of the tuple $(\wtilde{B},\wtilde{Y},\widetilde{w})$ such that}
 \begin{equation}
 \label{eq:mu minus 2}
 \mu_+(\ZZ(\wtilde{B},0))\le \gamma .
\end{equation}
 \bigskip

The proof rests heavily on Proposition \ref{prop:dhm2}. Thus,  our first task is to show that  it applies to $(\wtilde{B},\wtilde{Y},\wtilde{w})$.
In view of Step 1, we may apply Proposition   \ref{prop:maintool}    to the point
 $(\wtilde{B},\wtilde{Y},\wtilde{w})$ to guarantee the existence of a subspace $\mathscr H$  of $\cT(\wtilde{B},\wtilde{Y},\wtilde{w})$ of codimension at most one
  in $\cT(\wtilde{B},\wtilde{Y},\wtilde{w})$  such  that
\[
  \langle  p_{xx}(\wtilde{B},\wtilde{Y})[H]\wtilde{w},\wtilde{w} \rangle \, \le 0 \quad
  \textrm{for $H\in \mathscr H$},
\]
i.e., the first condition in  Proposition \ref{prop:dhm2} is met.

The validity of the second  condition in Proposition \ref{prop:dhm2} follows from the fact that
 $(\wtilde{B},\wtilde{Y},\wtilde{w})$ is a $\cC_p$ dominating point for $\xxx$. Thus,  if $q\in \cC_p^{1\times \kappa}$
  and $q(\wtilde{B},\wtilde{Y})\wtilde{w}=0$,
 then $q(A,X)v=0$ for all $(A,X,v)\in \xxx$.   Since
  $\partial\widehat{\fP}_p \cap \widehat{\cO}$  is  a $\cC_p$ dominating set for $\cO$ by hypothesis \eqref{it:YYY} of Theorem \ref{thm:main}, it follows that $q(A,X)v=0$ for all $(A,X,v)\in \widehat{\mathcal O}$.  Thus $q(A,X)=0$ for $(A,X)\in\mathcal O$ and since $\mathcal O$ is  open, $q=0$. The desired conclusion now follows from Remark \ref{rem:dhm2}.

 The third condition is hypothesis \eqref{it:theta3} of Theorem \ref{thm:main}.
  Thus, Proposition \ref{prop:dhm2} is applicable.
\qed
\bigskip

\noindent
 {\bf 3.} {\it $\ZZ(\hA,0)\preceq 0$.}
 \bigskip

    By Theorem \ref{lem:aug2a15}, the middle matrices $\ZZ(B,Y)$
  and $\ZZ(\wtilde{B},\wtilde{Y})$ for the Hessian of $p$ are
  polynomially congruent to the middle matrices  $\ZZ(B,0)$ and $\ZZ(\wtilde{B},0),$ of the Hessian of $p$,
    respectively. Thus,
       by choosing $(\wtilde{B},\wtilde{Y})$ sufficiently close to $(B,Y)$, it can be assumed that
\begin{equation}
 \label{eq:mu minus 1}
   \mu_{+}(\ZZ(B,0)) = \mu_{+}(\ZZ(B,Y)) \le \mu_{+} (\ZZ(\wtilde{B},\wtilde{Y})) = \mu_{+}(\ZZ(\wtilde{B},0)).
\end{equation}
(The middle inequality in \eqref{eq:mu minus 1} holds because the strictly positive eigenvalues
of $\ZZ(B,0)$ will stay positive under small perturbations of $B$.)
  Combining Equations \eqref{eq:mu minus 1} and \eqref{eq:mu minus 2} and the evident fact that
   $\mu_{+}$ is additive with respect to direct sums,
\[
 m\,   \mu_{+}(\ZZ(\widehat{A},0)) = \mu_{+}(\ZZ(B,0))\le \mu_{+}(\ZZ(\wtilde{B},0)) \le \gamma.
\]
 Since the right hand side of the last inequality is independent of $m$, it follows that $\ZZ(\widehat{A},0)\preceq 0$.
\qed

\bigskip

To complete the proof,  let $\tsU =\pi_1(\partial \fP_p\cap \cO)$,   let  $A\in\mathscr \tsU$ be given and let $X$ and $v$ be such that $(A,X,v)\in \xxx$. By Lemma \ref{lem:ind1}, there exists 
  a $\cC_p$ dominating point $(A_*,X_*,v_*)\in\xxx$ for $\xxx$.  

    Let $(\hA,\hX,\hv) = (A,X,v)\oplus (A_*,X_*,v_*)$ and
 note that $(\hA,\hX,\hv)\in \xxx$ and is a $\cC_p$ dominating point for $\xxx$.  Hence, by what is already proved,  $\ZZ(\hA,0)\preceq 0$ and therefore $\ZZ(A,0)\preceq 0$.   Hence,
\[
 \tsU \subset \mathscr U =\{A: \mbox{ there is an } X \mbox{ such that } \ZZ_{00}(A,X)\succeq 0\}.
\]
Since, by assumption $\tsU(N),$ and therefore  $\mathscr U(N),$ has nonempty interior for a sufficiently large $N$, an application of Lemma \ref{lem:aug6a15} completes the proof.


\begin{remark}\rm
  We make a final remark about the overall structure of the proof of Theorem \ref{thm:main}.  Proposition \ref{lem:generic} gives the existence of sufficiently many points $(A,X)$, independent of the size $n$, in the boundary of $\fP_p$ for which $p(A,X)$ has a one-dimensional kernel. An application of  Proposition \ref{prop:maintool} produces a subspace of codimension one (again regardless of the size of $(A,X)$) in the tangent space on which the {\it second fundamental form} of equation \eqref{eq:poshess} is negative.  The CHSY lemma, Theorem \ref{thm:chsymatrix} is then used to show that the middle matrix $Z(A,X)$ is positive semidefinite on a subspace of this tangent space with a small codimension independent of $n$ and thus $Z(A,X)$ is positive semidefinite sufficiently often imply it has a simple structure.  Likely the one dimensional kernel and subspace of codimension one in Propositions \ref{lem:generic} and \ref{prop:maintool} could be relaxed to requiring only some upper bound on these dimensions independent of $n$.  The hypothesis on $\cO$ in Theorem \ref{thm:main}, namely that it is a free open semialgebraic set can be relaxed considderably.
\end{remark}


\section{Appendix: Faithfulness of finite dimensional representations}
 \label{appendix:r=0}

 This section provides a proof of Remark \ref{prop:r=0}. Namely, any free polynomial vanishing on a large enough set for sufficiently large matrices is zero.

\begin{lemma}
\label{lem:pre r=0}
 Suppose $r(a)$ is a free  (not necessarily symmetric) polynomial  in $\tilde{g}$ variables of degree $\tilde{d}$ and let $N\ge \sum_{j=0}^{\td} \tg^j$ (see $k_b$  defined in equation (\ref{def:kb}). If $r(A)=0$ for all $A\in \mathbb S_N(\mathbb R^{\tilde{g}})$, then $r=0$.
\end{lemma}

\begin{proof}
It suffices to prove the result for scalar-valued polynomials homogeneous of degree $\td$ and for $N=\sum_{j=0}^{\td} \tg^j$.  Let $\tilde{d}$ denote the degree of $r$.  Let $\mathcal H$ denote the Hilbert space with orthonormal basis words $m(a)$ of degree at most $\tilde{d}$.  Thus the dimension of $\mathcal H$ is $N$.  Define the tuple $S=(S_1,\dots,S_\tg)$ on $\mathcal H$ by $S_j w =a_j w$ for words $w$ of degree (length) at most $\tilde{d}-1$ and $S_jw=0$ for words $w$ of degree $\tilde{d}$.  It is routine to verify that $S$ extends, by linearity, to an operator on $\mathcal H$.  It is also routine to verify, for words $w$ of degree at most $\tilde{d}-1$,
\[
 S_j^* a_k w = \begin{cases} 0 & \mbox{ if } j\ne k \\  w & \mbox{ if } j=k \end{cases}
\]
 and $S^* \emptyset =0$.

 Of course the $S_j$ are not symmetric. Let $T_j = S_j+S_j^*$. Thus $T=(T_1,\dots,T_\tg)\in \mathbb S_N(\mathbb R^\tg)$. By hypothesis $r(T)=0$. Let $v$ denote a word of degree $\tilde{d}$ and observe
\[
 v(T) \emptyset = v + h_v,
\]
 where $h_v$ is a linear combination of words of degree at most $\tilde{d}-2$. In particular, the vectors $v$ and $h_v$ are orthogonal. Writing  $r=\sum r_v v$, it follows that
\[
 0 =  r(T)\emptyset  = \sum_{v} r_v v + \sum r_v h_v.
\]
 Using orthogonality of words once again, we conclude $r_v =0$ for each $v$ and therefore $r=0$.
\end{proof}

\begin{proof}[Proof of Remark \ref{prop:r=0}]
 Let $U$ be an open set on which $r$ vanishes and $Y$ a given point in $U$. Let $(A,X)$ be a given tuple in $\mathbb S_N(\mathbb R^{\tilde{g}})$. The matrix polynomial of a single real variable, $r_{Y,Z}(t) = p(Y+tZ)$ has entries which are ordinary polynomials which vanish on an open interval containing $0$ and hence vanish everywhere. It follows that $r(X)=0$ for all $X\in\mathbb S_N(\mathbb R^{\tilde{g}})$ and therefore $r(X)=0$ for all $X\in\mathbb S_m(\mathbb R^{\tilde{g}})$ for all $m\le N$.  Lemma \ref{lem:pre r=0} now implies $r=0$.
\end{proof}

\section{Appendix: Illustrating the basic calculations}
{\bf Note to the referee: We defer to you opinion on whether to keep or not this, and the following, appendices.}

\label{app:nov10a14}
Fix $c,d\in\RR$ and let
$$
p_3(a,x)=ca_1x_1a_2x_2^2+dx_2^2a_2x_1a_1.
$$
In this section we compute explicitly many of the objects appearing in this article such as the border vector and middle matrix, the various Hessians and spaces appearing in the CHSY lemma and the chip set representations.
Since $p_3$ is a homogeneous polynomial of degree  $d=3$ in $x$ and is of degree
$\wt{d}=2$ in $a$ and there are no consecutive strings of $a$, it suffices to choose
\begin{equation}
\label{eq:nov11a14}
b=\begin{bmatrix}1\\ a_1\\ a_2\end{bmatrix} \quad\textrm{and}\quad x=\begin{bmatrix}x_1\\ x_2\end{bmatrix}.
\end{equation}
Then $k_b$, the number of entries in $b$, is equal to $3$ and  $p_3(a,x)$ is a linear combination of two of the entries in the vector polynomial $b\otimes (x\otimes b)_3$ of height $\ft=k_b(k_bg)^d=3(6^3)=648$,
i.e.,
$$
p_3(a,x)=\begin{bmatrix}c_1 &\cdots&c_t\end{bmatrix}b\otimes (x\otimes b)_3
$$
where all but two of the coefficients $c_1,\ldots,c_t\in\RR$ are equal to zero.

By successive applications of formula
\eqref{eq:aug9b13}, Kronecker products can be
reexpressed as ordinary matrix products:
\begin{align*}
b\otimes (x\otimes b)_3&=(b\otimes I_{\ft_2})(x\otimes b)_3\\
&=(b\otimes I_{\ft_2})(x\otimes b\otimes I_{\ft_1})(x\otimes b\otimes I_{\ft_0})(x\otimes b)\\
&=(b\otimes I_{\ft_2})K_1 K_0 W_0,
\end{align*}
where
$$
W_0=x\otimes b,\quad \ft_j=(k_b g)^{j+1} \quad\textrm{for $j=0,\ldots, 2$ (since $d=3$)}
$$
and
$$
K_j=K_j(a,x)=(x\otimes b)\otimes I_{\ft_j} \quad\textrm{is a matrix polynomial of size $\ft_{j+1}\times \ft_j$ }.
$$
Thus, upon setting
$$
\varphi_2=\begin{bmatrix}c_1 &\cdots&c_t\end{bmatrix}(b\otimes I_{\ft_{2}}),
$$
it is readily seen that
$$
p_3(a,x)=\varphi_2 K_1K_0W_0 \quad\textrm{with $W_0=x\otimes b$}
$$
and, analogously, since $\varphi_2$ is independent of $x$,
\begin{align*}
(p_3(a,x))_x[h]&=\varphi_2
\{(h\otimes b)\otimes(x\otimes b)_2+(x\otimes b)\otimes(h\otimes b)\otimes(x\otimes b)\\  &\qquad+
(x\otimes b)_2\otimes(h\otimes b)\}\\
&=\varphi_2\{ V_2+(x\otimes b)\otimes V_1+(x\otimes b)_2\otimes V_0\}\\
&=\varphi_2\{ V_2+K_1V_1+K_1K_0V_0\}.
\end{align*}
in which
$$
V_0=h\otimes b\quad\textrm{and}\quad V_j=(h\otimes b)\otimes (x\otimes b)_j\quad\textrm{for $j=1,2$}.
$$
Thus,
$$
p_3(a,x)=\begin{bmatrix}0&0&\varphi_2\end{bmatrix}\begin{bmatrix}I_{\ft_0}&0&0\\ K_0&I_{\ft_1}&0\\
K_1 K_0&K_1&I_{\ft_2}\end{bmatrix}\begin{bmatrix}W_0\\ 0\\ 0\end{bmatrix}
$$
and
$$
(p_3(a,x))_x[h]=\begin{bmatrix}0&0&\varphi_2\end{bmatrix}\begin{bmatrix}I_{\ft_0}&0&0\\ K_0&I_{\ft_1}&0\\
K_1 K_0&K_1&I_{\ft_2}\end{bmatrix}\begin{bmatrix}V_0\\ V_1\\ V_2\end{bmatrix}.
$$
Similarly, the homogeneous polynomials (in $x$)
$$
p_2(a,x)=e a_2x_2a_1x_1+f x_1a_1x_2a_2\quad\textrm{and}\quad p_1(a,x)=ga_1 x_2a_2+ma_2x_2a_1
$$
can be written as
\begin{align*}
p_2(a,x)&=\begin{bmatrix}d_1 &\cdots&d_s\end{bmatrix}b\otimes (x\otimes b)_2\\
&=\begin{bmatrix}d_1 &\cdots&d_s\end{bmatrix} (b\otimes I_{\ft_1})(x\otimes b)_2\\
&=\varphi_1 K_0W_0
\end{align*}
and
\begin{align*}
p_1(a,x)&=\begin{bmatrix}e_1 &\cdots&e_r\end{bmatrix}b\otimes (x\otimes b)\\
&=\begin{bmatrix}e_1 &\cdots&e_r\end{bmatrix}(b\otimes I_{\ft_0}) (x\otimes b)\\
&=\varphi_0 W_0,
\end{align*}
with the same choice of $b$ and $x$ as in \eqref{eq:nov11a14}, $d_1,\ldots,d_s;e_1,\ldots,e_r\in\RR$, $s=k_b(k_bg)^2=3(6^2)=108$,
$r=k_b(k_bg)=3(6)=18$,
$$
\varphi_1=\begin{bmatrix}d_1 &\cdots&d_s\end{bmatrix} \,(b\otimes I_{\ft_1})\quad\textrm{and}\quad \varphi_0=\begin{bmatrix}e_1 &\cdots&e_r\end{bmatrix}(b\otimes I_{\ft_0}).
$$
Since $\varphi_1$ and $\varphi_0$ are independent of $x$, it is readily checked that
\begin{align*}
(p_2)_x(a,x)[h]&=\varphi_1 \{(h\otimes b)\otimes(x\otimes b)+(x\otimes b)\otimes(h\otimes b)\}\\
&=\varphi_1\{V_1+K_0V_0\}
\end{align*}
and
\begin{align*}
(p_1)_x(a,x)[h]&=\varphi_0 \{(h\otimes b)\}\\
&=\varphi_0\{V_0\}.
\end{align*}
Consequently, the polynomial
$$
p(a,x)=p_1(a,x)+p_2(a,x)+p_3(a,x)
$$
admits the representation
\begin{equation*}
p(a,x)=
\begin{bmatrix}\varphi_0&\varphi_1&\varphi_2\end{bmatrix}\begin{bmatrix}I_{\ft_0}&0&0\\ K_0&I_{\ft_1}&0\\
K_1 K_0&K_1&I_{\ft_2}\end{bmatrix}\begin{bmatrix}W_0\\ 0\\ 0\end{bmatrix}
\end{equation*}
and
\begin{equation*}
p_x(a,x)[h]=\begin{bmatrix}\varphi_0&\varphi_1&\varphi_2\end{bmatrix}\begin{bmatrix}I_{\ft_0}&0&0\\ K_0&I_{\ft_1}&0\\
K_1 K_0&K_1&I_{\ft_2}\end{bmatrix}\begin{bmatrix}V_0\\ V_1\\ V_2\end{bmatrix}.
\end{equation*}
\bigskip

\subsection{Range inclusions}

If $p_x(a,x)[h]$ is expressed in the form
$$
p_x(a,x)[h]=U^T\begin{bmatrix}V_0\\ V_1\end{bmatrix}+W^TV_2,
$$
then
\begin{align*}
U^T&=\begin{bmatrix}\varphi_0&\varphi_1&\varphi_2\end{bmatrix}\,
\begin{bmatrix}I_{\ft_0}&0\\ K_0&I_{\ft_1}\\
K_1 K_0&K_1\end{bmatrix}\\
&=\begin{bmatrix}\varphi_0&\varphi_1\end{bmatrix}\,\begin{bmatrix}I_{\ft_0}&0\\ K_0&I_{\ft_1}
\end{bmatrix}+\varphi_2\begin{bmatrix}K_1 K_0&K_1\end{bmatrix}
\end{align*}
and
$$
W^T=\varphi_2.
$$
Thus,
$$
\textup{range}\,U^T\subseteq\textup{range}\,W^T
$$
if and only if
$$
\textup{range}\,\{\begin{bmatrix}\varphi_0&\varphi_1\end{bmatrix}\,\begin{bmatrix}I_{\ft_0}&0\\ K_0&I_{\ft_1}
\end{bmatrix}+\varphi_2\begin{bmatrix}K_1 K_0&K_1\end{bmatrix}\}\subseteq \textup{range}\,\varphi_2.
$$
Thus, as the matrix $\begin{bmatrix}I_{\ft_0}&0\\ K_0&I_{\ft_1}
\end{bmatrix}$ is invertible,
\begin{equation*}
\textup{range}\,\subseteq\textup{range}\,W^T\Longleftrightarrow
\textup{range}\,\begin{bmatrix}\varphi_0&\varphi_1\end{bmatrix}\subseteq \textup{range}\,\varphi_2.
\end{equation*}

\subsection{Chip set representation}
The chip sets for $p_1,p_2,p_3$ are:
\begin{align*}
\cC_{p_1}&=\{a_2,a_1\},\quad\cC_{p_2}=\{1,a_2,a_1x_1,a_1x_1a_2\}\quad\textrm{and}\\
\cC_{p_3}&=\{1,a_1,x_2,a_2x_1a_1,a_2x_2^2,x_2a_2x_1a_1\},
\end{align*}
respectively. Upon setting
\begin{align*}
w_1&=1, w_2=a_1, w_3=a_2,w_4=a_1x_1,w_5=a_1x_1a_2,w_6=x_2,\\
&w_7=a_2x_1a_1,w_8=x_2a_2x_1a_1\ \textrm{and}\ w_9=a_2x_2^2,
\end{align*}
it is readily checked that
\begin{align*}
p_1(a,x)&=(ga_1)x_2w_3+(ma_2)x_2w_2,\\
p_2(a,x)&=(ea_2)x_2w_4+(f)x_1w_5,\\
p_3(a,x)&=(ca_1)x_1w_9+(d)x_2w_8,
\end{align*}
\begin{align*}
(p_1)_x(a,x)[h]&=ma_2(h_2w_2)+ga_1(h_2w_3),
\\(p_2)_x(a,x)[h]&=fx_1a_1(h_2w_3)+ea_2x_2a_1(h_1w_1)+ea_2(h_2w_4)+f(h_1w_5)\
\textrm{and}\\
(p_3)_x(a,x)[h]&=dx_2^2a_2(h_1w_2)+ca_1x_1a_2x_2(h_2w_1)+dx_2(h_2w_7)\\ &+ca_1x_1a_2(h_2w_6)+
ca_1(h_1w_9)+d(h_2w_8).
\end{align*}

Since $w_8$ and $w_9$ are of degree $d-1=2$ in $x$,
the majorization condition in \S \ref{subsec:hdt} is the ranges of
$$
M\otimes A_2, G\otimes A_1, E\otimes A_2, F\otimes I_n
$$
are included in the space
$$
\textup{span}\{\textup{range}\,C\otimes A_1,\,\textup{range}\,D\otimes I_n\}.
$$
Moreover, $p_x(a,x)[h]$ can be  expressed as
$$
p_x(a,x)[h]=\begin{bmatrix}U^T&W^T\end{bmatrix} \begin{bmatrix}Q_1(a,x,h)\\ Q_2(a,x,h)\end{bmatrix},
$$
in which
$$
U^T=\begin{bmatrix}dx_2^2a_2&ma_2&fx_1a_1&ga_1&ea_2x_1a_1&ca_1x_1a_2&ea_2&dx_2&ca_1x_1a_2&f1\end{bmatrix},
$$
$$
W^T=\begin{bmatrix}ca_1&d1\end{bmatrix},
$$
and the entries in
$$
Q_1(a,x,h)=\textup{column}(h_1w_2,h_2w_2,h_1w_3,h_2w_3,h_1w_1,h_2w_1,h_1w_4,h_2w_7,h_2w_6,h_1w_5)
$$
are  a subcollection of the entries in $V_0$ and $V_1$, whereas the entries in
$$
Q_2(a,x,h)=\textup{column}(h_1w_9,h_2w_8)
$$
are a subcollection of the entries in $V_2$. Thus if the majorization condition of \S \ref{subsec:hdt} is
met, then the range of $U^T$ is a subspace of the range of $W^T$.

\section{Appendix: A range inclusion example}

In Example \ref{ex:sep17a15} it was shown that if $1\le k\le d$, then the highest degree terms of the polynomial $p(a,x)=a(1+x^k)+(1+x^k)a+x^d$ majorize. Thus, in view of item (3) of Theorem \ref{thm:dec29a13}, the range inclusion condition
Theorem \ref{thm:dec29a13} \eqref{eq:sep11a14} must hold. This is confirmed in this example. To ease the exposition only the case $k=d$ will be discussed.

Since
$$
p(a,x)=(1+a)x^d+x^da+p(a,0),
$$
it is readily checked that
\begin{align*}
p_x(a,x)[h]&=(1+a)(hx^{d-1}+xhx^{d-2}+\cdots+x^{d-1}h)\\ &+(hx^{d-1}+xhx^{d-2}+\cdots+x^{d-1}h)a.
\end{align*}
Thus, if $(A,X)\in\allmatn$, then
\begin{align*}
p_x(A,X)[h]&=
\begin{bmatrix}U^T&W^T\end{bmatrix}\begin{bmatrix}H\\ HA\\ HX\\ HXA\\ \vdots \\ HX^{d-1}\\ HX^{d-1}A\end{bmatrix},
\end{align*}
with
$$
W^T=\begin{bmatrix} (I_n+A)&I_n\end{bmatrix}
$$
and
$$
U^T=
\begin{bmatrix}(I_n+A)X^{d-1}&X^{d-1}&(I_n+A)X^{d-2}&X^{d-2}&\cdots &(I_n+A)X&X\end{bmatrix}.
$$
Consequently,
\begin{align*}
U^T&=\begin{bmatrix} (I_n+A)&I_n\end{bmatrix}\begin{bmatrix}X^{d-1}&0&X^{d-2}&0&\cdots&X&0\\
0&X^{d-1}&0&X^{d-2}&\cdots&X\end{bmatrix}\\
&=W^T\begin{bmatrix}X^{d-1}&0&X^{d-2}&0&\cdots&X&0\\
0&X^{d-1}&0&X^{d-2}&\cdots&0&X\end{bmatrix},
\end{align*}
which clearly displays the fact that $\textup{range}\,U^T\subseteq\textup{range}\,W^T$.

Moreover,
$$
p(A,X)=\begin{bmatrix}U_1^T&W^T\end{bmatrix}\begin{bmatrix}X\\ XA\\ XX\\ XXA\\ \vdots \\ XX^{d-1}\\ XX^{d-1}A\end{bmatrix}+p(A,0),
$$
with $U_1^T=0_{n\times 4n}$. Therefore,
$$
\textup{range}\, U_1^T\subseteq\textup{range}\,W^T,
$$
i.e., the highest degree terms in $p$,  $W^T$, clearly majorize $U_1^T$, the lower degree terms in $p$. \qed

\section{Appendix: Examples of chip set representations}

In this appendix chips sets are calculated for several polynomials.

\begin{example}\rm
\label{ex:sep4a14}
Let
$$
p(a,x)=cx_1ax_2^3+dx_2^3ax_1
$$
with $c,\,d\in\RR$. Then
$$
\cC_p^1=\{1,ax_2^3\}\quad\textrm{and}\quad \cC_p^2=\{1,x_2,x_2^2,ax_1,x_2ax_1,x_2^2ax_1\}
$$
and, by direct computation, it is readily checked that
\begin{align*}
p_x(a,x)[h]&=c\{h_1ax_2^3+h_1ax_2^2+x_1ax_2h_2x_2\}\\ &+d\{h_2x_2^2ax_1+x_2h_2x_2ax_1+x_2^2h_2ax_1+x_2^3ax_1\},
\end{align*}
and hence that $p_x(a,x)[h]$ is a linear combination of the 8 words in $\cC_p$ with polynomial coefficients
as it should be. By contrast, the general representation \eqref{eq:aug8a14} for nc polynomials $p(a,x)$ that are homogeneous in $x$ of degree 4 yields the formula
$$
p_x(a,x)[h]
=\varphi_{3}
\left\{\sum_{j=0}^{3}(x\otimes b)_j\otimes V_{3-j}\right\}
$$
with
$$
b=\begin{bmatrix}1\\ a\end{bmatrix}\quad\textrm{and}\quad x=\begin{bmatrix}x_1\\ x_2\end{bmatrix}.
$$
This exhibits $p_x(a,x)[h]$ as a linear combination of $\ft_0+\cdots+\ft_3=2^2+2^4+2^6+2^8=340$ terms in which at least 332 have zero coefficients.

The general middle matrix representation $$
p_{xx}(a,x)[h]=\sum_{i,j=0}^2 V_i(a,x)[h]^T Z_{ij}(a,x)V_j(a,x)[h]
$$
based on the representation \eqref{eq:aug8a14}
involves a middle matrix $Z(a,x)$ with $(\ft_0+\ft_1+\ft_2)^2
=84^2$ blocks. However, in fact for the matrix polynomial under consideration, direct computation yields the formula
\begin{align*}
&p_{xx}(a,x)[h]=\\
&2\begin{bmatrix}h_1&h_2& x_2h_2&x_1ah_2&x_2^2h_2&x_1ax_2h_2
\end{bmatrix}\\
&\begin{bmatrix}0&dax_2^2&dax_2&0&d  a&0\\
\vspace{6pt}
cx_2^2a&0&0&c x_2&0&c \\
\vspace{6pt}
cx_2A&0&0&c&0&0\\
\vspace{6pt}
0&d x_2&d &0&0&0\\c a&0&0&0&0&0\\
\vspace{6pt}
0&d &0&0&0&0\end{bmatrix}
\begin{bmatrix} h_1\\ \vspace{6pt}
h_2\\\vspace{6pt} h_2x_2\\ \vspace{6pt} h_2ax_1\\  \vspace{6pt} h_2x_2^2\\ \vspace{6pt} h_2x_2ax_1
\end{bmatrix}
\end{align*}

Thus, only two of the four entries in
$V_0=h\otimes b$,
two of the sixteen entries in $V_1=(h\otimes b)\otimes(x\otimes b)$ intervene  and two of the sixty four entries in
$V_2=(h\otimes b)\otimes(x\otimes b)_2$
come into play in this representation.
\end{example}

\begin{remark}\rm
\label{rem:aug20a14}
Since this polynomial is homogeneous of degree $4$ in $x$, it can also be expressed as
$$
p(a,x)=\begin{bmatrix}c_1&\cdots&c_t\end{bmatrix}b_0\otimes x\otimes b_1\otimes x\otimes b_2\otimes x\otimes b_3\otimes x\otimes b_4
$$
with
$$
x=\begin{bmatrix}x_1\\ x_2\end{bmatrix},\quad b_1=b_3=\begin{bmatrix}1\\ a\end{bmatrix}\quad \textrm{and}\ b_0=b_2=b_4=1.
$$
This is more efficient than the general representation \eqref{eq:aug8a14}, but less efficient than the representation based on the chip set $\cR\cC_p$.
\end{remark}

\begin{example}\rm
\label{ex:sep17a15-alt}
Let $\tg=1$, $a=a_1$, $g=2$, $C,\,D\in\RR^{\dddd}$ and
\begin{equation*}
p(a,x)=C\otimes x_1ax_2^3+D\otimes x_2^3ax_1.
\end{equation*}
Then
\begin{align*}
&p_{xx}(a,x)[h]=\\
&2C\otimes\ \{h_1a(h_2x_2^2+x_2h_2x_2+x_2^2h_2)
+x_1a(h_2^2x_2+h_2x_2h_2+x_2h_2^2)\}\\
&2D\otimes\{(h_2^2x_2+h_2x_2h_2+x_2h_2^2)ax_1
+(h_2x_2^2+x_2h_2x_2+x_2^2h_2)ah_1\}.
\end{align*}
Correspondingly, if $A\in\SS_n$ and $X\in\SS_n(\RR^2)$, then
\begin{align*}
&p_{xx}(A,X)[h]=\\
&2C\otimes\ \{H_1A(H_2X_2^2+X_2H_2X_2+X_2^2H_2)
+X_1A(H_2^2X_2+H_2X_2H_2+X_2H_2^2)\}\\
&2D\otimes\{(H_2^2X_2+H_2X_2H_2+X_2H_2^2)AX_1
+(H_2X_2^2+X_2H_2X_2+X_2^2H_2)AH_1\}.
\end{align*}
Thus, as
$$
C\otimes (M_iH_iM_0H_jM_j)=(I_\kappa\otimes M_iH_i)(C\otimes M_0)(I_\kappa\otimes H_jM_j)
$$
for monomials $M_i$, $M_0$ and $M_j$ in $A$ and $X$,
\begin{align*}
&p_{xx}(A,X)[h]=\\
&2\begin{bmatrix}I_\kappa\otimes H_1&I_\kappa\otimes H_2&I_\kappa\otimes X_2H_2&I_\kappa\otimes X_1AH_2&I_\kappa\otimes X_2^2H_2&I_\kappa\otimes X_1AX_2H_2
\end{bmatrix}\\
&\begin{bmatrix}0&D\otimes AX_2^2&D\otimes AX_2&0&D\otimes  A&0\\
\vspace{6pt}
C\otimes X_2^2A&0&0&C\otimes X_2&0&C\otimes I_n\\
\vspace{6pt}
C\otimes X_2A&0&0&C\otimes I_n&0&0\\
\vspace{6pt}
0&D\otimes X_2&D\otimes I_n&0&0&0\\C\otimes A&0&0&0&0&0\\
\vspace{6pt}
0&D\otimes I_n&0&0&0&0\end{bmatrix}
\begin{bmatrix}I_\kappa\otimes H_1\\ \vspace{6pt} I_\kappa\otimes
H_2\\\vspace{6pt} I_\kappa\otimes H_2X_2\\ \vspace{6pt} I_\kappa\otimes H_2AX_1\\  \vspace{6pt}I_\kappa\otimes H_2X_2^2\\ \vspace{6pt}I_\kappa\otimes H_2X_2AX_1
\end{bmatrix}
\end{align*}

Let $M$ denote the $6\times 6$ block matrix with blocks of size $\kappa n\times\kappa n$ in the last formula for $p_{xx}(A,X)[h]$. If $A\in\SS_n(\RR^{\tg})$, $X\in\SS_n(\RR^g)$ and $D=C^T$, then $M=M^T$. Moreover, if $M$ is partitioned as a $3\times 3$ block matrix with
blocks of size $2\kappa n\times 2\kappa n$, then
$$
\begin{bmatrix}M_{00}&M_{01}&M_{02}\\M_{10}&M_{11}&0\\M_{20}&0&0\end{bmatrix}
\begin{bmatrix}
I_{2\kappa n}&0&0\\-\wt{K}_0&I_{2\kappa n}&0\\0&-\wt{K}_1&I_{2\kappa n}
\end{bmatrix}=\begin{bmatrix}0&0&M_{02}\\0&M_{11}&0\\M_{20}&0&0\end{bmatrix},
$$
where
$$
\wt{K}_0=\begin{bmatrix}0&I_\kappa\otimes X_2\\I_\kappa\otimes X_2A&0\end{bmatrix}\quad\textrm{and}\quad
\wt{K}_1=\begin{bmatrix}I_\kappa\otimes X_2&0\\0&I_\kappa\otimes X_2\end{bmatrix}.
$$

Correspondingly, $p_x$ can be expressed as
$$
p_x(a,x)[h]=\begin{bmatrix}0&0&0&Q\end{bmatrix}\begin{bmatrix}I_{2\kappa n}&0&0&0\\ \vspace{6pt} \wt{K}_0&I_{2\kappa n}&0&0\\ \wt{K}_1\wt{K}_0&\wt{K}_1&I_{2\kappa n}&0\\ \vspace{6pt}
\wt{K}_2\wt{K}_1\wt{K}_0&\wt{K}_2\wt{K}_1&\wt{K}_2&I_{2\kappa n}\end{bmatrix}
\begin{bmatrix}G_0\\ \vspace{6pt} G_1\\ \vspace{6pt} G_2\\ \vspace{6pt} G_3\end{bmatrix}
$$
with $Q=\begin{bmatrix}C\otimes I_n & D\otimes I_n\end{bmatrix}$,
\begin{align*}
G_0&=\begin{bmatrix}I_\kappa\otimes H_1\\ \vspace{6pt} I_\kappa\otimes
H_2\end{bmatrix},\ G_1=
\begin{bmatrix}I_\kappa\otimes H_2X_2\\ \vspace{6pt} I_\kappa\otimes H_2AX_1\end{bmatrix},\
G_2=\begin{bmatrix}I_\kappa\otimes H_2X_2^2\\ \vspace{6pt}I_\kappa\otimes H_2X_2AX_1\end{bmatrix}\\ \vspace{6pt}
 G_3&=\begin{bmatrix}I_\kappa\otimes H_2X_2^2\\ \vspace{6pt}I_\kappa\otimes H_2X_2AX_1\end{bmatrix}\quad\textrm{and}\quad \wt{K}_2=\begin{bmatrix}I_\kappa\otimes X_1A&0\\ \vspace{6pt} 0&I_\kappa\otimes X_2\end{bmatrix}.
\end{align*}
Thus, the entries $U$ and $W$ in the middle matrix (\ref{eq:dec29a13}) in the representation of $$p_{xx}(A,X)[h]+\lambda p_x(A,X)[h]^Tp_x(A,X)[h]$$ are equal to
$$
U^T=Q\begin{bmatrix}\wt{K}_2\wt{K}_1\wt{K}_0&\wt{K}_2\wt{K}_1&\wt{K}_2\end{bmatrix}
\quad\textrm{and}\quad W^T=Q.
$$
Therefore, the range inclusion condition  Theorem \ref{thm:dec29a13} \eqref{eq:sep11a14}) is met  and hence Theorem \ref{thm:dec29a13}
guarantees that the  matrix (\ref{eq:dec29a13}) is  congruent to the matrix on the right hand side of \eqref{eq:sep4a14}.  Compare with Remark \ref{rem:special majorize}.

In Example \ref{ex:sep17a15-alt}
$$
W^TW=(C^TC+D^TD)\otimes I_n,
$$
is invertible  if and only if $C^TC+D^TD$ is invertible.  If $D=C^T$, as is the case for symmetric $p$, then $W^TW$ is invertible if and only if  $C^TC+CC^T$ is invertible. This will fail for every symmetric real matrix with zero eigenvalues.
\end{example}

\begin{remark}\rm
\label{rem:aug28b14}
The condition that the range of $U^T$ is contained in the range of $W^T$
is   necessary for the nice factorization in (\ref{eq:oct9a13}) to hold.
In this example $W$ is independent of $A$ because the words in the polynomial under
consideration are all of the form $x_i\cdots x_j$
(i.e., they begin and end with one of  the $x$ variables).
\end{remark}

\begin{remark}
\label{rem:oct9a13}
The $6\times 6$ block matrix with blocks of size $\kappa n\times\kappa n$ in the last formula for
$p_{xx}(A,X)[h]$ can also be rewritten in terms of the permutation matrix $\Pi$  that is defined by the formula
$$
\Pi\,\left(I_\kappa\otimes
\begin{bmatrix}H_1\\ \vspace{6pt}
H_2\\\vspace{6pt}  H_2X_2\\ \vspace{6pt}  H_2AX_1\\  \vspace{6pt}  H_2X_2^2\\ \vspace{6pt} H_2X_2AX_1
\end{bmatrix}\right)
=\begin{bmatrix}I_\kappa\otimes H_1\\ \vspace{6pt} I_\kappa\otimes
H_2\\\vspace{6pt} I_\kappa\otimes H_2X_2\\ \vspace{6pt} I_\kappa\otimes H_2AX_1\\  \vspace{6pt}I_\kappa\otimes H_2X_2^2\\ \vspace{6pt}I_\kappa\otimes H_2X_2AX_1
\end{bmatrix},
$$
then, since $\Pi^T\Pi=I$,
\begin{align*}
p_{xx}(A,X)[h]&=2(I_\kappa \otimes
\begin{bmatrix}H_1&H_2&X_2H_2&X_1AH_2&X_2^2H_2&X_1AX_2H_2
\end{bmatrix})\\
&\times (\Pi^TM\Pi)
\left(I_\kappa\otimes
\begin{bmatrix}H_1\\ \vspace{6pt}
H_2\\\vspace{6pt}  H_2X_2\\ \vspace{6pt}  H_2AX_1\\  \vspace{6pt}  H_2X_2^2\\ \vspace{6pt} H_2X_2AX_1
\end{bmatrix}\right)
\end{align*}
\end{remark}


\begin{example}\rm
\label{ex:sep4b14as}
Suppose that $c,d\in\RR$ and
$$
p(a,x)=ca_1x_1a_2x_2^2+dx_2^2a_2x_1a_1.
$$
Then
$$
\cR\cC_p^1=\{a_1,a_2x_2^2\}\quad\textrm{and}\quad \cR\cC_p^2=\{1,x_2,a_2x_1a_1,x_2a_2x_1a_1\}
$$
and, by direct computation,
\begin{equation*}
\begin{split}
p_x(a,x)[h]&=c\{a_1h_1a_2x_2^2+a_1x_1a_2h_2x_2+a_1x_1a_2x_2h_2\}\\
&+d\{h_2x_2a_2x_1a_1+x_2h_2a_2x_1a_1+x_2^2a_2h_1a_1\}
\end{split}
\end{equation*}
is a linear combination of the words in $\cR\cC_p$ with polynomial coefficients, as it should be and can also be expressed as
\begin{equation}
\label{eq:jul30a15}
\begin{split}
p_x(a,x)[h]&=\begin{bmatrix}1&a_1&x_2&a_1x_1a_2&x_2^2a_2&a_1x_1a_2x_2\end{bmatrix}\\
&\begin{bmatrix}0&0&0&0&0&d\\ 0&0&0&0&c&0\\ 0&0&0&d&0&0\\  0&0&c&0&0&0\\
0&d&0&0&0&0\\ c&0&0&0&0&0\end{bmatrix}\begin{bmatrix} h_2\\ h_1a_1\\ h_2x_2\\ h_2a_2x_1a_1\\
h_1a_2x_2^2\\ h_2x_2a_2x_1a_1\end{bmatrix}.
\end{split}
\end{equation}
Moreover,
\begin{align*}&p_{xx}(a,x)[h]=
2\begin{bmatrix} h_2& a_1h_1&x_2h_2&a_1x_1a_2h_2\end{bmatrix}\\
&\begin{bmatrix}0&dx_2a_2&0&d\\ ca_2x_2&0&ca_2&0\\0&d a_2&0&0\\
c&0&0&0\end{bmatrix}
\begin{bmatrix}h_2\\ h_1a_1\\  h_2x_2\\ h_2a_2x_1a_1\end{bmatrix}
\end{align*}
The reduced middle matrix will be Hermitian if $c=d$.

Correspondingly, if
$$
p(A,X)=C\otimes (A_1X_1A_2X_2^2)+D\otimes (X_2^2A_2X_1A_1)
$$
with $C,\,D\in\RR^{\kappa\times\kappa}$ and  $A,\,X\in\SS_n(\RR^2)$.
Then, by straightforward computation,
\begin{align*}&p_{xx}(A,X)[h]=
2\begin{bmatrix}I_\kappa\otimes H_2&I_\kappa\otimes A_1H_1&I_\kappa\otimes X_2H_2&I_\kappa\otimes A_1X_1A_2H_2\end{bmatrix}\\
&\begin{bmatrix}0&D\otimes X_2A_2&0&D\otimes I_n\\C\otimes A_2X_2&0&C\otimes A_2&0\\0&D\otimes A_2&0&0\\
C\otimes I_n
&0&0&0\end{bmatrix}
\begin{bmatrix}I_\kappa\otimes H_2\\ I_\kappa\otimes H_1A_1\\ I_\kappa\otimes H_2X_2\\ I_\kappa\otimes H_2A_2X_1A_1\end{bmatrix}
\end{align*}

Corresponding to \eqref{eq:jul30a15}
\begin{equation*}
\begin{split}
p_x(A,X)[h]&=\begin{bmatrix}I_n&A_1&&X_2&A_1X_1A_2&X_2^2A_2&A_1X_1A_2X_2\end{bmatrix}\\
&\begin{bmatrix}0&0&0&0&0&dI_n\\ 0&0&0&0&cI_n&0\\ 0&0&0&dI_n&0&0\\  0&0&cI_n&0&0&0\\
0&dI_n&0&0&0&0\\ cI_n&0&0&0&0&0\end{bmatrix}\begin{bmatrix} H_2\\ H_1A_1\\ H_2X_2\\ H_2A_2X_1A_1\\
H_1A_2X_2^2\\ H_2X_2A_2X_1A_1\end{bmatrix}.
\end{split}
\end{equation*}
In view of \eqref{eq:aug6a13}, the derivative with respect to $x$ of the  matrix valued polynomial
$$
p(a,x)=C\otimes a_1x_1a_2x_2^2+D\otimes x_2^2a_2x_1a_1
$$
can be expressed as
\begin{align*}
&p_x(A,X)[h]\\
&=\begin{bmatrix}I_\kappa\otimes I_n&I_\kappa\otimes A_1&I_\kappa\otimes X_2&I_\kappa\otimes A_1X_1A_2&I_\kappa\otimes X_2^2A_2&I_\kappa\otimes A_1X_1A_2X_2\end{bmatrix}\\
&\begin{bmatrix}0&0&0&0&0&D\otimes I_n\\ 0&0&0&0&C\otimes I_n&0\\ 0&0&0&D\otimes I_n&0&0\\  0&0&C\otimes I_n&0&0&0\\
0&D\otimes I_n&0&0&0&0\\ C\otimes I_n&0&0&0&0&0\end{bmatrix}
\begin{bmatrix}I_\kappa\otimes H_2\\I_\kappa\otimes  H_1A_1\\  I_\kappa\otimes H_2X_2\\ I_\kappa\otimes H_2A_2X_1A_1\\
I_\kappa\otimes H_1A_2X_2^2\\ I_\kappa\otimes H_2X_2A_2X_1A_1\end{bmatrix}\\
&=U^T\begin{bmatrix}I_\kappa\otimes H_2\\I_\kappa\otimes  H_1A_1\\  I_\kappa\otimes H_2X_2\\ I_\kappa\otimes H_2A_2X_1A_1\end{bmatrix}+W^T\begin{bmatrix}I_\kappa\otimes H_1A_2X_2^2\\ I_\kappa\otimes H_2X_2A_2X_1A_1\end{bmatrix}
\end{align*}
with
$$
W^T=\begin{bmatrix}C\otimes A_1&D\otimes I_n\end{bmatrix}
$$
and
\begin{align*}
U^T&=\begin{bmatrix}C\otimes A_1X_1A_2X_2&D\otimes X_2^2A_2&C\otimes A_1X_1A_2&D\otimes X_2\end{bmatrix}\\
&=W^T\begin{bmatrix}I_\kappa\otimes X_1A_2X_2&0&I_\kappa\otimes X_1A_2&0\\
0&I_\kappa\otimes X_2^2A_2&0&I_\kappa\otimes X_2\end{bmatrix}.
\end{align*}
Thus, condition Theorem \ref{thm:dec29a13} \eqref{eq:sep11a14} is met, whereas
$$
W^TW=\begin{bmatrix}C\otimes A_1&D\otimes I_n\end{bmatrix}\begin{bmatrix}C^T\otimes A_1\\ D^T\otimes I_n\end{bmatrix}=CC^T\otimes A_1^2+DD^T\otimes I_n
$$
is not necessarily invertible, even if $D=C^T$.

The polynomial
$p$ will be symmetric if and only if $C=D^T$.

\end{example}


\section{Appendix: Illustrating the CHSY Lemma}
Let
$$
p(a,x)=ca_1x_1a_2x_2^2+dx_2^2a_2x_1a_1.
$$
Let $A_1=A_2=X_1=X_2=I_n$ and $c=d=1$. Thus,
\begin{align*}&p_{xx}(A,X)[h]=
2\begin{bmatrix} H_2& H_1&H_2&H_2\end{bmatrix}\\
&\begin{bmatrix}0&I_n&0&I_n\\ I_n&0&I_n&0\\0&I_n&0&0\\
I_n&0&0&0\end{bmatrix}
\begin{bmatrix}H_2\\ H_1\\  H_2\\ H_2\end{bmatrix}
\end{align*}
and
\begin{equation*}
\begin{split}
p_x(A,X)[h]&=\begin{bmatrix}I_n&I_n&I_n&I_n&I_n&I_n\end{bmatrix}\\
&\begin{bmatrix}0&0&0&0&0&I_n\\ 0&0&0&0&I_n&0\\ 0&0&0&I_n&0&0\\  0&0&I_n&0&0&0\\
0&I_n&0&0&0&0\\ I_n&0&0&0&0&0\end{bmatrix}\begin{bmatrix} H_2\\ H_1\\ H_2\\ H_2\\
H_1\\ H_2\end{bmatrix}\\
&=\begin{bmatrix}I_n&I_n&I_n&I_n&I_n&I_n\end{bmatrix}\begin{bmatrix} H_2\\ H_1\\ H_2\\ H_2\\
H_1\\ H_2\end{bmatrix}.
\end{split}
\end{equation*}
The last formula identifies the blocks $U$ and $W$ that are introduced in Theorem 6.3 as
$$
U^T=\begin{bmatrix}I_n&I_n&I_n&I_n\end{bmatrix}\quad\textrm{and}\quad W^T=\begin{bmatrix}I_n&I_n&\end{bmatrix}.
$$
Thus, as  $U^TU=4I_n$ and $W^TW=2I_n$,
$$
\cR_{U^T}=\cR_{U^TU}=\cR_{W^TW}=\cR_{W^T}.
$$
It is convenient to write
$$
\begin{bmatrix} H_2\\ H_1\\ H_2\\ H_2\\
H_1\\ H_2\end{bmatrix}=Q_n\begin{bmatrix}H_1\\ H_2\end{bmatrix}\ \textrm{with}\ Q_n=\begin{bmatrix}0&I_n\\ I_n&0\\ 0&I_n\\ 0&I_n\\I_n&0\\ 0&I_n\end{bmatrix}
$$
and to choose $v=e_1$, where $\{e_1,\ldots,e_n\}$ is the standard orthonormal basis for $\RR^n$.
Then
$$
p_x(A,X)[h]=\begin{bmatrix}I_n&I_n&I_n&I_n&I_n&I_n\end{bmatrix}Q_n\begin{bmatrix}H_1\\ H_2\end{bmatrix}.
$$
Moreover,
the linear transformation $T$ from $\SS_n(\RR^2)$ into $\RR^{6n}$ that is defined by the formula
\begin{equation}
\label{eq:sep22a14}
TH=\begin{bmatrix} H_2\\ H_1\\ H_2\\ H_2\\
H_1\\ H_2\end{bmatrix}e_1=\begin{bmatrix}0&I_n\\ I_n&0\\ 0&I_n\\ 0&I_n\\I_n&0\\ 0&I_n\end{bmatrix}\begin{bmatrix}H_1\\ H_2\end{bmatrix}e_1
\end{equation}
maps the $2(n^2+n)/2$ dimensional space $\SS_n(\RR^2)$ onto the $2n$ dimensional subspace
$$
\cR_T=\left\{Q_n\begin{bmatrix}u_1\\ u_2\end{bmatrix}:\,u_1\in\RR^n \ \textrm{and}\ u_2\in\RR^n\right\}
$$
of $\RR^{6n}$. Correspondingly, the inner products based on the terms in the relaxed Hessian
$$
p_{xx}(A,X)[h]+\lambda p_{x}(A,X)[h]^Tp_{x}(A,X)[h]+\delta \begin{bmatrix}H_1&H_2\end{bmatrix}Q_n^TQ_n\begin{bmatrix}H_1\\ H_2\end{bmatrix}
$$
can be calculated as follows:
\begin{equation}
\label{eq:aug3a15}
\begin{split}
e_1^Tp_{xx}(A,X)[h]e_1
&=2y^T\left\{Q_n^T\begin{bmatrix}0&I_n&0&I_n&0&0\\ I_n&0&I_n&0&0&0\\
0&I_n&0&0&0&0\\I_n&0&0&0&0&0\\
0&0&0&0&0&0\\0&0&0&0&0&0\end{bmatrix}Q_n\right\}y\\
&=y^T\left\{4\begin{bmatrix}0&I_n\\I_n&I_n\end{bmatrix}\right\}y,
\end{split}
\end{equation}
with
$$
y=\begin{bmatrix}H_1\\ H_2\end{bmatrix}e_1,
$$
$$
e_1^Tp_{x}(A,X)[h]^Tp_{x}(A,X)[h]e_1=y^T\begin{bmatrix}4I_n&8I_n\\ 8I_n&16 I_n\end{bmatrix}y
$$
and
$$
e_1^T\begin{bmatrix}H_1&H_2\end{bmatrix}Q_n^T Q_n\begin{bmatrix}H_1\\ H_2\end{bmatrix} e_1=y^T\begin{bmatrix}2I_n&0\\0&4I_n\end{bmatrix}y.
$$
Combining terms (with $A$, $X$ and $v$ as specified above) yields the formula
\begin{equation*}
\langle p_{\lambda,\delta}^{\prime\prime}(A,X)v,v\rangle=2y^T\begin{bmatrix}(2\lambda+\delta)I_n&(2+4\lambda)I_n\\(2+4\lambda)I_n&(2+8\lambda+2\delta)I_n\end{bmatrix}y.
\end{equation*}

Now, for $i,j=1,\ldots,n$, let
$$
\Sigma_{ij}=\frac{e_ie_j^T+e_je_i^T}{\sqrt{2}}\quad\textrm{if $i\ne j$\quad and \quad}\Sigma_{ii}=e_ie_i^T
$$
and let $\cN=\cN_T$ denote the nullspace of $T$. The set of $(n^2+n)/2$ matrices
$$
\{\Sigma_{ij}:\, 1\le i\le j\le n\}
$$
is an orthonormal basis for $\SS_n$ equipped with the inner product
$$
\langle H,H^\prime\rangle=\textup{trace}\{(H_1^\prime)^TH_1+(H_2^\prime)^TH_2\}.
$$
Moreover,
$$
\cN^\perp=\textup{span}\{(\Sigma_{11},0),\ldots,(\Sigma_{n1},0), (0,\Sigma_{11}),\ldots,(0,\Sigma_{n1})\}
$$
is the orthogonal complement of $\cN$ with respect to this inner product
and the linear transformation $T$ defined by formula \eqref{eq:sep22a14} maps $\cN^\perp$ injectively onto $\cR_T$.
Thus, the dimensions
of the  maximal positive and negative subspaces of $\SS_n(\RR^2)$ with respect to the quadratic form
$\langle p_{\lambda,\delta}^{\prime\prime}(A,X)v,v\rangle$ may be evaluated from the signature of the matrix
$$
\begin{bmatrix}(2\lambda+\delta)I_n&(2+4\lambda)I_n\\(2+4\lambda)I_n&(2+8\lambda+2\delta)I_n\end{bmatrix},
$$
i.e., from the signature  of $(2\lambda +\delta) I_n$ and, assuming that $2\lambda+\delta\ne 0$,  the signature of it's Schur complement
$$
\left\{(2+8\lambda+2\delta)-\frac{(4\lambda+2)^2}{2\lambda+\delta}\right\}I_n=2\left\{\frac{(6\lambda+\delta+2)(\delta-1)}{2\lambda+\delta}\right\}I_n.
$$
If $\lambda\le 0$ and $\delta<0$, then $2\lambda+\delta<0$ and $\delta-1<0$ and hence the  Schur complement will be positive definite if $6\lambda+\delta+2>0$ and negative definite if $6\lambda+\delta+2<0$. Consequently,
\begin{equation}
\label{eq:sep20a14}
e_+^n(A, X,v;p^{\prime\prime}_{\lambda,\delta}, \SS_n(\RR^2)) =0\quad\textrm{if $\delta<0$, $\lambda\le 0$
and $6\lambda<-\delta-2$}.
\end{equation}
By similar calculations,
\begin{equation}
\label{eq:sep20b14}
e_-^n(A, X,v;p^{\prime\prime}_{\lambda,\delta}, \SS_n(\RR^2)) =n\quad\textrm{if $0<\delta<1$ and  $\lambda\ge 0$}.
\end{equation}

Now, as
$$
p_x(A,X)[h]v=(2H_1+4H_2)e_1,
$$
it is readily seen that if $(H_1,H_2)$ is restricted to $\cN^\perp$,  then
$$
p_x(A,X)[h]v=0\Longleftrightarrow H_2=-\frac{1}{2}H_1,
$$
i.e., if and only if $H=(H_1,H_2)$ belongs to the space
$$
\cM=\textup{span}\{(\Sigma_{11},-\frac{1}{2}\Sigma_{11}),\ldots,(\Sigma_{n1},-\frac{1}{2}\Sigma_{n1})\}.
$$
Therefore, as follows easily from \eqref{eq:aug3a15},
$$
H_2=-\frac{1}{2}H_1\Longrightarrow \langle p_{xx}(A,X)[0,H)e_1,e_1\rangle=-3e_1^TH_1^2e_1.
$$
Consequently,
\begin{equation*}
e_+^n(A,X,v;p_{xx}, \cT(A,X,v) )=0,
\end{equation*}
which coincides with \eqref{eq:sep20a14}, and
\begin{equation*}
 e_-^n(A,X,v;p_{xx}, \cT(A,X,v) )=n,
\end{equation*}
which coincides with \eqref{eq:sep20b14}. Moreover,
$$
\cN^\perp=\cM\oplus\cL
$$
with
$$
\cL=\textup{span}\{(\Sigma_{11},2\Sigma_{11}),\ldots,(\Sigma_{n1},2\Sigma_{n1})\}.
$$


\newpage

\section{Not for publication}

\tableofcontents

\printindex

\end{document}